\documentclass[12pt,reqno,a4paper]{amsart}

\usepackage[ngerman,french,english]{babel}
\usepackage{mathrsfs}
\usepackage{amssymb,epsfig,graphics, color}
\usepackage{amsmath}
\usepackage{bbold}
\usepackage{a4wide}
\usepackage{hyperref}
\usepackage{mathtools}
\usepackage{yfonts}

\newtheorem{theorem}{Theorem}[section]
\newtheorem{proposition}[theorem]{Proposition}

\newtheorem{corollary}[theorem]{Corollary}
\newtheorem{lemma}[theorem]{Lemma}
\newtheorem{sub-lemma}[theorem]{Sub-Lemma}

\theoremstyle{remark}

\newtheorem{remark}[theorem]{Remark}

\newcommand{\mdlo}[1]{\left| #1\right|}

\def\tom{\tau^\omega}
\def\vp{\varphi}

\def\C{\mathcal{C}}

\def\L{\mathcal{L}}
\def\F{\mathcal{F}}

\def\P{\mathcal{P}}

\def\Q{\mathcal{Q}}

\def\NN{\mathbb{N}}

\def\RR{\mathbb{R}}
\def\TT{\mathbb{T}}
\def\ZZ{\mathbb{Z}}
\def\DD{\mathbb{D}}

\def\SS{\mathbb{S}}

\def\O{\Omega}

\DeclareMathOperator{\diam}{diam}
\DeclareMathOperator{\dist}{dist}

\DeclareMathOperator{\lip}{Lip}
\DeclareMathOperator{\Leb}{Leb}
\DeclareMathOperator{\leb}{Leb}

\DeclareMathOperator{\diff}{Diff}

\newcommand{\qand}{\quad\text{and}\quad}

\let\eps=\varepsilon

\def\w{\omega}
\def\E{\mathcal{E}}

\def\RR{{\mathbb R}}
\def\NN{{\mathbb N}}

\newcommand{\Dom}{\Delta_\omega}
\newcommand{\Domo}{\Delta_{\omega, 0}}
\newcommand{\bl}{\tilde\lambda}

\begin{document}
\title[Almost sure rates of mixing for partially hyperbolic attractors]{Almost sure rates of mixing\\ for partially hyperbolic attractors}
\author{Jos\'e F. Alves}
\author{Wael Bahsoun}
\author{Marks Ruziboev}
\address{Centro de Matem\'atica da Universidade do Porto, Rua do Campo Alegre 687, 4169-007 Porto, Portugal}
\email{jfalves@fc.up.pt}
\address{Department of Mathematical Sciences, Loughborough University,
Loughborough, Leicestershire, LE11 3TU, UK}
\email{W.Bahsoun@lboro.ac.uk}
\address{Faculty of Mathematics, University of Vienna,
Oskar-Morgnstern Platz, 1
1090, Austria}
\email{marks.ruziboev@univie.ac.at}
\thanks{JFA was partially supported by CMUP (UID/MAT/00144/2013) and the project PTDC/MAT-CAL/3884/2014 funded by Funda\c{c}\~ao para a Ci\^encia e a Tecnologia (FCT) Portugal with national (MEC) and European structural funds through the program FEDER, under the partnership agreement PT2020.  JFA, WB and MR would like to thank The Leverhulme Trust for supporting their research through the Visiting Professorship VP2-2017-004 and the Research Grant RPG-2015-346. MR would like to thank the Austrian Science Fund (FWF): M2816 Meitner Grant.}
\keywords{Partial hyperbolicity, random dynamical systems, decay of correlations.}
\subjclass{Primary 37A05, 37C10, 37E05}

\begin{abstract}
We introduce random towers to study almost sure rates of correlation decay for random partially hyperbolic attractors. Using this framework, we obtain abstract results on almost sure exponential, stretched exponential and polynomial correlation decay rates. We then apply our results to small random perturbations of  Axiom~A   attractors, small perturbations of derived from Anosov partially hyperbolic systems and to solenoidal attractors  with  random intermittency.
\end{abstract}

\date{\today}
\maketitle
\tableofcontents
\section{Introduction}\label{setup}
One of the most fundamental questions in ergodic theory of chaotic dynamical systems is that of the rate of correlation decay with respect to a physical measure.
For deterministic chaotic systems there are two popular techniques to obtain such rates. The first one is based on functional analytic techniques that require strong hyperbolic properties of the underlying system (see the recent breakthrough \cite{BDL18} and references therein). The other method, which is based on Markov extensions \cite{Y98} and coupling techniques  \cite{Y99, D00}, has been applied to obtain correlation decay rates  \cite{ADL17, BLV03, C04, G06} and probabilistic limit theorems \cite{G05, MN08, MN09} for different classes of deterministic chaotic systems that admit a `weak' form of hyperbolicity.

Recently, there has been a significant interest in obtaining probabilistic limit theorems for random dynamical systems \cite{ANV15, ALS09, BB16, DFG19, DFG18, DFG18a, DH20, H20, HK18, SSV21}, related time dependent systems \cite{Blu17, HL19, HS19,L18a} and for stochastic flows \cite{DKK04, BBPS19}. As in the case of deterministic dynamical systems, some knowledge on the correlation decay rate of the underlying system is required to obtain probabilistic limit laws for random dynamical systems. However, despite the pioneering work of \cite{BBM02, Bu99}, which was followed by \cite{LV18}, and the recent progress made in \cite{BBR19} on almost sure correlation decay rates for random endomorphisms, there are no general tools on the almost sure correlation decay rates for random dynamical systems where the constituent maps are partially hyperbolic diffeomorphisms. 

The purpose of this work is to provide, for the first time, a tool that can be used to obtain almost sure correlation decay rates for \emph{random attractors} \cite{A98, CF94, LQ95} that exhibit a `weak' form of hyperbolicity. Indeed, we introduce random hyperbolic towers and show that such towers can be used to study pathwise statistics on random \emph{partially hyperbolic attractors}. We obtain abstract results on the almost sure rate of correlation decay and apply these results to small random perturbations of Axiom~A attractors and derived from Anosov (DA) partially hyperbolic attractors, and to solenoidal attractors with random intermittency. To the best of our knowledge, even in the particular classical case of Axiom~A attractors there are no such statements in the literature and existing work~\cite{DFG19,S11a}, albeit using different methods, only cover the specific case of Anosov diffeomorphisms, and hyperbolic maps with common expanding and contracting directions \cite{B95a, B95b}. Moreover, other than the solenoidal attractors with random intermittency that we present in this paper, our technique can be applied to random versions of deterministic systems considered in \cite{H99,ZH19, PSS20} that admit only polynomial rates of mixing. In this regards, the interesting goal will not only be to obtain polynomial rates, but also to discover whether one obtains similar to what we get in this paper for the solenoid with random intermittency: the fastest mixing fibre map dominates the almost sure mixing rate of the random system. Since the random versions of the families considered in \cite{ZH19,PSS20} require lengthy arguments, we only include the solenoid example, along with the examples of Axiom~A and DA attractors, in this paper and leave the other mentioned applications as a problem in this direction of research.

Let $M$ be a smooth compact Riemannian manifold of finite dimension. Denote by $\diff^{1+}(M)$ the set of $C^{1}$ diffeomorphisms whose derivative is H\"older endowed with the $C^1$ topology. Given  $\mathcal F\subset \diff^{1+}(M)$, consider ${\bf p}$ a Borel probability measure with compact support $\mathbb B\subset \mathcal F$. 
Let  $\Omega = \mathbb{B}^{\mathbb Z}$ and $P= {\bf p}^{\mathbb Z}$. Then the left  shift map  $\sigma:\Omega\to \Omega$ preserves $P$. 
Also, notice that every $\omega\in\Omega$ is a bi-infinite sequence of diffeomorphisms $\omega=(\dots, f_{-1}, f_0, f_1, \dots)$.
We consider the random dynamical system
\[
 f_\omega^0=\text{Id},  \quad  f_\omega^1=f_\omega=f_0, \quad f_\omega^n=f_{\sigma^{n-1}\omega}\circ\cdots\circ f_\omega \quad \text{and} \quad f^{-n}_\omega=(f_\omega^n)^{-1},\text{ for } n\ge 0.
 \]
We call  a family of Borel probability measures $\{\mu_\omega\}_{\omega\in\Omega}$ on $M$  \emph{equivariant} if
 $$(f_\omega)_\ast\mu_\omega=\mu_{\sigma\omega}$$ for $P$ almost all $\omega\in\Omega$. We call
 $\{\mu_\omega\}_{\omega\in\Omega}$ \emph{physical} if for $P$-almost all $\omega\in\Omega$, the set   of points $x\in M$ such that
\[
\frac{1}{n}\sum_{i=0}^{n-1}\delta_{f^n_{\sigma^{-n}\omega}(x)} \xrightarrow[\text{}]{\text{weakly}}\mu_\omega 
\]
has positive Lebesgue measure. This  set  is called the \emph{basin} of $\mu_\omega$ and denoted $B(\mu_\omega)$.

In this paper we study the existence of equivariant families of physical measures  $\{\mu_\omega\}_{\omega\in\Omega}$ and their statistical properties.
In particular,  we consider 
random  correlations
 \begin{equation}\label{dycorr}
\begin{split}
\mathcal C_n(\varphi, \psi,\mu_\omega)=
\int (\varphi\circ f^{n}_\omega)\psi d\mu_\omega-\int \varphi d\mu_{\sigma^n\omega}\int \psi d\mu_{\omega},
\end{split}
\end{equation}
for observables $\varphi, \psi:  M \to \mathbb R$.  
 Let $\mathcal F_1$ and $\mathcal F_2$ be   spaces of observables on $  M$ and let $\{\rho_n\}_{n\in\mathbb N}$ be a sequence of positive numbers such that $\lim_{n\to \infty}\rho_n=0$. We say that the random dynamical system  admits ``quenched" decay of  \emph{correlations} with rate $\rho_n$ if 
for $P$-almost all $\omega\in\Omega$ and any  $\varphi\in \mathcal F_1$, $\psi \in \mathcal F_2$, there are constants $C_\omega$, measurable in $\omega$, and $C_{\varphi, \psi}$ such that  
\begin{equation}\label{decorr}
\left|\mathcal C_n(\varphi, \psi,\mu_\omega)\right|\le C_\omega C_{\varphi, \psi}\rho_n.
\end{equation}


\subsection{Hyperbolic product structures}\label{se.hps}
Let $\leb$ denote Lebesgue measure on Borel subsets of $M$. Given a submanifold $\gamma\subset M$ we use $ \text{Leb}_\gamma$ to denote the volume induced by the restriction of the Riemannian metric on~$\gamma$.  We say that  an embedded disk $\gamma\subset M$ is  a stable manifold at fibre $\omega$ if $\dist(f^n_\omega(x), f^n_\omega(y))\to 0$ as $n\to \infty$ for every $x, y\in\gamma$.  Similarly, an embedded disk $\gamma\subset M$ is called an  unstable manifold at $\omega$ if $\dist(f^{-n}_{\sigma^{-n}\omega}(x), f^{-n}_{\sigma^{-n}\omega}(y))\to0$ as
$n\to\infty$ for every $x,y\in\gamma$. 

Notice that if $\gamma$ is a stable manifold at $\omega$ then $f_\omega(\gamma)$ is a stable manifold at $\sigma\omega$.  
 Let  $\text{Emb}^1(D^u, M)$, where $D^u$ is a unit disk of $\RR^{d}$, be the space of $C^1$ embeddings from
$D^u$ into $M$. We say that $\Gamma^u_\omega=\{\gamma^u_\omega\}$ is a
\emph{continuous family of $C^1$ unstable manifolds at $\omega$} if there is a
compact set $K^s$, and a map
$\Phi^u_\omega\colon K^s\times D^u\to M$ such~that
\begin{itemize}
\item  $\gamma^u_\omega=\Phi^u_\omega(\{x\}\times D^u)$ is an unstable
manifold;
\item  $\Phi^u_\omega$ maps $K^s\times D^u$ homeomorphically onto its image; \item[$\bullet$]  $x\mapsto \Phi^u_\omega\vert(\{x\}\times D^u)$ is a
continuous map from $K^s$ to $\text{Emb}^1(D^u, M)$.
\end{itemize}
Continuous families of $C^1$ stable manifolds are defined similarly.
 We say that $\Lambda_\omega\subset M$ has a \emph{product
structure} if  there exist a continuous family of unstable
manifolds $\Gamma^u_\omega=\{\gamma^u_\omega\}$ and a continuous family of
stable manifolds $\Gamma^s_\omega=\{\gamma^s_\omega\}$ such that
\begin{itemize}
    \item $\Lambda_\omega=(\cup \gamma^u_\omega)\cap(\cup\gamma^s_\omega)$;
    \item $\dim \gamma^u_\omega+\dim \gamma^s_\omega=\dim M$;
    \item each  $\gamma^s_\omega$ meets each $\gamma^u_\omega$  in exactly one point;
    \item stable and unstable manifolds meet transversally with angles bounded
    away from~0 by a constant  independent of $\omega$. 
\end{itemize}
 We call a product structure measurable if $\omega\mapsto\Lambda_\omega$ 
 is measurable. Here measurability is understood in the sense of \cite[Section~3]{CF94}.

Let $\Lambda_\omega\subset M$ be a set with a product structure,  whose   defining families are $\Gamma^s_\omega$ and $\Gamma^u_\omega$. A subset $\Lambda_{0, \omega}\subset \Lambda_\omega$ is called an $s$-subset if $\Lambda_{0, \omega}$ also has a  product structure, and its defining families, $\Gamma_{0, \omega}^s$ and $\Gamma_{0, \omega}^u$, can be chosen with $\Gamma_{0, \omega}^s\subset\Gamma^s_\omega$ and $\Gamma_{0, \omega}^u=\Gamma^u_\omega$;  $u$-subsets are defined similarly. Given $x\in\Lambda_\omega$, let $\gamma^{*}(x)$ denote the element of $\Gamma^{*}_\omega$ containing $x$, for $*=s,u$. For each $n\ge 1$ we let $(f^n_\omega)^u$ denote the restriction of the map $f^n_\omega$ to unstable manifolds $\gamma^{u}_{\omega}$ and $\det D^uf^n_\omega$ denote   the Jacobian of $D^uf^n_\omega$.

The random dynamical system admits a \emph{hyperbolic product structure} on $\{\Lambda_\omega\}_{\omega\in\Omega}$ if for almost every $\omega\in\Omega$ the set $\Lambda_\omega$ has a  product structure and properties  (P$_0$)-(P$_4$) below hold:
\begin{itemize} 
\item[(P$_0$)]  \emph{Detectable}:   $\Leb_{\gamma_\omega}(\Lambda_\omega\cap \gamma_\omega)>0$  for some $\gamma_\omega\in\Gamma^u_\omega$.
\item[(P$_1$)] \emph{Markov}: there is a partition $\mathcal{P}_\omega=\{\Lambda_{i, \omega}\}$ of  $\Lambda_\omega$  into $s$-subsets such  that
   \begin{enumerate}
 	\item $\Leb_{\gamma_\omega}\left((\Lambda_\omega\setminus\cup\Lambda_{i, \omega})\cap\gamma_\omega\right)=0$ on each 	$\gamma_\omega\in\Gamma^u_\omega$;
	 \item For each $\Lambda_{i, \omega}$ there is    $ R_{i, \omega}\in\NN$ with  $\omega\to R_{i,\omega}$  measurable such that $f^{R_{i, \omega}}(\Lambda_{i, \omega})$ is a $u$-subset and for all $x\in \Lambda_{i, \omega}$
        $$
        f^{R_{i, \omega}}_\omega(\gamma^s_\omega(x))\subset \gamma^s_{\sigma^{R_{i, \omega}}\omega}(f^{R_{i, 		\omega}}_\omega(x))  \quad\text{and} \quad       
        f^{R_{i, \omega}}_\omega(\gamma^u_\omega(x))\supset \gamma^u_{\sigma^{R_{i, \omega}}\omega}(f^{R_{i, 		\omega}}_\omega(x)).
        $$
    \end{enumerate}
\end{itemize}
This   allows us to introduce the random return time $R_\omega: \Lambda_\omega\to \NN$ by  $R_\omega|_{\Lambda_{i, \omega}}=R_{i, \omega}$, and the random induced map $f^R_\omega:\Lambda_\omega\to \cup_{n\ge 1}\Lambda_{\sigma^n\omega}$ by $f^{R_\omega}_\omega|_{\Lambda_{i, \omega}}=f^{R_{i, \omega}}_\omega$. 
For  $x\in\Lambda_\omega$, we define a sequence $(\tau_{  \omega,i})_i$ of return times 
\begin{equation}\label{oms}
\tau_{\omega, 1}(x)=R_\omega(x), \quad  \tau_{\omega, i+1}(x)=\tau_{\omega, i}(x)+ R_{\sigma^{\tau_{\omega, i}}\omega}\circ f^{\tau_{\omega, i}}_\omega(x).
\end{equation} 
%
We also define a \emph{separation time} $s_\omega (x,y)$ for $x,y\in\Lambda_\omega$ as the smallest  $n\ge 0$ such that  $(f_\omega^{R_\omega})^n(x)$ and $(f_\omega^{R_\omega})^n(y)$  lie in distinct elements of $\mathcal P_{\sigma^{\tau_\omega, n}\omega}$. Note that $s_\omega$ is constant on stable leaves.
In the properties below we consider constants $C>0$ and
$0<\beta<1$   independent of $\omega\in\Omega$ and we assume for almost all  $\omega\in\Omega$ the following are satisfied: 
\begin{itemize}
\item[(P$_2$)] \emph{Contraction on stable leaves}:  for  all $\gamma^s\in\Gamma^s_\omega$, all $x,y\in\gamma^s$ and all $ n\in \mathbb N$ we have
$$\dist(f^n_\omega(y),f^n_\omega(x))\le C \beta^n.$$
\item[(P$_3$)] \emph{Expansion and  distortion on unstable leaves}: for all $ \Lambda_{i, \omega}\in\mathcal P_\omega$  and all $x, y\in\Lambda_{i, \omega}$ with $y\in\gamma_\omega^u(x)$ we have 
    \begin{enumerate}
	\item 
$\|D^u(f_\omega^{R_\omega})^{-1}(x)\|\le \beta$;
         \item 
	$\log\frac{\det D^uf_\omega^{R_{ \omega}}(x)}{\det D^uf_\omega^{R_{ \omega}}(y)}\le 
	C\beta^{s (f_\omega^{R_{ \omega}}(x),f_\omega^{R_{ \omega}}( y))}$. 
   \end{enumerate}
\item[(P$_4$)] \emph{Regularity of the   foliations}: given   $\gamma_\omega, \gamma_\omega'\in \Gamma_\omega^u$, define $\Theta_\omega:\gamma_\omega\cap\Lambda_\omega \to \gamma_\omega'\cap \Lambda_\omega$  as $\Theta_\omega(x)=\gamma^s_\omega(x)\cap \gamma_\omega'$; we assume  
	\begin{enumerate}
	 \item  $\Theta_\omega$ is absolutely continuous and 
	\[
	 \frac{d([\Theta_\omega]_\ast\Leb_{\gamma_\omega})}{d\Leb_{\gamma_\omega'}}(x)=\prod_{i=0}^\infty\frac{\det D^uf_{\sigma^i\omega}(f^i_\omega(x))}{\det D^uf_{\sigma^i\omega}(f^i_\omega(\Theta^{-1}_\omega(x)))}.
	 \]
	\item Denoting $\rho_\omega $ the density in the previous item, we have for all  $x, y\in \gamma_\omega\cap \Lambda_\omega$
	$$\log\dfrac{\rho_\omega(x)}{\rho_\omega(y)}\le C\beta^{s(x, y)}.$$ 
	\end{enumerate}
\end{itemize}
\subsection{Tower extensions} We now introduce our main tool to study statistical properties of systems satisfying conditions {(P$_0$)}-{(P$_4$)}. We define a \emph{random tower} for almost every $\omega$ as
\begin{equation}\label{tower}
\Delta_\omega=\left\{(x, \ell) \mid
x\in  \Lambda_{ \sigma^{-\ell}\omega} , 0\le \ell\le R_{\sigma^{-\ell}\omega}(x)-1\right \}
\end{equation} 
and the random tower map $F_\omega: \Delta_\omega\to \Delta_{\sigma\omega}$ by 
\begin{equation}
F_\omega(x, \ell)=\begin{cases}
(x, \ell+1), &\quad\text{if}\quad \ell+1<R_{\sigma^{-\ell}\omega}(x), \\
(f^{\ell+1}_{\sigma^{-\ell}\omega}x, 0), &\quad\text{if}\quad \ell+1=R_{\sigma^{-\ell}\omega}(x).
\end{cases}
\end{equation}
Let  $\Delta_{\omega, \ell}=\{(x, \ell)\in\Delta_\omega\}$ be  the $\ell$th level of the tower, which is a copy of $\{x\in\Lambda_\omega\mid R_{\sigma^{-\ell}\omega}(x)>\ell\}$. 
The partition $\P_\omega$ of $\Lambda_\omega=\Domo $ naturally lifts to a partition $\Q^\omega=\{\Delta_{\omega, \ell, i}\}$ of $\Delta_\omega$ for almost every $\omega\in\Omega$. Define a sequence of partitions $\{\Q_n^\omega\}$ as follows:
\begin{equation}\label{Qn}
\mathcal Q_0^\omega=\Q^\omega \text{ and } \Q_n^\omega=\bigvee_{i=0}^{n-1}F_{\omega}^{-i}\Q^{\sigma^i\omega} \text{ for  } n\ge 1.
\end{equation}
We shall denote by $Q_{n}(z)$ the element of $\Q_n^\omega$ containing the point $z\in\Dom$. For almost every $\omega\in \Omega$ and $(x, \ell)\in \Dom $ we define tower projections $\pi_\omega:\Dom\to M$  as $\pi_\omega(x, \ell)= f^\ell_{\sigma^{-\ell}\omega}(x)$. Then $\pi_\omega$ is a semi-conjugacy,  i.e. $\pi_{\sigma\omega}\circ F_\omega =f_\omega\circ\pi_\omega$.  Let 
\begin{equation}\label{deltak}
\delta_{\sigma^k\omega, k}=\sup_{A\in \mathcal Q_{2k}^\omega}\diam(\pi_{\sigma^k\omega}(F^k_\omega(A))).
\end{equation}
By definition for any $k\ge 1 $ and $Q\in \Q_{2k}^\omega$ we have 
$$
\diam(\pi_{\sigma^k\omega} F^k_\omega(Q))\le \delta_{\sigma^k\omega, k}.
$$

\begin{remark}\label{re.exdelta}
In two of our applications (Axiom A and DA attractors)  the systems admit a stronger expansion than the one stated  in (P$_3$)(1), namely for all $\omega$ and  for all $ \Lambda_{i, \omega}\in\mathcal P_\omega$  and all $x\in\Lambda_{i, \omega}$ we have
$$\|D^u(f_\omega^{R_\omega-j})^{-1}(x)\|\le \beta^{j},\quad\text{for all   $0\le j\le R_\omega$}.$$
Therefore, in situations like this, it is straightforward to verify that $\delta_{\sigma^k\omega, k}$ decays exponentially fast in $k$, with constants not depending on $\omega$. However, this is not the case of the solenoid with intermittency presented in Subsection~\ref{se.solemio}.
\end{remark}

\subsection{Statements of abstract results}\label{gs} 
We say a hyperbolic product structure admits \emph{uniformly summable tails} if there exists $C>0$ so that for almost every $\omega\in\Omega$ we have
 $\gamma_{\sigma^{-n}\omega}\in\Gamma_{\gamma_{\sigma^{-n}\omega}}^u$ for which 
 \begin{equation}\label{eq:unie}
 \sum_{n\ge0} \Leb_{\gamma_{\sigma^{-n}\omega}}\{R_{\sigma^{-n}\omega} >n\}\le C.
 \end{equation}
 The uniform summability condition above is needed to prove the existence of an equivariant family of measures.
 
 \begin{theorem}\label{existence} 
Every  random dynamical system   with a measurable hyperbolic product structure and uniformly summable tails admits a family of equivariant physical measures. 
\end{theorem}  

 To obtain  quenched correlation decay  rates for the equivariant physical measures  we need more information on the tail of return times. 
We say a hyperbolic product structure $\{\Lambda_\omega\}_{\omega\in\Omega}$ is \emph{aperiodic} if there are $N_0\in\mathbb N$ and 
$\{t_i\in\mathbb Z_+\mid i=1, 2, ..., N_0\}$  such that for almost every  $\omega\in\Omega$ there exists  $\gamma\in \Gamma^u_\omega$ such that   gcd$\{t_i\}=1$ and 
$\Leb_{\gamma}\{x\in\Lambda_\omega\mid R_\omega(x)=t_i\}>0$.
Let $\eta\in (0,1]$ and  $C^\eta(M)$ be  the set of H\"older continuous functions of exponent\footnote{When $\eta=1$ we denote the space by $C^{\text{Lip}}$; i.e., the space of Lipschitz continuous functions.} $\eta$ on $M$. 

\begin{theorem}\label{exrates}
Consider a  random dynamical system with a measurable aperiodic  hyperbolic product structure and uniformly summable tails. 
\begin{enumerate} 
 \item If there are $C, c>0$  such that for almost every $\omega$ and some  $\gamma_\omega\in\Gamma^u_\omega$  we have  
	 \begin{equation*}
	\Leb_{\gamma_\omega}\{R_\omega>n\}\le C e^{-cn},
	\end{equation*} 	
then there are $C', c'>0$ such that  for any $\varphi, \psi\in C^\eta(M)$ there is   $C_{\varphi, \psi}>0$ so~that for almost every $\omega$
$$ \left|\mathcal C_n(\varphi, \psi,\mu_\omega)\right|\le C_{\varphi, \psi}  \max\{e^{-c'n},  \delta_{\sigma^{[n/4]}\omega, [n/4]}^\eta\}.$$ 
\item If there are   $a\in(0,1],C, c>0$ and    $n_1:\Omega\to \mathbb N$  with  $P\{n_1>n\} \le C e^{-cn^a}$  such that for almost every $\omega$ and some    $\gamma_\omega\in\Gamma^u_\omega$   for all $n\ge n_1(\omega)$ we have 
	 \begin{equation*}
	\Leb_{\gamma_\omega}\{R_\omega>n\}\le C e^{-cn^a},  
	\end{equation*} 	
then there exist 
$C', c'>0$ and $0<a'\le 1$ such that
for any $\varphi, \psi\in C^\eta(M)$ there is   a constant  $C_{\varphi, \psi}>0$ and   $C_\omega:\Omega\to \RR$ such that $P\{C_\omega>n\}\le C'e^{-c'n^{a'}}$ and for almost every $\omega$
\begin{equation*}	 
	\left|\mathcal C_n(\varphi, \psi,\mu_\omega)\right| \le  C_{\varphi, \psi}\max\{C_\omega e^{-c'n^{a'}},  \delta_{\sigma^{[n/4]}\omega, [n/4]}^\eta\}.
	\end{equation*} 
  Moreover, if $a=1$ then $a'=1$. 
 \end{enumerate}
 \end{theorem}

\begin{theorem}\label{polyrates}
Consider a  random dynamical system with a measurable aperiodic hyperbolic product structure and uniformly summable tails. Assume, there exist $C,c,a>0$,  $b\ge0$ and $d\in(0,1]$ 
 such that for almost every $\omega$ and some $\gamma_\omega\in\Gamma^u_\omega$ 
\begin{enumerate}
\item 
$\displaystyle
\int\Leb_{\gamma_\omega}\{R_\omega=n\}dP(\omega)\le C (\log n)^bn^{-2-a};
$
\item
there exists     $n_1:\Omega\to \mathbb N$  with  $P\{n_1>n\} \le C e^{-c n^d}$  such that   for all $n\ge n_1(\omega)$ 
	 \begin{equation*}
	\Leb_{\gamma_\omega}\{R_\omega>n\}\le C (\log n)^bn^{-1-a}.
	%
	\end{equation*} 	
\end{enumerate}
Then there are $C',c'>0$ and $d'\in(0,1]$ such that for every $\eps>0$ and every $\varphi, \psi\in C^\eta(M)$ there is   $C_{\varphi, \psi}$ (independent of $\eps$) and a random variable $C_\omega:\Omega\to \RR$ with  $P\{C_\omega>n\}\le C'e^{-c'n^{d'}}$  such that 
$  \left|\mathcal C_n(\varphi, \psi,\mu_\omega)\right|\le  C_{\varphi, \psi}\max\{C_\omega n^{-a+\eps},  \delta_{\sigma^{[n/4]}\omega, [n/4]}^\eta\}$ for almost every $\omega$.
\end{theorem}

\subsection{Applications}
In this subsection we provide examples of random dynamical systems where our general results above apply. 
\subsubsection{Solenoid  with   intermittency} \label{se.solemio}
Here we present an application of Theorem~\ref{polyrates} to a variation of the classical solenoid attractor on the solid torus. The idea is to replace the uniformly expanding base map by a map with a neutral fixed point. The deterministic case has been treated in~\cite{AP08a}. Here we consider a randomised version of this example.

Let $0<\alpha_0<\alpha_1<1$  be fixed real numbers and let ${\bf p}$ be the normalised Lebesgue measure on $[\alpha_0, \alpha_1]$. For each $\alpha\in [\alpha_0, \alpha_1]$, consider the circle map $T_\alpha: \SS^1\to \SS^1$ defined by 
\[
T_\alpha(x)=\begin{cases} x(1+(2x)^\alpha), & \text{ if }  x\in [0, 1/2); \\ x-2^\alpha(1-x)^{1+\alpha},  &\text{ if }  x\in [1/2, 1). \end{cases}
\]
Let $M=\SS^1\times \DD^2$ denote the solid torus in $\RR^3$, where $\DD^2$ is the unit disk in $\RR^2$, and  consider a skew product  $g_\alpha:M\to M$,
\begin{equation}\label{map}
g_\alpha(x, y, z)=\left(T_\alpha(x), \frac{1}{10}y+\frac{1}{2}\cos x, \frac{1}{10}z+\frac{1}{2}\sin x\right).
\end{equation}
As in the case of classical solenoid (see for example \cite{KH95}) it is easily checked that $g_\alpha$ is a diffeomorphism onto its image for every $\alpha\in [\alpha_0, \alpha_1]$. 

Now let $\Omega=[\alpha_0, \alpha_1]^{\ZZ}$, $P={\bf p}^{\ZZ}$  and $\omega_k$ denote the $k$th component of $\omega\in \Omega$. 
Note that the composition defined as  $g_\omega^n=g_{\omega_{n-1}}\circ\dots\circ g_{\omega_0}$ defines a random dynamical system on $M$ as in Section~\ref{setup}. The proof of the next result is provided in Section~\ref{se.sole}.

\begin{theorem}\label{Appl}
  The random system $\{g_\omega\}_{\omega\in\Omega}$ admits a family~$\{\mu_\omega\}_{\omega\in\Omega}$ of equivariant physical measures. Moreover, for every $\eps>0$ and  almost every $\omega\in\Omega$ there exists a constant $C_\omega>0$ such that for any  $\varphi,\psi\in C^{Lip}(M)$  
$$|\mathcal C_n(\varphi, \psi,\mu_\omega)|\le C_\omega C_{\varphi, \psi}n^{1-1/\alpha_0+\eps},$$
where $C_{\varphi, \psi}$ is a positive constant depending only on $\varphi$ and $\psi$, while  $C_\omega$ satisfies $ P\{C_\omega>n\}\le Ce^{u'n^{v'}}$ for some $u'>0$ and $v'\in(0, 1)$.
\end{theorem}

Note that the quenched decay rate in the above theorem is driven by the fastest mixing map for the set of parameters considered. Obviously such a rate can only hold for almost all $\omega$ and cannot hold for all $\omega$.

\subsubsection{Partially hyperbolic attractors}\label{PHA}
Next we state an auxiliary result, Theorem \ref{randomPH} below, which is our main tool to verify that random perturbations of Axiom A and derived from Anosov attractors admit a random hyperbolic product structure with exponential tails. Theorem \ref{randomPH} extends the results of \cite{AL15} to the random setting. We start by introducing random perturbations of partially hyperbolic attractors with uniformly contracting direction and mostly expanding centre-unstable direction.  To build such structures, we need to assume that the constituent maps of  the random dynamical system belong in some small neighbourhood of a deterministic transitive dynamical system, in particular to assure the  existence of recurrence times. 

Let $f\in \diff^{1+}(M)$ and $U\subset M$ be an open set for which ${f(\bar U)}\subset U$, which we refer to  as a \emph{trapping region}. We  are going to consider a partially hyperbolic situation as in \cite{ABV00}, for which the corresponding attractor $K=\cap_{n\ge0}f^n(U)$ has a dominated splitting
$T_K M =E^{cs}\oplus E^{cu},$
with a uniform contracting direction $E^{cs}$. Let us now introduce these concepts in the random setting.
We   consider  $\mathcal F$  a small neighbourhood of $f$ in the $C^1$ topology and ${\bf p}$ a compactly supported  Borel probability measure whose support $\mathbb B$   is contained in $\mathcal F$ and $f\in\mathbb B$. 
For   $\mathcal F$  sufficiently small, 
we   still have   $f_\omega(\bar U)\subset U$ for   every  $\omega\in\Omega=\mathbb B^{\mathbb Z}$ and $f_\omega\in\mathcal F$.
Defining the random attractor
\begin{equation}\label{eq.randomatt}
K_\omega=\bigcap_{n\ge0}f^n_{\sigma^{-n}\omega}(U)
\end{equation}
and using the $C^1$ persistence of dominated splittings, for each $\omega\in\Omega$  there exists a continuous splitting of  $T_{K_\omega}M=E^{cs}_\omega \oplus E^{cu}_\omega $, which is \emph{$Df_\omega$-equivariant}:   
$$
Df_\omega(x)E^{cs}_\omega(x)=E^{cs}_{\sigma\omega}(f_\omega (x))\qand Df_\omega(x)E^{cu}_\omega(x)=E^{cu}_{\sigma\omega}(f_\omega (x)),
$$  
see Section~\ref{sec:bounded-curvature}. Moreover, this is still  \emph{partially hyperbolic}: there exists $0<\lambda<1$ such that (for some choice of Riemannian metric  on $M$, which is independent of $\omega$) for each $\omega\in\Omega$ and $x\in K_\omega$ we have
 \begin{equation}\label{domination1}
    \|Df_\omega \mid E^{cs}_\omega(x)\| \cdot \|Df^{-1}_\omega \mid E^{cu}_{\sigma\omega}(f_\omega(x))\| <\lambda \qand
    \|Df_\omega \mid E^{cs}_\omega(x)\|< \lambda.
    \end{equation}
To emphasise uniform contraction we will write $E^s$ instead of $E^{cs}$ and say that $K_\omega$ is a partially hyperbolic attractor of the type $E^{s}\oplus E^{cu}$. We say that $f_\omega$ has  \emph{non-uniform expansion} in  the  \emph{centre-unstable direction} $E^{cu}$ 
  on some set  $H_\omega\subset U$ if there exists $c>0$ such that for   every $\omega\in\Omega$ and  $\leb$ almost every $x\in H_\omega$  
\begin{equation}\label{NUE}
 \limsup_{n\to+\infty} \frac{1}{n}
    \sum_{j=1}^{n} \log \|Df^{-1}_{\sigma^{j-1}\omega} \mid E^{cu}_{\sigma^j\omega}(f^{j}_\omega(x))\|<-c.
\end{equation} 
For the specific case of $f$, this  implies that there is a  local  unstable manifold $D\subset K$ such that points in $D$ have  non-uniform expansion  in the centre-unstable direction; see \cite[Theorem~A]{AP08}.
We say that    $\{f_\omega\}_{\omega\in\Omega}$    is \emph{$C^1\!$-close to $f|_D$ on domains $\{D_\omega\}_{\omega\in\Omega}$  of $cu$-nonuniform expansion}, if 
\begin{enumerate}
\item  each  $f_\omega$ is  close to $f$ in the $C^1$ topology;
\item each $f_\omega$ has a  partially hyperbolic  set $K_\omega$ of the type $E^{s}\oplus E^{cu}$  with  a local  unstable manifold $D_\omega\subset K_\omega$  close to $D$ in the $C^1$ topology;
\item each $f_\omega$ is non-uniformly expanding on random orbits along the $E^{cu}_\omega$ direction for $\Leb_{D_\omega}$ almost every  $x\in D_\omega$.
\end{enumerate}
If   \eqref{NUE} holds at  $x\in D_\omega$,  then
the \emph{expansion time} function  
\begin{equation}\label{exptime}
    \mathcal E_\omega(x) = \min\left\{N\ge 1\colon  \frac{1}{n}
\sum_{j=1}^{n} \log \|Df^{-1}_{\sigma^{j-1}\omega} \mid E^{cu}_{\sigma^j\omega}(f^{j}_\omega(x))\| < -c, \quad
\forall\,n\geq N\right\}
\end{equation}
is well defined and finite. The set $\{\mathcal E_\omega>n\}$ plays an important role in the statement of our next result.

\begin{theorem}\label{randomPH}
Let $f\in \diff^{1+}(M)$ have a transitive partially hyperbolic set $K\subset M$ and a local  unstable manifold $D\subset K$. If   $\{f_\omega\}_{\omega\in\Omega}$  is a random perturbation   $C^1\!$-close to $f|_D$ on domains $\{D_\omega\}_{\omega\in\Omega}$ of $cu$-nonuniform expansion,
then there exists $\Lambda_\omega\subset K_\omega$ with a   hyperbolic product structure.  Moreover, 
 if there exist $C,c>0$  and $0< \tau \le 1$  such that 
$\Leb_{D_\omega}\{\E_\omega>n\}=  Ce^{-cn^\tau} $ for all $\omega\in\Omega$, then there exist $C',c'>0$ such that $\Leb_{D_\omega}\{R_\omega>n\}= C'e^{-c'n^\tau}$ for all $\omega\in\Omega$.
\end{theorem}

\subsubsection{Axiom A attractors}
As an application of Theorem~\ref{exrates} and Theorem~\ref{randomPH} we consider the classical  case of topologically mixing uniformly hyperbolic Axiom A  attractors.  Observe that this is a particular case of   partially hyperbolic attractors, as considered in Section~\ref{PHA}. In this case, the random attractors are defined as in~\eqref{eq.randomatt}. 
As in the general  partially hyperbolic setting,  we   consider  $\mathcal N$  a small neighbourhood of $f$ in the $C^1$ topology and $\theta$ a compactly supported  Borel probability measure whose support contains $f$ and   is contained in $\mathcal N$.  
The proof of the next result is given in Subsection~\ref{se.AxA}.


 \begin{theorem}\label{co.DCAxA}
Let $f\in \diff^{1+}(M)$  have a topologically mixing uniformly  hyperbolic attractor $K\subset M$.
If $\{f_\omega\}_{\omega\in\Omega}$ is a small random perturbation of $f$, 
 then it admits a unique family of equivariant physical measures $\{\mu_\omega\}_{\omega\in\Omega}$ supported on the random attractors. Moreover,  there is $c>0$ such that for any $\varphi, \psi\in C^\eta(M)$ there is a constant  $C_{\varphi, \psi}>0$    for which 
$  |\mathcal C_n(\varphi, \psi,\mu_\omega)|\le   C_{\varphi, \psi}e^{-cn }$.
\end{theorem}

 
%

\subsubsection{Derived from Anosov}\label{se.exampleDA}

Here we present random perturbations of partially hyperbolic diffeomorphisms 
whose centre-unstable direction is non-uniformly expanding. The deterministic case was introduced in \cite[Appendix~A]{ABV00} as a perturbation (not necessarily small) of an  Anosov diffeomorphism.  We sketch  below the main steps of its construction.


Consider a linear Anosov diffeomorphism $ f_0$ on the $d$-dimensional torus $M=\TT^d$, for some $d\ge 3$, having a  hyperbolic splitting $TM=E^u \oplus E^s$ with $\dim(E^u)\ge2$. Let $V\subset M$ be a small compact domain, such that for  the canonical projection $\pi:\mathbb R^d\to  M$ there exist unit open cubes $K^0,K^1\subset \mathbb R^d$ with $V\subset\pi(K^0)$ and $f_0(V)\subset\pi(K^1)$.  Let $f$ be a diffeomorphism on $M$ such that:
\begin{enumerate}
\item $f$ has invariant cone fields $C^{cu}$ and $C^s$  with small width  containing the unstable bundle $E^u$ and the stable bundle $E^s$ of~$f_0$, respectively, with $Df$ contracting uniformly vectors in $C^s$;
\item $f$ is \emph{$cu$-volume expanding in $M$}: there is $\sigma_1>0$ such that  $|\det(Df|T_x D^{cu})|>~\sigma_1$ for any $x\in M$ and any disk $D^{cu}$ through $x$ whose tangent space is contained in~$C^{cu}$;
\item $f$   is \emph{$cu$-expanding   $M\setminus V$}: there is $\sigma_2<1$ such that $\|(Df|T_x D^{cu})^{-1}\|<\sigma_2$ for $x\in M\setminus V$ and any disks $D^{cu}$ whose tangent space is contained in~$C^{cu}$;
\item $f $ is \emph{not too $cu$-contracting} on $V$: there is a small $\delta_0>0$ such that $\|(Df|T_x D^{cu})^{-1}\|<1+\delta_0$ for any $x\in V$ and any disk $D^{cu}$  whose tangent space is contained in~$C^{cu}$.
\end{enumerate}
For example, if $f_1:\TT^d \to \TT^d$ is a diffeomorphism satisfying items (1), (2) and (4) above and coinciding with $f_0$ outside $V$, then any $f$ in a $C^1$ neighbourhood of $f_1$ satisfies all the conditions (1)-(4). It is not difficult to see that examples of this type can be produced in such a way that the following property holds:
\begin{enumerate}
\item[(5)] $f$ {\em preserves} the  unstable foliation of $f_0$. 
\end{enumerate}
The above construction yields a family   in $\diff^{1+}(M)$ such 
  for   any centre-unstable disk~$D$ we have  $\leb_D\{\mathcal E>n\}$ with exponential decay; see~\cite[Appendix~A]{ABV00}. The existence of physical measures for these maps follows from the results in~\cite{ABV00}.
  Stably ergodic examples of this type with a unique physical measure   are given in~\cite{T04} under a volume hyperbolicity condition. This in particular assures the next property:
  \begin{enumerate}
\item[(6)] $f$ is \emph{transitive} in $M$. 
\end{enumerate}
 Let $\mathcal F$ be a family of maps in $\diff^{1+}(M)$ satisfying properties (1)-(6) above.

In the next theorem we need to take powers of random perturbations since  aperiodicity cannot be a priori assured. We consider the random perturbations $\{f^N_\omega\}_{\omega\in\Omega}$ of $f^N$, with the shift map $\sigma^N$. Accordingly, the notation of correlations in the next theorem is understood for $\{f^N_\omega\}_{\omega\in\Omega}$.
  \begin{theorem}\label{co.Derived}
Let $f\in\mathcal F$  and  $\{f_\omega\}_{\omega\in\Omega}$ be  a small random perturbation of $f$   in~$\mathcal F$. Then there is some $N\in\mathbb N$ such that
 $\{f^N_\omega\}_{\omega\in\Omega}$ admits a   family of equivariant physical measures $\{\mu_\omega\}_{\omega\in\Omega}$. Moreover,  there exists $c>0$ such that for any $\varphi, \psi\in C^\eta(M)$ we have a constant  $C_{\varphi, \psi}>0$    for which 
$ |\mathcal C_{Nn}(\varphi, \psi,\mu_\omega)|\le   C_{\varphi, \psi}e^{-cn }$.
\end{theorem}
The proof of Theorem \ref{co.Derived} is provided in Subsection~\ref{se:DA}.

\section{Random hyperbolic towers}\label{se.towers}
Our proofs of Theorems \ref{existence}, \ref{exrates} and \ref{polyrates} are based on considering quotients along stable manifolds to obtain a random Young tower for the quotient dynamics. For this purpose we first introduce measures that are suitable to take quotients along stable manifolds.
\subsection{Natural measures} 
Fix $\hat \gamma_\omega\in \Gamma^u_\omega$. For any   $\gamma_\omega\in \Gamma_\omega^u$ and $x\in \gamma_\omega\cap \Lambda_\omega$ let $\hat x_\omega$ be the intersection point $\gamma^s_\omega(x)\cap \hat\gamma_\omega$. 
Define $\hat \rho_\omega: \gamma_\omega\cap \Lambda_\omega\to \RR $ as 
\[
\hat \rho_\omega(x)=\prod_{i=0}^\infty\frac{\det D^uf_{\sigma^i\omega}(f^i_\omega(x))}{\det D^uf_{\sigma^i\omega}(f^i_\omega(\hat x_\omega))}.
\]
Further, let $m_{\gamma_\omega}$ be the measure defined by 
\[
\frac{dm_{\gamma_\omega}}{d\Leb_{\gamma_\omega}}=\hat \rho_\omega \mathbb{1}_{\gamma_\omega\cap \Lambda_\omega}.
\]
For any $\omega\in\Omega$, $\gamma_\omega'\in \Gamma_\omega^u$  and $x'\in \gamma_\omega'\cap\gamma_\omega^s(x)$ i.e. $x'=\Theta_\omega(x)$, by  (P$_4$)(1) we have 
$$
\frac{\hat\rho_\omega(x')}{\hat\rho_\omega(x)}= \frac{d([\Theta_\omega]_\ast\Leb_{\gamma_\omega})}{d\Leb_{\gamma_\omega'}}(x'),
$$which implies 
$$
\frac{d([\Theta_\omega]_\ast m_{\gamma_\omega})}{ d\Leb_{\gamma_\omega'}}(x')=\hat\rho_\omega(x)\frac{d([\Theta_\omega]_\ast \Leb_{\gamma_\omega})}{ d\Leb_{\gamma_\omega'}}(x')= \hat\rho_\omega(x')
$$
Thus, 
\begin{equation}\label{eq.teta} 
[\Theta_\omega]_\ast m_{\gamma_\omega}=m_{\gamma_\omega'}.
\end{equation}

\begin{lemma}\label{lem:bbdR}
Assume that $f^{R_\omega}_\omega(\gamma\cap \Lambda_\omega)\subset \gamma'$ for $\gamma\in \Gamma_\omega^u$ and $\gamma'\in \Gamma_{\sigma^{R_\omega}\omega}^u$.  Let $Jf^{R_\omega}_\omega$ denote the Jacobian of $f^{R_\omega}_\omega$ with respect to $m_\gamma$ and $m_\gamma'$.  Then 
\begin{enumerate}
\item  $Jf^{R_\omega}_\omega(x)=Jf^{R_\omega}_\omega(y)$ for any $x\in \gamma\cap \Lambda_\omega$ and $y\in \gamma^s_\omega(x)$. 
\item There exits a constant $C_1>0$ such that for any $x, y\in \Lambda_\omega\cap\gamma$
$$
\left|\frac{Jf^{R_\omega}_\omega(x)}{Jf^{R_\omega}_\omega(y)}-1\right|\le C_1\beta^{s(f_\omega^{R_\omega}(x), f_\omega^{R_\omega}(y))}.
$$
\end{enumerate}
\end{lemma}
\begin{proof}
 (1) For $\Leb_\gamma$ almost every $x\in\gamma\cap\Lambda_\omega$ we
 have
  \begin{equation}\label{eq.jaco}
  Jf^{R_\omega}_\omega(x)=\left|\det D^uf^{R_\omega}_\omega(x)\right|\cdot \frac{\hat \rho_{\sigma^{R_\omega}\omega}(f^{R_\omega}_\omega(x))}{\hat \rho_\omega(x)}.
  \end{equation}
 Denoting $\varphi_\omega(x)=\log|\det D^uf_\omega(x)|$, we write  the right hand side of \eqref{eq.jaco}  as 
 
 \begin{align*}
  \sum_{i=0}^{R_\omega-1}\varphi_{\sigma^i\omega}(f^i_\omega(x)) 
   - &\sum_{i=0}^{\infty}\left(\varphi_{\sigma^i\omega}(f^i_\omega(x))- \varphi_{\sigma^i\omega}(f^i_\omega(\hat     x)\right)\\
  + &\sum_{i=0}^{\infty}\left(\varphi_{\sigma^{i+R_\omega}\omega}(f^i_{\sigma^{R_\omega}\omega}(f^{R_\omega}_\omega(x)))- \varphi_{\sigma^{i+R_\omega}\omega}(f^i_{\sigma^{R_\omega}\omega}(\widehat{f^{R_\omega}_\omega(x)})\right)\\
     =\sum_{i=0}^{R_\omega-1}\varphi_{\sigma^i\omega}(f^i_\omega(\hat x))+&
\sum_{i=0}^{\infty}\left(\varphi_{\sigma^{i+R_\omega}\omega}(f^i_{\sigma^{R_\omega}\omega}(f^{R_\omega}_\omega(\hat x)))- \varphi_{\sigma^{i+R_\omega}\omega}(f^i_{\sigma^{R_\omega}\omega}(\widehat{f^{R_\omega}_\omega(x)}))\right).
 \end{align*}
Thus we have shown that $Jf^{R_\omega}_\omega(x)$ can be expressed just in terms of
$\hat x$ and $\widehat{f^{R_\omega}_\omega(x)}$, which is enough for proving the
first part of the lemma.
\smallskip

 (2) It follows from \eqref{eq.jaco} that
  \[
  \log\frac{Jf^{R_\omega}_\omega(x)}{Jf^{R_\omega}_\omega(y)}=\log\frac{\det D(f^{R_\omega}_\omega)^u(x)}{\det
  D(f^{R_\omega}_\omega)^u(y)}+
  \log\frac{\hat\rho_{\sigma^{R_\omega}\omega}(f^{R_\omega}_\omega(x))}{\hat\rho_{\sigma^{R_\omega}\omega}(f^{R_\omega}_\omega(y))}+\log\frac{\hat \rho_\omega(y)}{\hat \rho_\omega(x)}.
  \]
Observing that $s(x,y)> s(f^{R_\omega}_\omega(x),f^{R_\omega}_\omega(y))$ the conclusion follows from (P$_3$) and (P$_4$)(2).
\end{proof}

Notice that the family $\{m_\gamma\}$  introduced in the previous section  defines a measurable system on $\Lambda_\omega$ since $\Gamma_\omega^u$ is a continuous family. Thus, it defines a measure  $m_\omega$ on $\Lambda_\omega$.   We introduce a  measure on $\Delta_\omega$  that we still denote $m_\omega$ by letting $m_\omega\vert\Delta_{\omega, \ell}=m_{\sigma^{-\ell}\omega}\vert\{R_{\sigma^{-\ell}\omega} >\ell\}$. 
We let $JF_\omega$ denote the Jacobian of $F_\omega$ with respect to the measure
$m_\omega$.

\begin{lemma}\label{lem:distFk}
There exists $C_3>0$ such that for almost every $\omega\in \Omega$,  for any $k\ge 1$ and $x, y\in Q\in \Q_{k-1}^\omega$ 
$$
\left|\frac{F^k_\omega(x)}{F^k_\omega(y)}-1\right|\le C_3\beta^{s(F_\omega^k(x), F_\omega^k(y))}.
$$
\end{lemma}

 \begin{proof} Recall that  by the definition of $\tau_{\omega, j}$ we have $(F^{R_\omega})^{j}_\omega=F^{\tau_{\omega, j}}_\omega$.  Also let $\P_{\omega, k}= \vee_{j=0}^{n-1} (F^{\tau_{\omega, j}})^{-1}_{\sigma^{\tau_{\omega, j}} \omega} \P_{\sigma^{\tau_{\omega, i}}\omega}$. 
By Lemma~\ref{lem:bbdR}  for a.e. $\omega$, for all $i\ge 1$ and all $ x,y\in \Delta_{\omega, 0,i}$ holds
\begin{equation}\label{jac}
\mdlo{ \frac{J F^{R_\omega}_\omega(x)}{J  F^{R_\omega}_\omega(y)}-1 } \le  C_1 \beta^{s( F^{
R_\omega}_\omega(x), F^{ R_\omega}_\omega(y))}.
\end{equation}
It follows that there is a constant $C_F>0$, which is independent of $\omega$ such that for all
$n\ge1$ and all $x,y$ belonging to a same element of $\P_{\omega, n}$ 
\begin{equation}\label{jacfrn}
\left| \frac{J(F^{R_\omega}_\omega)^n(x)}{J(F^{R_\omega}_\omega)^n(y)}-1\right| \leq C_F
\beta^{s((F^{R_\omega}_\omega)^n (x),(F^{R_\omega}_\omega)^n(y))},
\end{equation}
By definition if $x$ and $y$ belong to a same element of $\P_{\omega, n}$, then $(F^{R_\omega}_\omega)^j(x)$ and
$(F^{R_\omega}_\omega)^j(y)$ belong to a same element of $\P_{\sigma^{\tau_{\omega, j}}\omega}$ for every
$0\le j<n$. Moreover,

\begin{equation}\label{esse}
s((F^{R_\omega}_\omega)^j(x),(F^{R_\omega}_\omega)^j(y))=s((F^{R_\omega}_\omega)^n(x),(F^{R_\omega}_\omega)^n(y))+(n-j).
\end{equation}
Then by \eqref{jac} and \eqref{esse}  we have 
\begin{align}
&\log\frac{J(F^{R_\omega}_\omega)^n(x)}{J(F^{R_\omega}_\omega)^n(y)} =
\sum_{j=0}^{n-1}\log\frac{JF^{R_\omega}_\omega((F^{R_\omega}_\omega)^j(x))}{JF^{R_\omega}_\omega((F^{R_\omega}_\omega)^j(y))}\nonumber\\
&\le\sum_{j=0}^{n-1}C_1\beta^{s((F^{R_\omega}_\omega)^n(x),(F^{R_\omega}_\omega)^n(y))+(n-j)-1} \le
 C_1\beta^{s((F^{R_\omega}_\omega)^n(x),(F^{R_\omega}_\omega)^n(y))}\sum_{j=0}^{\infty}\beta^j
 \nonumber\\
&\le C_F\beta^{s((F^{R_\omega}_\omega)^n(x),(F^{R_\omega}_\omega)^n(y))},
\label{logo}
\end{align}
where $C_F>0$ depends only on $C_1$ and $\beta$. This implies that
\eqref{jacfrn} holds.

Since $JF_\omega (x) \equiv 1$  for every $x\in \Delta_{\omega, \ell}$ with $\ell >0$,  it follows that $JF^k_\omega(x)=J(F^{R_\omega}_\omega)^n(x')$ and
$JF^k_\omega(y)=J(F^{R_\omega}_\omega)^n(y')$, where $n$ is the number of visits of $x$ and
$y$ to the base  prior to time $k$, and $x', y'$ are the projections of $x,y$ to the zeroth level of  the tower $\Delta_{\omega'}$ for some $\omega'\in\Omega$.  The point $x',y'$ belong to a same element of
$\P_{\omega', n}$ and $s(x,y)=s(x',y')$. Using \eqref{jacfrn}  and  \eqref{jacfrn}  we obtain that 
\begin{eqnarray} \left| \frac{JF^k_\omega(x)}{JF^k_\omega(y)}-1\right| \leq C_F
\beta^{s(F^k_\omega(x),F^k_\omega(y))}. \label{jac5}
\end{eqnarray}
For all $k\ge1$ and all  $x,y\in\Delta_\omega$ belonging to a same element of $\Q_{k-1}^\omega$.
 \end{proof}
 
 \subsection{Equivariant measures}
In this subsection we prove Theorem \ref{existence}.
For $x\in\Lambda_\omega$ we set $\tilde F_\omega(x)=f^{R_\omega(x)}_\omega(x)\in \Lambda_{\sigma^{R_\omega(x)}(\omega)}$. We
refine recursively $\mathcal P_\w$ on $\Lambda_\w$ with the partitions associated to the images of each element of $\mathcal P_\w$: 
$$\mathcal P_\w^{(k)}=\bigvee_{j=0^{\phantom{j}}}^{k}\bigvee_{n\in
L_{\w}^j}(\tilde F_\w^j)^{-1}{\mathcal P}_{\sigma^n\w},$$
 where
$L_{\w}^j=\{n\in\NN_0:\tilde F_\w^j(\Lambda_\w)\cap\Lambda_{\sigma^n\w}\neq\emptyset\}$, and set $\mathcal{P}^{(k),n}_{\w}=\mathcal
P_\w^{(k)}\cap(\tilde F_\w^k)^{-1}(\Lambda_{\sigma^n\w})$, for
$n\in\NN$. Note that if $B\in\mathcal{P}^{(k),n}_{\w}$ then
$\tilde F_\w^k(B)=\Lambda_{\sigma^n\w}$.
Given $B\subset\Lambda_{\omega}$ set
$$(\tilde F^{-1})_{\omega}(B)=\bigsqcup_{n\in\NN}\left\{x\in\Lambda_{\sigma^{-n}
(\omega) } : R_{\sigma^{-n}(\omega)}(x)=n \qand
\tilde F_{\sigma^{-n}(\omega)}(x)\in B\right\},$$
and define $[(\tilde F^{j})^{-1}]_{\omega}(B)$ by induction. Let
$\lambda_{\sigma^{-n}(\omega)}$ be a probability measure on
$\Lambda_{\sigma^{-n}(\omega)}$ and define
\begin{equation*}\label{pseudopf}
{(\tilde F^{j})_{\omega}}^{*}\{\lambda_{\sigma^{-n}(\omega)}\}_{n\in\NN}(B)=\sum_{n\in\NN}
\lambda_{\sigma^{-n}(\omega)}([(\tilde F^{j})^{-1}]_{\omega}(B)\cap\Lambda_{\sigma^{-n}
(\omega) } ).
\end{equation*}
\begin{proposition}\label{exist.mu.induced} For almost every $\omega\in
\Omega$ there exists a probability measure $\tilde\nu_{\sigma^{-n}\w}$ on $\Lambda_{\sigma^{-n}\w}$ such that
${(\tilde F)_\w}^{*}\{\tilde\nu_{\sigma^{-n}(\omega)}\}_{n\in\NN}=\tilde\nu_\w$. Moreover, $\tilde\nu_\omega$ admits absolutely continuous conditional measures on unstable leaves. 
\end{proposition}

\begin{proof}
For each $\omega\in\Omega$, choose  $\gamma_\omega^0\in\Gamma^u_\omega$ and set $m_{\omega}^0=m_{\gamma^0_\omega} $. For $j\geq 1$ define  
$$m_\w^j=\sum_{k=1}^j\sum_{B\in\mathcal{P}^{(k),j}_{\sigma^{-j}(\w)}}({\tilde F^k_{\sigma^{-j}(\w)}})_{*}(m^0_{ {\sigma^{-j}(\w)}}\vert B).$$ 
By Lemma \ref{lem:bbdR}, for any measurable set $A\subset\Lambda_\w$
\begin{eqnarray*}
m_\w^j(A)
&\leq&\sum_{k=1}^j\sum_{B\in\mathcal{P}^{(k),j}_{\sigma^{-j}(\w)}} C_1{m^0_{ \omega}(A)}m_{\sigma^{-j}(\w)}(B)\\
&\leq&C_2m^0_{\omega}(A),
\end{eqnarray*}
where we have used assumption \eqref{eq:unie}. 
Let $$\tilde \nu_\w^n=\frac1n\sum_{j=0}^{n-1}m_\w^j.$$
 Then $\tilde \nu_\w^n$ has an accumulation point, in the weak* topology. Let $\tilde \nu_\w^{n_k}$ be a convergent subsequence. By a diagonal argument we construct along the sequence $\{ \sigma^k \omega \}$, for almost every $\omega\in\Omega$, a convergent subsequence $\{\tilde\nu_{\sigma^\ell\omega}^{n_k}\}$ for every $\ell \in \mathbb Z$.  The limiting measure~$\tilde\nu_\omega$ satisfies the equivariance property; i.e. ${(\tilde F)_\w}^{*}\{\tilde\nu_{\sigma^{-n}(\omega)}\}_{n\in\NN}=\tilde\nu_\w$.
 
 We now show that $\tilde\nu_\omega$ admits absolutely continuous conditional measures on unstable leaves. Let 
 $\rho^j_{ \gamma_\omega}$ be the density of $m_{\omega}^j$ with respect to $m_{ \gamma_\omega}$ on an unstable leaf ${ \gamma_\omega}$. Observe that $\rho^j_{ \gamma_\omega}\equiv0$  or by Lemma \ref{lem:bbdR}
 \begin{equation*}
 \frac{\rho^j_{ \gamma_\omega}(y)}{\rho^j_{ \gamma_\omega}(x)} \le \text{exp}(C_F \beta^{s(x, y)}),
\end{equation*}
for all $x,y\in \gamma_\omega\cap\Lambda_\omega$. This implies that there exists $C_0>0$ such that
\begin{equation}\label{eq:density_reg}
\frac1{C_0}\le {\rho^j_{ \gamma_\omega}(x)}\le C_0
\end{equation}
for all $x\in \gamma_\omega\cap\Lambda_\omega$. Let $(\mathcal U_k)$ be a sequence, where $\mathcal U_k$ is a finite partition of $\Lambda_\omega$ consisting of $u$-subsets with $\mathcal U_1 \prec \mathcal U_2  \prec\cdots$ and $\vee_{i=1}^{\infty}\mathcal U_k$ is a partition of $\Lambda_\omega$ into unstable leaves. Let $U_k\in \mathcal U_k$ 
 containing $ \gamma_\omega$ and shrinking to $ \gamma_\omega$.
Given  $\mathcal O\subset \gamma_\omega$   an open set with $m_{ \gamma_\omega}(\partial\mathcal O)=0$, let  $  S_\mathcal O$ be the $s$-subset of $\Lambda_{\omega}$ corresponding to $\mathcal O$.  By \eqref{eq:density_reg} we have
 $$
\frac{1}{C_0^2}\frac{m_{ \gamma_{\omega}}(  S_{\mathcal O})}{m_{ \gamma_\omega}(\Lambda_\omega)}\le \frac{  m^j_{ \omega}(  U_k\cap  S_{\mathcal O}) }{m^j_{ \omega}( U_k)}
 =\frac{\int\rho^j_{ \gamma_\omega}\cdot \mathbb 1_{  U_k\cap  S_{\mathcal O}}dm_{\gamma_\omega}}{\int\rho^j_{ \gamma_\omega}\cdot \mathbb 1_{  U_k}dm_{ \gamma_\omega}}\le C_0^2\frac{m_{ \gamma_{\omega}}(  S_{\mathcal O})}{m_{ \gamma_\omega}(\Lambda_\omega)}.
 $$
 Consequently,
 $$
\frac{1}{C_0^2}\frac{m_{ \gamma_{\omega}}( S_{\mathcal O})}{m_{\gamma_\omega}(\Lambda_\omega)}\le \frac{\tilde \nu^n_{ \omega}(  U_k\cap  S_{\mathcal O}) }{\tilde \nu^n_{ \omega}(  U_k)}\le C_0^2\frac{m_{ \gamma_{\omega}}(  S_{\mathcal O})}{m_{ \gamma_\omega}(\Lambda_\omega)},
$$
and so
 $$
\frac{1}{C_0^2}\frac{m_{ \gamma_{\omega}}( S_{\mathcal O})}{m_{\gamma_\omega}(\Lambda_\omega)}\le \frac{\tilde \nu_{ \omega}(  U_k\cap  S_{\mathcal O}) }{\tilde \nu_{ \omega}(  U_k)}\le C_0^2\frac{m_{ \gamma_{\omega}}(  S_{\mathcal O})}{m_{ \gamma_\omega}(\Lambda_\omega)}.
$$
Notice that $X_n=  \mathbb E_{\tilde \nu_{ \omega}}(\mathbb 1_{S_{\mathcal O}}|\mathcal F_k)$, where $\mathcal F_k$ is the sigma algebra generated by $\mathcal U_k$, is a martingale. By the Martingale Convergence Theorem
 $$
\frac{1}{C_0^2}\frac{m_{ \gamma_{\omega}}( S_{\mathcal O})}{m_{\gamma_\omega}(\Lambda_\omega)}\le \tilde \nu_{\gamma_\omega}(S_{\mathcal O}) \le C_0^2\frac{m_{ \gamma_{\omega}}(  S_{\mathcal O})}{m_{ \gamma_\omega}(\Lambda_\omega)}.
$$
for almost every $\gamma_\omega$. 
\end{proof}

Notice, by the above construction and the absolute continuity of the holonomy map, that the family of measures $\{\tilde\nu_\omega\}$ are physical. We introduce a  measure on $\Delta_\omega$  that we denote $\nu_\omega$ by letting $\nu_\omega\vert\Delta_{\omega, \ell}=\tilde\nu_{\sigma^{-\ell}\omega}\vert\{R_{\sigma^{-\ell}\omega} >\ell\}$. 

\subsection{Quotient dynamics}\label{s.quotientdynamics}

Let $\bar\Lambda_\omega=\Lambda_\omega/\sim$, where $x\sim y$ if and only if $y\in
\gamma^s_\omega(x)$. This quotient space  gives rise to a quotient tower
$\bar\Delta_\omega$ with levels $\bar \Delta_{\omega, \ell}=\Delta_{\omega, \ell}/\sim$. A
partition of $\bar\Delta_\omega$ into $\bar\Delta_{\omega, 0,i}$, that we denote by
$\bar\P_\omega$, and a sequence $\bar\Q_n^\omega$ of partitions of $\bar\Delta_\omega$
as in \eqref{Qn} are defined in a natural way.

As $f^{R_\omega}_\omega$ takes $\gamma^s$-leaves to $\gamma^s$-leaves and $R_\omega$ has
been defined in such a way that it is constant along the stable manifolds,  the return time $\bar R\colon \bar\Delta_{\omega, 0}\to\NN$, tower map $ \bar
F_\omega\colon\bar\Delta_\omega\to\bar\Delta_{\sigma\omega}$ and the separation time $\bar
s_\omega\colon \bar\Delta_{\omega, 0}\times \bar\Delta_{\omega, 0}\to \NN$ naturally induced by
the corresponding ones defined for $\Delta_{\omega, 0}$ and $\Delta_\omega$ for almost every $\omega$.
We extend the separation time to $\bar\Delta_\omega$ by taking $ \bar s(x,y)= \bar s(x',y')$ if $x$ and $y$ belong to the same $\bar\Delta_{\omega, \ell,i}$, where $x',y'$ are the corresponding elements of $\bar\Delta_{\sigma^{-\ell}\omega, 0,i}$, and  $\bar s(x,y)=0$ otherwise. Since \eqref{eq.teta} holds, for a.e. $\omega\in \Omega$ we define a  measure $\bar m_\omega$ on
$\bar\Delta_\omega$ whose representative  is $m_\gamma$ on each $\gamma\in\Gamma^u_\omega$. We also extend the return time by setting 
$$
\hat R_\omega(x)=\min\{ n> 0 \colon \textrm{ $\bar
F^n_\omega(x)\in\bar\Delta_{\sigma^n\omega 0}$}\}.
$$
 Note that
$\hat R_\omega(x)=\bar R_\omega(x)$ for all $x\in \bar\Delta_{\omega, 0}$, and
$$
\bar m_\omega\{\hat{R}_\omega>n\}=\sum_{\ell>n}\bar m_{\omega}(\bar\Delta_{\omega, \ell})=\sum_{\ell>n}\bar
m_{\sigma^{-\ell}\omega}\{\bar R_{\sigma^{-\ell}\omega}>\ell\}.
$$
Obviously, 
$\bar F_\omega$ satisfies uniform expansion and Markov property.  Let $J\bar F_\omega$ denote the Jacobian of $\bar F_\omega$ with respect to this measure $\bar m_\omega$.  The following lemma shows bounded distortion
\begin{lemma}\label{l.cfbeta} For all $k\ge1$ and all
 $x,y\in\bar\Delta_\omega$ which belong to a same element of $\bar\Q_{k-1}^\omega$
\begin{eqnarray*} \left| \frac{J\bar F^k_\omega(x)}{J\bar F^k_\omega(y)}-1\right| \leq C_F
\beta^{ \bar s(\bar F^k_\omega(x),\bar F^k_\omega(y))}.
\end{eqnarray*}
\end{lemma}

\begin{proof}
The first item of Lemma~\ref{lem:bbdR} implies that the Jacobian $J\bar F_\omega$ is well
defined with respect to $\bar m_\omega$ for almost every $\omega$. Lemma~\ref{lem:distFk} implies the desired estimates.  
\end{proof}

For   $\omega\in\Omega$ we introduce some  function spaces. First of all we consider 
\begin{equation*}\begin{aligned}
\F^+_\beta=\{\varphi_\omega:\bar\Delta_\omega \to \mathbb R\mid & \exists C_\varphi>0, \forall I_\omega\in \bar\Q_\omega, \,\,\, 
\text{either} \,\, \varphi_\omega|{I_\omega}\equiv 0 \\
&\text{or} \,\,\, \varphi_\omega | {I_\omega}> 0 \,\,\, \text{and} \,\,\, \left| \log \frac{\varphi_\omega(x)}{\varphi_\omega(y)}\right|\le C_\varphi \beta^{\bar s(x, y)}, \forall x, y\in  I_\omega\}.
\end{aligned}
\end{equation*}
Let  $\theta\in(0,1)$ and $\kappa_\omega: \Omega \to \mathbb R_+$ be a random variable with $\inf_{\Omega} \kappa_\omega >0$ such that
\begin{equation}\label{Komega}
P\{\omega\mid \kappa_\omega >n\} \le \theta^n.
\end{equation}
Define the space of random  bounded  functions as 
\begin{equation*}\begin{aligned}
\L^{\kappa_\omega}_\infty=\{\varphi_\omega:\bar\Delta_\omega \to \mathbb R\mid \exists C^\prime_\varphi>0,  \sup_{x\in\bar\Delta_\omega} |\varphi_\omega(x)| \le C_\varphi' \kappa_\omega \}
\end{aligned}
\end{equation*}
and a space of random Lipschitz  functions 
\begin{equation*}\begin{aligned}
\F^{\kappa_\omega}_\beta=\{\varphi_\omega\in \L^{\kappa_\omega}_\infty \mid \exists C_\varphi>0,   |\varphi_\omega(x)-\varphi _\omega(y)| \le C_\varphi \kappa_\omega \beta^{\bar s(x, y)}, \,\, \forall  x, y \in \bar\Delta_\omega\}.
\end{aligned}
\end{equation*}

Note that in our setting $\bar m_\omega(\bar\Delta_\omega)$ is uniformly bounded. To obtain mixing rates we need finer information on the tail of the return times.  Suppose that there exists a random variable $n_1:\Omega\to \NN$  and a decreasing sequence $\{u_n\}_{n\in \NN}$ such that 
\[\begin{cases}  \bar m_\omega\{\bar R>n\}\le u_n \text{ for all } n>n_1(\omega),\\
P\{n_1>n\}\le Ce^{-cn^\theta}.
\end{cases}
\] We have the following theorem.

\begin{theorem}\label{thm:QDC} Let $\bar F_\omega:\bar\Delta_\omega\to \bar\Delta_{\sigma\omega}$ be the quotient tower map. 
\begin{enumerate}
\item $\{\bar F_\omega\}$ admits an equivariant family of measures $\{\bar\nu_\omega\}$ such that  $d\bar\nu_\omega/d\bar m_\omega\in \mathcal F^+_\beta\cap \mathcal F^1_\beta$
 and $\frac{1}{C_0}\le d\bar\nu_\omega/d\bar m_\omega\le C_0$ for some $C_0>0$.
\item Let  be a probability measure $ \lambda_\omega$  on $\bar \Delta_\omega$  with $\varphi={d\lambda}/{dm_\omega}\in \mathcal F_\beta^{\kappa_\omega}\cap \F_\beta^+$.
\begin{enumerate}
\item  If $u_n\le  Ce^{-cn^{\theta}}$,  for some $C,c>0$,  $\theta \in(0, 1]$.  Then there is a random variable  $n_0:\O\to \NN$  such that 
\[\begin{cases}
 \mdlo{(\bar F^{n}_\omega)_\ast\lambda_\omega-\bar\nu_{\sigma^{n}\omega} }\le C' e^{-c'n^{\theta}} \text{ for all }n > n_0(\omega),\\
 P\{ n_0 > n \} \leq C'' n^{-b},
 \end{cases}
 \]
  for some $C', C'',c'>0$, $b>1$.  If $n_1(\omega)$ is uniformly bounded, then for all $n>0$
  $$\mdlo{(\bar F^{n}_\omega)_\ast\lambda_\omega-\bar\nu_{\sigma^{n}\omega} }\le C' e^{-c'n^{\theta}}.$$  
 
\item   If $u_n\le  Cn^{-a}$ and $\int \bar m_\omega\{\bar R_\omega>n \}dP(\omega)\le  Cn^{-a-1}$ for some $C>0$,  $a>1$. Then  there is a random variable  $n_0:\O\to \NN$  such that 
\[\begin{cases}
 \mdlo{(\bar F^{n}_\omega)_\ast\lambda_\omega-\bar\nu_{\sigma^n\omega}} \le C'n^{1+\eps-a} \text{ for all }n > n_0(\omega), \text{ and } \eps>0,\\
 P\{ n_0 > n \} \leq C'' e^{-cn^{\theta}},
 \end{cases}
 \]
  for some $C', C'',c'>0$, $b>1$.
  \end{enumerate}
  Moreover, $c'$ does not depend on $\varphi$ and $C'$ depends only on the Lipschitz constant of  $\varphi$.  Furthermore, the analogous estimates hold for  $\mdlo{(\bar F^{n}_{\sigma^{-n}\omega})_\ast\lambda_{\sigma^{-n}\omega}-\bar\nu_{\omega} }$ in each of the above cases respectively. 
  \end{enumerate}
\end{theorem}
The proof of Theorem \ref{thm:QDC} is given in Subsection~\ref{se.convequi} below. In fact, the proof of item~(1) is similar to the proof of Proposition \ref{exist.mu.induced} and Item~(2) essentially follows from results in \cite{BBR19,BBM02,BBM03, D15} after showing that the constants appearing in the results \cite{BBR19, D15} depend on the observables only through their Lipschitz constants.  

\subsubsection{Convergence to equilibrium}\label{se.convequi}
Now we prove item (2) of Theorem \ref{thm:QDC}. Let $\bar\Delta =\{ (\omega, x) | \omega \in \Omega,  x \in \bar\Delta_\omega \}$. 
Denote by $\bar\Delta \otimes_\omega \bar\Delta $ the relative product over $\Omega$, that is
$\bar\Delta \otimes_\omega \bar\Delta=\{ (\omega, x, x') | \omega \in \Omega, x, x' \in \bar\Delta_\omega\}$.
These are measurable subsets of the appropriate product spaces and naturally carry the measures $P \times \bar m_\omega$  and $P \times \bar m_\omega \times\bar m_\omega$ respectively. 
We can lift the tower map $\bar F$ to a product action on $\bar\Delta \otimes_\omega \bar\Delta$
with the property
$\bar F_\omega \times \bar F_\omega :\bar\Delta_\omega\times \bar\Delta_\omega\to \bar\Delta_{\sigma\omega}\times \bar\Delta_{\sigma\omega}$ by applying $\bar F$ in each of the $x,x'$ coordinates. With respect to this map, we define \emph{auxiliary stopping times} $\tau_1^{\omega} <\tau_2^{\omega} < ... $ to the base  as follows: let $\ell_0$ be a fixed large constant. For $(\omega, x, x') \in \bar\Delta \otimes_\omega \bar\Delta$ set 
\begin{align*}
&\tau_1^\omega(x, x')=\inf\{n\ge \ell_0 \mid  \bar F_\omega^{n}x\in\bar\Delta_{\sigma^n\omega, 0}\};
\\
&\tau_2^\omega(x, x')=\inf\{n\ge \tau_1^\omega(x, x')+\ell_0 \mid  \bar F_\omega^{n}x'\in\bar \Delta_{\sigma^n\omega, 0}\};
\\
&\tau_3^\omega(x, x')=\inf\{n\ge \tau_2^\omega(x, x')+\ell_0 \mid  \bar F_\omega^{n}x\in\bar\Delta_{\sigma^n\omega, 0}\};
\\
&\tau_4^\omega(x, x')=\inf\{n\ge \tau_3^\omega(x, x')+\ell_0 \mid  \bar F_\omega^{n}x'\in\bar \Delta_{\sigma^n\omega, 0}\};
\end{align*}
and so on, with the action alternating between $x$ and $x'$.  Notice that for odd $i$'s the first (resp. for even $i$'s the second) coordinate of $(\bar F_\omega\times \bar F_\omega)^{\tau_i^\omega}(x, x')$ makes a return to $\bar\Delta_{\sigma^{\tau_i^\omega}\omega, 0}$. Let $i\ge 2$ be the smallest integer such that $(\bar F_\omega\times\bar F_\omega)^{\tau_i^\omega}(x, x')\in \bar \Delta_{\sigma^{\tau_i^\omega}\omega, 0}\times  \bar \Delta_{\sigma^{\tau_i^\omega}\omega, 0}$. Then we define the \emph{stopping time} $T_\omega$ by
$$T_\omega(x, x')=\tau_i^\omega(x, x').$$

Next define a sequence of partitions $\xi_1^\omega \prec \xi_2^\omega \prec \xi_3^\omega \prec ... $  of $\bar\Delta_\omega\times\bar\Delta_\omega$  so that $\tau_i^\omega$ is constant on the elements of $\xi_j^\omega$ for all $i\le j,$ $i, j\in\mathbb N$.  Given a partition $\bar{\mathcal Q}$ of $\bar\Delta_\omega$ we write $\bar{\mathcal Q}(x)$ to denote the element of $\bar{\mathcal Q}$ containing
$x$.  With this convention, we let

$$
\xi_1^\omega(x, x')=\left(
 \bigvee_{k=0}^{\tau_{1}^\omega-1}\bar F^{-k}_{\omega}\bar\P_{\sigma^k\omega}
\right)(x)\times \bar\Delta_\omega.
$$

Letting $\pi:\bar\Delta_\omega\times\bar\Delta_\omega\to\bar\Delta_\omega$ be the projection to the first coordinate, we define

$$
\xi_2^\omega(x, x')=\pi\xi_1^\omega(x, x')\times \left(
 \bigvee_{k=0}^{\tau_{2}^\omega-1}\bar F^{-k}_\omega\bar\P_{\sigma^k\omega}
\right)(x').
$$
Let $\pi'$ be the projection onto the second coordinate. We define $\xi_3^\omega$ by refining the partition on the first coordinate, and so on. If $\xi_{2i}^\omega$ is defined then we define $\xi_{2i+1}^\omega$ by refining each element of $\xi_{2i}^\omega$ in the first coordinate so that $\tau_{2i+1}^\omega$ is constant on each element of $\xi_{2i+1}^\omega$. Similarly $\xi_{2i+2}^\omega$ is  defined by refining each element of $\xi_{2i+1}^\omega$ in the second coordinate so that $\tau_{2i+1}^\omega$ is constant on each new partition element. 
Now we define a partition $\hat\P_\omega$ of $\bar\Delta_\omega\times\bar \Delta_\omega$ such that $T_\omega$ is constant on its element.  For definiteness suppose that $i$ is even and choose $\Gamma\in\xi_i^\omega$ such that $T_\omega|_{\Gamma}>\tau_{i-1}^\omega$. By construction $\Gamma =A\times B$ such that $\bar F^{\tau_i^\omega}(B)=\bar\Delta_{\sigma^{\tau_i^\omega}, 0}$  and $\bar F^{\tau_i^\omega}A$ is spread around $\bar\Delta_{\sigma^{\tau_i^\omega}\omega}$. We refine $A$ into countably  many pieces and choose those parts which are mapped onto the corresponding base at time $\tau_{i}^\omega$. Note
that $\{T_\omega=\tau_i^\omega\}$ may not be measurable with respect to $\xi_{i}^\omega.$ However,  since $\tau_{i+1}^\omega\ge \ell_0+\tau_{i}^\omega$ and $\xi_{i+1}^\omega$ is defined by dividing $A$ into pieces where $\tau_{i+1}^\omega$ is constant, $\{T_\omega=\tau_i^\omega\}$ is measurable with respect to $\xi_{i+1}^\omega.$ 

\begin{lemma}\label{lem:tailtau} Let $\bar\lambda_\omega$ and $\bar\lambda'_\omega$ be two probability measures on $\{\bar\Delta_\omega\}$ with densities $\bar\vp, \bar\vp'\in\F_\beta^+\cap \L^{\kappa_\omega}_\infty$. 
 Let $\tilde \lambda = \bar\lambda_\omega \times \bar\lambda_\omega^\prime$. 
\itemize
\item [1.] For each $\omega$, for each  $i\ge 2 $ and $\Gamma \in \xi_i^\omega$ such that $T_\omega |_\Gamma > \tau^\omega_{i-1}$ we have
\begin{equation*}
\tilde \lambda\{ T_\omega = \tau^\omega_i |  \Gamma\} \ge C_{\tilde\lambda} V_{\sigma^{\tau^\omega_{i-1}}\omega}^{\tau^\omega_i - \tau^\omega_{i-1}}.
\end{equation*}
\item [2.] For each $\omega$, for each  $i$ and $\Gamma \in \xi_i^\omega$ 
\begin{equation*}
\tilde \lambda\{ \tau_{i+1}^\omega - \tau_i^\omega > \ell_0 + n | \Gamma\} 
\leq M_0 MC_{\tilde \lambda}^{-1} \cdot m\{ \hat R_{\sigma^{\tom_i + \ell_0}\omega} > n \},
\end{equation*}
\enditemize
where $0<C_{\bl}<1$, which only depends on the Lipschitz constant of $\bar\vp,\bar\vp'$. We can fix $C_{\tilde\lambda}= \frac{2D+1}{2(D+1)^2}$, independent of $\bl$, for all $i$ sufficiently large, i.e. $i\ge i_0(\bl)$, which only depends on the Lipschitz constant of $\bar\vp,\bar\vp'$.
\end{lemma}
The above lemma shows the constant appearing in the estimates in  \cite{BBR19} (see Lemmas 7.3 and 7.4 in \cite{BBR19}) depend on the observables $\bar\vp,\bar\vp'$ only through their Lipschitz constants and consequently estimates of $\bar m_\omega\times\bar m_\omega\{T_\omega>n\}$ has the same property. 

 We now consider $\hat F_\omega=(\bar F_\omega\times \bar F_\omega)^{T_\omega}$ which is a mapping from  $\hat\Delta_\omega=\bar\Delta_\omega\times \bar\Delta_\omega$ into $\hat\Delta_{\sigma^{T_\omega}\omega}$.  
Let $\hat\xi_{1}^\omega$ be the partition of $\hat\Delta_\omega$ on which $T_{\omega}$ is constant. 
 Let $T_{1, \omega}<T_{2, \omega} \dots$ be stopping times  on $\hat\Delta_\omega$  defined as
$$
T_{1, \omega}=T_\omega, \quad T_{n, \omega}=T_{n-1, \omega}+ T_{\sigma^{T_{n-1, \omega}}\omega}\circ \hat F^{n-1}_\omega.
$$
Further as in \cite{BBR19} we obtain estimates  for $\bar m_\omega\times\bar m_\omega\{T_{i, \omega}>n\}$ of the same type as for  $\bar m_\omega\times\bar m_\omega\{R_\omega>n\}$ for all $i, n\ge 1$. In particular, if $\bar m_\omega\times\bar m_\omega\{R_\omega>n\}$ is uniform with respect to $\omega$  then  $\bar m_\omega\times\bar m_\omega\{T_{i, \omega}>n\}$ is uniform.  Thus the following lemma  completes the proof of Theorem \ref{thm:QDC}. 

\begin{lemma}\label{lem:main}
Let $\bl=\lambda_\omega\times \lambda_\omega'$. There exists $0<\eps_1<1$ independent of $d\bl/(d\bar m_\omega\times d\bar m_\omega)$ and a $C>0$ such that for almost every $\omega$ and all $n\in \mathbb N$ 
$$
|(\bar F_\omega^n)_*(\lambda_\omega)-(\bar F_\omega^n)_*(\lambda'_\omega)|\le C\sum_{i=0}^\infty \eps_1^{i}\bl\{T_{i, \omega}\le n <T_{i+1, \omega}\}.
$$
\end{lemma}

\subsection{Decay of  correlations}
In this subsection we prove Theorems \ref{exrates} and \ref{polyrates} by establishing a relation between decay of future correlations for the original dynamics and that of quotient dynamics.  Recall that for $\varphi, \psi:M\to \RR$ we have
\[
\C_{n}(\varphi, \psi; \mu_\omega)=
\int(\varphi\circ f^n_\omega)\psi d\mu_\omega -\int\varphi d\mu_{\sigma^n\omega}\int\psi  d\mu_\omega.
\]
Let $\pi_\omega:\Dom\to M$,  be the tower projection. Let $\bar\pi_\omega:\Dom\to \bar\Delta_\omega$ be the projection to the corresponding quotient tower. Then we have $\mu_\omega=(\pi_\omega)_\ast\nu_\omega$ and $\bar\nu_\omega=(\bar\pi_\omega)_\ast\nu_\omega$,  where $\{\nu_\omega\}$ is the equivariant family of physical measures for hyperbolic tower, $\{\mu_\omega\}$ is the equivariant family of physical measures of the original dynamics, $\{\bar\nu_\omega\}$ is the equivariant family of absolutely continuous measures for the quotient tower. Notice that $\{\bar\nu_\omega\}$ are the same as the measures constructed in Theorem \ref{thm:QDC} since the conditional measures of  $\nu_\omega$ are equivalent to $m_{\gamma_\omega}$. For $\varphi, \psi\in C^\eta(M)$ define $\tilde\psi_\omega=\psi\circ\pi_\omega$ and $\tilde\varphi_\omega=\varphi\circ\pi_\omega$. Then we have 
\begin{equation}\label{corduction}
\begin{aligned}
&\C_{n}(\varphi, \psi; \mu_\omega)=
\int(\varphi\circ f^n_\omega)\psi d\mu_\omega -\int\varphi d\mu_{\sigma^n\omega}\int\psi  d\mu_\omega=\\
&\int(\tilde\varphi_{\sigma^n\omega}\circ F^n_\omega)\tilde\psi_\omega d\nu_\omega-\int\tilde\varphi_{\sigma^n\omega} d\nu_{\sigma^n\omega}\int\tilde\psi_\omega d\nu_\omega = \C_{n}(\tilde\varphi_\omega, \tilde\psi_\omega; \nu_\omega).
\end{aligned}
\end{equation}
Fix some integer $k=[n/4]$ and define discretizations   $\bar\varphi_{\omega, k}$ of $\tilde\varphi_\omega$ on $\Dom$ as follows 
$$
\bar\varphi_{\omega, k}|_A=\inf\{\tilde\varphi_{\sigma^k\omega}\circ F^k_\omega(x)\mid x\in A\in \mathcal Q_{2k}^\omega\}.
$$
 $\bar\psi_{\omega, k}$ is defined similarly. The main result of this section is the following proposition.
\begin{proposition} For any $n\in \NN$, $\varphi, \psi\in C^\eta(M)$ 
\[
|\C_{n}(\varphi, \psi; \mu_\omega)|\le C_{\varphi, \psi}\max\{\delta_{\sigma^{[n/4]}\omega, [n/4]}^\eta, \C_{n}(\bar\varphi_{\sigma^{n-k}\omega, k}, \bar\psi_{\sigma^{-k}\omega, k}; \bar\nu_{\sigma^{-k}\omega})\}
\]
for some constant  $C_{\varphi, \psi}>0$ and for a.e. $\omega \in \Omega$.
\end{proposition}
\begin{proof}
By \eqref{corduction} it is sufficient to estimate $\C_{n}(\tilde\varphi_\omega, \tilde\psi_\omega; \nu_\omega)$. We will estimate this in several steps. 
Obviously, for any $x\in A$, $A\in \mathcal Q_{2k}^\omega$, 
\begin{equation}\label{eq:estim}
|\tilde\varphi_{\sigma^k\omega}\circ F^k_\omega(x)-\bar\varphi_{\omega, k}(x)|\le {\|\varphi\|}_{C^\eta}\diam(\pi_{\sigma^k\omega}(F^k_\omega(A)))^\eta.
\end{equation}
Recall that 
\[
\delta_{\omega, k}=\sup_{A\in \mathcal Q_{2k}^\omega}\diam(\pi_{\sigma^k\omega}(F^k_\omega(A))).
\]
Notice that  $\C_{n}(\tilde\varphi_\omega, \tilde\psi_\omega; \nu_\omega)=\C_{n-k}(\tilde\varphi_{\sigma^n\omega}\circ F^k_{\sigma^{n-k}\omega}, \tilde\psi_\omega; \nu_\omega)$. Indeed, by equivariance of $\nu_\omega$ we have
\begin{align*}
&\int (\tilde\varphi_{\sigma^n\omega}\circ F^n_{\omega}) \tilde\psi_\omega  d\nu_\omega-\int\bar\varphi_{\sigma^{n}\omega, k}d\nu_{\sigma^n\omega}\int\tilde\psi_\omega d\nu_\omega \\
&= \int (\tilde\varphi_{\sigma^n\omega}\circ F^k_{\sigma^{n-k}\omega})\circ F^{n-k}_\omega \tilde\psi_\omega  d\nu_\omega-\int\bar\varphi_{\sigma^{n}\omega, k}\circ F^k_{\sigma^{n-k}\omega}d\nu_{\sigma^{n-k}\omega}\int\tilde\psi_\omega d\nu_\omega.
\end{align*}
We have 
\begin{equation}\label{step1}
\begin{aligned}
&|\C_{n}(\tilde\varphi_\omega, \tilde\psi_\omega; \nu_\omega)-\C_{n-k}(\bar\varphi_{\sigma^{n-k}\omega}, \tilde\psi_\omega; \nu_\omega)| \\
\le&\left|\int (\tilde\varphi_{\sigma^n\omega}\circ F^k_{\sigma^{n-k}\omega}-\bar\varphi_{\sigma^{n-k}\omega, k})\circ F^{n-k}_{\omega}\tilde\psi_\omega d\nu_\omega\right|\\
+& \left|\int (\tilde\varphi_{\sigma^n\omega}\circ F^k_{\sigma^{n-k}\omega}-\bar\varphi_{\sigma^{n-k}\omega, k})d\nu_\omega\int\tilde\psi_\omega d\nu_\omega\right|\\
&\le 2{\|\varphi\|}_{C^\eta}{\|\psi\|}_{\infty}\delta_{\sigma^{n-k}\omega, k}^\eta
\end{aligned}
\end{equation}

Denote by $\bar\psi_{\omega, k}\nu_\omega$ the signed measure whose density with respect to $\nu_\omega$ is $\bar\psi_{\omega, k}$. Let $\hat \psi _{\omega, k}$ denote the density of $(F_{\sigma^{-k}\omega}^k)_\ast(\bar\psi_{\sigma^{-k}\omega, k}\nu_{\sigma^{-k}\omega})$. Then we have 
\begin{equation}\label{step2}
|\C_{n-k}(\bar\varphi_{\sigma^{n-k}\omega, k}, \tilde\psi_\omega; \nu_\omega)-\C_{n-k}(\bar\varphi_{\sigma^{n-k}\omega, k}, \hat\psi_{\omega, k}; \nu_\omega)| \le 2{\|\varphi\|}_{\infty} {\|\psi\|}_{C^\eta}\delta_{\sigma^{-k}\omega, k}^\eta.
\end{equation}
Indeed, the left hand side of \eqref{step2} is
\begin{align*}
&\le \left|\int \varphi_{\sigma^{n-k}\omega, k}\circ F_{\omega}^{n-k}(\tilde\psi_\omega-\hat\psi_{\omega, k})d\nu_{\omega} \right|\\
+
\left|\int \varphi_{\sigma^{n-k}\omega, k}d\nu_\omega\right.& \left. \int (\tilde\psi_\omega-\hat\psi_{\omega, k})d\nu_{\omega} \right|\le 
2{\|\varphi\|}_{\infty}\left|\int (\tilde\psi_\omega-\hat\psi_{\omega, k})d\nu_{\omega} \right|.
\end{align*}
Observing  that 
$$ 
\tilde\psi_{\omega}\nu_{\omega}=(F^k_{\sigma^{-k}\omega})_\ast(\tilde\psi_{\omega}\circ F^k_{\sigma^{-k}\omega})\nu_{\sigma^{-k}\omega}.
$$
and letting $|\cdot|$ denote the total  variation of a signed measure 
we have 
\[
\begin{aligned}
&\left|\int (\tilde\psi_\omega-\hat\psi_{\omega, k})d\nu_{\omega} \right|=|\tilde\psi_\omega\nu_\omega -\hat\psi_{\omega, k}\nu_\omega|
\\&=
|(F^k_{\sigma^{-k}\omega})_\ast(\tilde\psi_{\omega}\circ F^k_{\sigma^{-k}\omega})\nu_{\sigma^{-k}\omega}
- (F^k_{\sigma^{-k}\omega})_\ast\bar\psi_{\sigma^{-k}\omega, k}\nu_{\sigma^{-k}\omega}|\\
&=\int|\tilde\psi_{\omega}\circ F^k_{\sigma^{-k}\omega}-\bar\psi_{\sigma^{-k}\omega, k}|d\nu_{\sigma^{-k}\omega}
\le {\|\psi\|}_{C^\eta}\delta_{\sigma^{-k}\omega, k}^\eta.
\end{aligned}
\] This shows \eqref{step2}.
Next we observe that
\begin{equation}\label{step3}
\C_{n-k}(\bar\varphi_{\sigma^{n-k}\omega, k}, \hat\psi_{\omega, k}; \nu_\omega)=
\C_{n}(\bar\varphi_{\sigma^{n-k}\omega, k}, \bar\psi_{\sigma^{-k}\omega, k}; \bar\nu_{\sigma^{-k}\omega}).
\end{equation}

Indeed, 
\begin{align*}
\int \bar\varphi_{\sigma^{n-k}\omega, k}\circ F^{n-k}_\omega\hat \psi_{\omega, k}d\nu_\omega=
\int \bar\varphi_{\sigma^{n-k}\omega, k}\circ F^{n-k}_\omega d(\hat\psi_{\omega, k}\nu_\omega) \\
=\int\bar\varphi_{\sigma^{n-k}\omega, k} d( (F^{n-k}_\omega)_\ast(\hat\psi_{\omega, k}\nu_\omega))=
\int \bar\varphi_{\sigma^{n-k}\omega, k} d( (F^{n}_{\sigma^{-k}\omega})_\ast(\bar\psi_{\sigma^{-k}\omega, k}\nu_{\sigma^{-k}\omega}))\\
=\int  \bar\varphi_{\sigma^{n-k}\omega, k}\circ F^{n}_{\sigma^{-k}\omega}\bar\psi_{\sigma^{-k}\omega, k}d\nu_{\sigma^{-k}\omega}
=\int  \bar\varphi_{\sigma^{n-k}\omega, k}\circ \bar F^{n}_{\sigma^{-k}\omega}\bar\psi_{\sigma^{-k}\omega, k}d\bar\nu_{\sigma^{-k}\omega}.
\end{align*}

Also, 
\begin{align*}
\int \bar\varphi_{\sigma^{n-k}\omega, k}d\nu_{\sigma^{n-k}\omega}\int\hat \psi_{\omega, k}d\nu_\omega=
\int \bar\varphi_{\sigma^{n-k}\omega, k}d\nu_{\sigma^{n-k}\omega}\int d (F^k_{\sigma^{-k}\omega})_\ast \bar\psi_{\sigma^{-k}\omega, k}\nu_{\sigma^{-k}\omega}\\
= \int \bar\varphi_{\sigma^{n-k}\omega, k}d\bar\nu_{\sigma^{n-k}\omega}\int \bar\psi_{\sigma^{-k}\omega, k}\bar\nu_{\sigma^{-k}\omega}.
\end{align*}

Finally we show that $\bar\varphi_{\sigma^{n-k}\omega, k}$ and $\bar\psi_{\sigma^{-k}\omega, k}$ belong to right  functions spaces and thus the results for the quotient tower are applicable.

Suppose that $\psi$ is not identically zero. Fix $k\le n/4$ and let 
$$
\Psi_{\omega, k}=b_{\omega, k}(\bar\psi_{\omega, k}+2{\|\bar\psi_{\omega, k}\|}_{\infty}),
$$
where $(3{\|\psi\|}_{\infty})^{-1}\le b_{\omega, k}\le {\|\psi\|}_{\infty}^{-1}$ is chosen so that $\int \Psi_{\omega, k}d\bar\nu_\omega=1$. Then $1\le {\|\Psi_{\omega, k}\|}_{\infty}\le 3$.  Let  $\bar m_\omega$ denote the reference measure on quotient tower $\bar\Delta_\omega$, and let 
$$
\rho_\omega=\frac{d\bar\nu_\omega}{d\bar m_\omega}, \quad d\hat\lambda_{\omega, k}=\Psi_{\omega, k}\rho_\omega d\bar m_\omega.
$$
Recalling that $\int\Psi_{\sigma^{-k}\omega, k} d\nu_{\sigma^{-k}\omega}=1$ we have  
\begin{equation*}
\begin{split}
&\left|\int  \bar\varphi_{\sigma^{n-k}\omega, k}\circ \bar F^{n}_{\sigma^{-k}\omega}\bar\psi_{\sigma^{-k}\omega, k}d\bar\nu_{\sigma^{-k}\omega} - \int \bar\varphi_{\sigma^{n-k}\omega, k}d\bar\nu_{\sigma^{n-k}\omega}\int \bar\psi_{\sigma^{-k}\omega, k}d\bar\nu_{\sigma^{-k}\omega}\right|\\
&=\frac{1}{b_{\sigma^{-k}\omega, k}}\left|\int  \bar\varphi_{\sigma^{n-k}\omega, k}\circ \bar F^{n}_{\sigma^{-k}\omega} \Psi_{\sigma^{-k}\omega, k}d\nu_{\sigma^{-k}\omega} 
-  \int \bar\varphi_{\sigma^{n-k}\omega, k}d\bar\nu_{\sigma^{n-k}\omega}\right|\\
&=\frac{1}{b_{\sigma^{-k}\omega, k}}\left|\int  \bar\varphi_{\sigma^{n-k}\omega, k}d (\bar F^{n}_{\sigma^{-k}\omega})_\ast \hat\lambda_{\sigma^{-k}\omega} 
-  \int \bar\varphi_{\sigma^{n-k}\omega, k}d\bar\nu_{\sigma^{n-k}\omega}\right|\\
&\le 3{\|\varphi\|}_{\infty}{\|\psi\|}_{\infty} 
\int \left|\frac{d(\bar F^n_{\sigma^{-k}\omega})_\ast \hat \lambda_{\sigma^{-k}\omega, k}}{d\bar m_{\sigma^{n-k}}}-\rho_{\sigma^{n-k}\omega}\right| d\bar m_{\sigma^{n-k}\omega}.
\end{split}
\end{equation*}
Letting   $\bar\lambda_{\sigma^k\omega, k}=(\bar F_{\sigma^{-k}\omega}^{2k})_\ast\hat\lambda_{\sigma^{-k}\omega, k}$,  the above equality implies  
\begin{equation}
\C_{n}(\bar\varphi_{\sigma^{n-k}\omega, k}, \bar\psi_{\sigma^{-k}\omega, k}; \bar\nu_{\sigma^{-k}\omega}) \le 
3{\|\varphi\|}_{\infty}{\|\psi\|}_{\infty} |(\bar F_{\sigma^{k}\omega}^{2k})_\ast\bar\lambda_{\sigma^k\omega, k}-\bar\nu_{\sigma^{n-k}\omega}|.
\end{equation}

Let $\phi_{\omega, k}$  denote the density of the measure  $\bar\lambda_{\omega, k}$
with respect to $\bar m_\omega$. The next lemma shows that
$\phi_{\omega, k}\in\mathcal F_\beta^+$, with the constant $C$ independent of $\omega$ and $k$. This is enough for using 
Theorem~\ref{thm:QDC} and concluding the proof of the proposition.  
\end{proof}
\begin{lemma}\label{Liberta}
There is $C>0$, independent of $\omega$ and $k$, and a random variable $\kappa_\omega:\Omega\to \RR_+$ such that
$$
|\phi_{\omega, k}(\bar x)-\phi_{\omega, k}(\bar y)|\le C\kappa_\omega\beta^{\bar s(\bar x,\bar y)}, \quad \text{for all } \bar x,\bar y \in \bar\Delta_\omega.
$$
\end{lemma}
\begin{proof} Since $(\bar F^{2k}_\omega)_*\bar\nu_\omega=\bar\nu_{\sigma^{2k}\omega}$  and
$\rho_\omega=d\bar \nu_\omega/d\bar m_\omega$,  we write 
\begin{equation}\label{eq.ro}
\rho_{\omega}(\bar x)=\sum_{Q\in \bar \Q_{2k}^{\omega}}
\frac{\rho_{\sigma^{-2k}\omega}\big((\bar F^{2k}_{\sigma^{-2k}\omega} \vert Q)^{-1}(\bar x)\big)}{J\bar F^{2k}_{\sigma^{-2k}\omega}\big((\bar F^{2k}_{\sigma^{-2k}\omega}\vert Q)^{-1}(\bar x)\big)}.
\end{equation}
By definition we have 
 $$\phi_{\omega, k}=\frac {d\bar\lambda_{\omega, k}}{dm}=\frac d{d\bar m_\omega} 
 (\bar F^{2k}_{\sigma^{-2k}\omega})_\ast\hat\lambda_{\sigma^{-2k}\omega, k}  \text{ and }\frac {d\hat\lambda_{\sigma^{-2k}\omega, k}}{dm_\omega}=\Psi_{\sigma^{-2k}\omega, k}\rho_{\sigma^{-2k}\omega}.$$
 Since $\Psi_{\omega, k}$ is constant on elements of $\Q_{2k}^\omega$, letting $\Psi_{\sigma^{-2k}\omega}\big((\bar F^{2k}_{\sigma^{-2k}\omega} \vert Q)^{-1}(\bar x)\big)=c_{\omega, Q}$   we have
 \begin{equation*}
\phi_{\omega, k}(\bar x)=\sum_{Q\in \bar \Q_{2k}^{\omega}}
c_{\omega, Q}\cdot
\frac{\rho_{\sigma^{-2k}\omega}\big((\bar F^{2k}_{\sigma^{-2k}\omega} \vert Q)^{-1}(\bar x)\big)}{J\bar F^{2k}_{\sigma^{-2k}\omega}\big((\bar F^{2k}_{\sigma^{-2k}\omega}\vert Q)^{-1}(\bar x)\big)}.
\end{equation*}
  Hence,
 \begin{equation}\label{figa}
 \begin{aligned}
\phi_k(\bar x)-\phi_k(\bar y) &= 
\\  \sum_{Q\in \bar\Q_{2k}^\omega}c_{\omega, Q}
&\left(
\frac{\rho_{\sigma^{-2k}\omega}\big((\bar F^{2k}_{\sigma^{-2k}\omega} \vert Q)^{-1}(\bar x)\big)}
{J\bar F^{2k}_{\sigma^{-2k}\omega}\big((\bar F^{2k}_{\sigma^{-2k}\omega}\vert Q)^{-1}(\bar x)\big)}
 -
\frac{\rho_{\sigma^{-2k}\omega}\big((\bar F^{2k}_{\sigma^{-2k}\omega} \vert Q)^{-1}(\bar y)\big)}{J\bar F^{2k}_{\sigma^{-2k}\omega}\big((\bar F^{2k}_{\sigma^{-2k}\omega}\vert Q)^{-1}(\bar y)\big)}
 \right).
 \end{aligned}
\end{equation}
Let $\omega'=\sigma^{-2k}\omega$. For  $Q\in\bar\Q_{2k}^\omega$,let $\bar x',\bar y',\in Q$ be such
$\bar F^{2k}_{\omega'}(\bar x')=\bar x$ and  $\bar F^{2k}_{\omega'}(\bar x')=\bar x$.
We have
 \begin{equation}\label{eq.outra}
\frac{\rho_{\omega'}(\bar x')} {J\bar F^{2k}_{\omega'}(\bar x')}
 -\frac{\rho_{\omega'}(\bar y')} {J\bar F^{2k}_{\omega'}(\bar y')}
 =
 \left(\frac{\rho_{\omega'}(\bar y')} {J\bar F^{2k}_{\omega'}(\bar y')}\right)
 \left( \frac{\rho_{\omega'}(\bar x')} { \rho_{\omega'}(\bar y')}\frac{J\bar F^{2k}_{\omega'}(\bar y')}
 {J\bar F^{2k}_{\omega'}(\bar x')} -1 \right).
\end{equation}
It follows from Theorem~\ref{existence} that there is $C_{\rho}>0$ independent of $\omega$ such that 
 $$
 \left| \frac{\rho_{\omega'}(\bar x')} { \rho_{\omega'}(\bar y')}-1\right|\le C_{\rho}\beta^{s(\bar x',\bar y')}.
 $$
On the other hand, by Lemma~\ref{l.cfbeta} there is $D>0$
such that independent of $\omega$ such that 
$$ \left| \frac{J\bar F^{2k}_{\omega'}(\bar y')} {J\bar F^{2k}_{\omega'}(\bar x')}-1\right| \leq D\beta^{ \bar s(\bar
F^{2k}_{\omega'}(\bar x'),F^{2k}_{\omega'}(\bar y'))}.
$$
Since $s(\bar x',\bar y')\ge s(\bar F^{2k}_{\omega'}(\bar x'),F^{2k}_{\omega'}(\bar
y'))=s(\bar x,\bar y)$, we have
 \begin{equation}\label{label.eq}
\log \left| \frac{\rho(\bar x')} { \rho(\bar y')}\frac{J\bar F^{2k}(\bar y')} {J\bar F^{2k}(\bar x')} \right|
 \le
 (C_{\rho}+D)\beta^{ \bar s(\bar x,\bar y)}.
\end{equation}
Recalling \eqref{eq.ro} and the fact that $|c_Q|\le
\|\hat\Psi_{\omega, k}\|_\infty\le 3$, it follows from \eqref{figa},
\eqref{eq.outra} and \eqref{label.eq} that there is some constant
$C>0$ not depending on $\phi_k$ such that
 \begin{eqnarray*}
\mdlo{\phi_{\omega, k}(\bar x)-\phi_{\omega, k}(\bar y)} &\le & C\kappa_\omega
 \beta^{ \bar s(\bar x,\bar y)}.
\end{eqnarray*}
\end{proof}
Lemma \ref{Liberta} shows that the Lipschitz constant of $\phi_{\omega, k}\in\mathcal F_\beta^+$ is independent of $\omega$ and $k$. Thus, the rate of decay of
$|(\bar F_{\sigma^{k}\omega}^{2k})_\ast\bar\lambda_{\sigma^k\omega, k}-\bar\nu_{\sigma^{n-k}\omega}|$ can be estimated as in Theorem \ref{thm:QDC} with uniform constants for all $0\le k\le n/4$. This completes the proofs of Theorems \ref{exrates} and \ref{polyrates}.

\section{Partially hyperbolic attractors}\label{se.PH}

In this section we prove Theorem~\ref{randomPH}.
%
%
%
%


\subsection{Preliminary results }
\label{sec:bounded-curvature}  Here we present
the ``random version'' of the main results in \cite[Section
2]{ABV00}, where it was  shown 
that $f$ satisfies a bounded curvature property over disks
having the tangent space at each point contained in a cone
field around the centre-unstable direction.

First of all we fix continuous extensions (not necessarily invariant under $Df$) of the two subbundles $E^{s}$ and
$E^{cu}$ to the trapping region  $U\supset K$.
  Given $0<a<1$, we define the \emph{centre-unstable cone
field $C^{cu}_a=(C^{cu}_a(x))_{x\in U}$} by
\begin{equation}
  \label{eq:cu-cone-field}
  C^{cu}_a(x) = \{
v_1+v_2\in E^{s}_x\oplus E^{cu}_x \colon \| v_1 \| \le a
\| v_2 \| \}
\end{equation}
and the \emph{stable cone field} $C^{s}_a=(C^{s}_a(x))_{x\in U}$ similarly, reversing the roles of the bundles
in~(\ref{eq:cu-cone-field}).
Slightly increasing $\lambda<1$, if necessary, we may fix $a$  
 so that  dominated
decomposition extends to  the cone fields for
all maps nearby $f$, i.e.
\begin{equation}
  \label{eq:domination-cone-fields}
  \| Df_\omega(x) v^{s} \|\cdot\| Df_{\omega}^{-1}(f_\omega (x)) v^{cu}\|
 \le\lambda \|v^{s}
  \|\cdot \|v^{cu}\|
\end{equation}
for all $v^{s}\in C^{s}_a(x)$, $v^{cu}\in C^{cu}_a(f_\omega(
x))$, $x\in U$ and $\omega\in\Omega$. Moreover, the
domination property   implies that for $f_\omega$ sufficiently close to $f$ we have 

\begin{equation}\label{invcones}
 Df_\omega^{-1} C_{a}^{s}(x)\subset C_{\lambda a}^{s}(f_\omega^{-1}(x))\qand  Df_\omega C_a^{cu}(x) \subset C_{\lambda a}^{cu}(f_\omega
(x)).
\end{equation}
Given $x\in K_\omega= \cap_{n\ge0}f^n_{ \sigma^{-n}\omega}(U)$, set
$$
 E^{s}_\omega(x)=\bigcap_{n\ge0}Df^{-n}_{ \omega}(f_{ \omega}^{n}(x))(C^s_a(f_{ \omega}^{n}(x)))
 $$
and
 $$E^{cu}_\omega(x)=\bigcap_{n\ge0}Df^n_{\sigma^{-n}\omega}(f_{\sigma^{-n}\omega}^{-n}(x))(C^{cu}_a(f_{\sigma^{-n}\omega}^{-n}(x))).
$$
It is easily verified that
 $$
Df_\omega(x)E^{s}_\omega(x)=E^{s}_{\sigma\omega}(f_\omega (x))\qand Df_\omega(x)E^{cu}_\omega(x)=E^{cu}_{\sigma\omega}(f_\omega (x)).
$$ 
Standard arguments give that that $E^{s}_\omega(x)$ is a subspace of the same dimension of $E^{s}(x)$, $E^{cu}_\omega(x)$ is a subspace of the same   dimension of $E^{cu}(x)$ and $T_xM=E^{s}_\omega(x)\oplus E^{cu}_\omega(x)$.

We say that an embedded sub-manifold $S\subset U$ is
\emph{tangent to the centre-unstable cone field} if $T_x S\subset
C^{cu}_a(x)$ for all $x\in S$. Note that as we assume contraction in the $E^s$ direction,  the local unstable manifolds are necessarily tangent to the centre-unstable cone field.


The \emph{curvature} of  local unstable manifolds and their iterates will be approximated in local
coordinates by the notion of H\"older variation of the tangent
bundle as follows.
Let  $\varepsilon>0$ be sufficiently small so that if  
$V_x=B(x,\varepsilon)$, then the exponential map $\exp_x:V_x\to
T_x M$ is a diffeomorphism onto its image for all $x\in M$.
We are going to identify $V_x$ through the local chart
$\exp_x^{-1}$ with the neighborhood $U_x=\exp_x V_x$ of the
origin in $T_x M$. Identifying $x$ with the origin in $T_x
M$ and reducing $\varepsilon$, if necessary, we get that $E_x^{cu}$ is contained in $C^{cu}_a(y)$ for
all $y\in U_x$. Then the
intersection of $E^{s}_x$ with $C^{cu}_a(y)$ is the zero
vector. So if $x\in S$ then $T_y S$ is the graph of a
linear map $A_x(y):E_x^{cu}\to E_x^{s}$ for $y\in U_x\cap
S$.
Given  $C>0$ and $\zeta\in (0,1)$, we say that the tangent bundle
of
  $S$ is \emph{$(C,\zeta)$-H\"older} if
\begin{equation}
  \label{eq:holder-bundle}
  \| A_x(y) \| \le C \dist_S(x,y)^\zeta, \quad\mbox{for all $
y\in U_x\cap S $  and all $ x\in U$,}
\end{equation}
where $\dist_S(x,y)$ is the distance measured along $S$.
Up to choosing  smaller $a>0$  we may assume
that there are $\lambda<\lambda_1<1$ and $0<\zeta<1$ such that for
all norm one vectors $v^{s}\in C^{s}_a(x), v^{cu}\in
C^{cu}_a(x)$, $x\in U$ it holds $$ \| Df_\omega(x) v^{s}
\|\cdot\| Df_{\omega}^{-1}(f_\omega( x)) v^{cu}\|^{1+\zeta}
 \le\lambda_1.
$$ For these  values of $\lambda_1$ and $\zeta$, given a $C^1$
submanifold $S\subset U$ tangent to the centre-unstable
cone field we define
\begin{equation}
  \label{eq:curvature-definition}
  \kappa(S)=\inf\{
C>0 : TS \mbox{ is }(C,\zeta)\mbox{-H\"older} \}.
\end{equation}
The proofs of the conclusions in the next proposition  may be
obtained by mimicking the proofs of the corresponding ones
in \cite[Proposition 2.2 \& Corollary 2.4]{ABV00}.
Actually, the proof of  Proposition 2.2  in \cite{ABV00} holds in a more general setting of submanifolds tangent to the centre-unstable cone fields, than just local unstable manifolds, but here we do not need it in its full generality.
The main ingredients  are the cone
invariance and dominated decomposition properties that we
have already extended for nearby perturbations $f_\omega$ of the
diffeomorphism $f$; see also~\cite[Section 4]{AAV07} in the random setting.

\begin{proposition}\label{pr:bounded-curvature}
There is $C_1>0$ such that
for every local unstable manifold $\Sigma\subset  U$ 
and every $\omega\in \Omega$
\begin{enumerate}
\item there exists $n_1$ such that $\kappa(f^n_{\omega} (\Sigma))\le
  C_1$ for all $n\ge n_1$ with $f^k_{\omega}(\Sigma)\subset U$ for
  all $1\le k\le n$;
\item if $\kappa(S)\le C_1$ then $\kappa(f^n_{\omega} (\Sigma))\le
  C_1$ for all $n\ge1$ such that $f^k_{\omega}(\Sigma)\subset U$ for
  all $1\le k\le n$;
\item in particular, if $\Sigma$ is as in the previous item, then for every $n\ge1$ we have
$$ J_n: f^n_{\omega} (\Sigma) \ni x \longmapsto \log | \det (Df_{\sigma^n\omega}\vert
T_x f^n_{\omega}(\Sigma)) | $$ is $(L_1,\zeta)$-H\"older
continuous with $L_1>0$ depending only on $C_1$ and $f$.
\end{enumerate}
\end{proposition}




From the condition of non-uniform expansion along the
centre-unstable direction we will be able to deduce some uniform
expansion at certain times which are precisely defined through the
following notion.
  Given $0<\alpha<1$, we say that
$n\ge1$ is a  $\alpha$-\emph{hyperbolic time} for $(\omega,x)\in
\Omega\times U$ if $$ \prod_{j=n-k+1}^n \| Df^{-1}_{\sigma^{j}\omega} \vert
E^{cu}_{f^j_{\omega} (x)} \| \le \alpha^k, \quad\mbox{for
all $ k=1,\dots,n.$} $$  

%
%
%
%
%
The condition of non-uniform expansion for random orbits along the
centre-unstable direction is enough to ensure that almost all points
have infinitely many $\alpha$-hyperbolic times   easily adapting the proof of \cite[Corollary
3.2]{ABV00},   based on an idea of Pliss~\cite{P72}.

\begin{proposition} \label{pr.hyptimes} There exist
$\rho,\alpha>0$ depending only on $f$ such that for
$P\times \leb$ almost all $(\omega,x)\in \Omega\times U$
and a sufficiently large  integer $n\ge1$, there exist $1\le
n_1<\dots<n_k\le n$, with $k\ge\rho n$, which are
 $\alpha$-hyperbolic times for $(\omega,x)$.  \end{proposition}

Let  $n$ be a  $\alpha$-hyperbolic time for $(\omega,x)\in
\Omega\times U$. This implies  that $Df^{-k}_{\sigma^{n-k}\omega}\vert
E^{cu}_{f^n_{\omega} (x)}$ is a contraction for all
$k=1,\ldots,n$. In addition, if $a>0$ and $\epsilon>0$ are taken
small enough in the definition of the cone fields and the
random perturbations, then taking $\delta_1>0$ also small, we
have by continuity
\begin{equation}\label{eq:spare-contraction}
  \| Df_\omega^{-1} \vert E^{cu}_{f_\omega (y)} \| \le \alpha^{-1/4}
  \| Df^{-1}_\omega \vert E^{cu}_{f_\omega(x)} \|,
\end{equation}
for all $\omega\in\Omega, x\in\overline{f (U)}$ and $y\in U$ with
$\dist(x,y)\le\delta_1$. 

In the sequel, we will refer to  a local unstable manifold of size $\delta_1$ simply as \emph{$cu$-disk} of a certain radius $r>0$. 
Using \eqref{eq:spare-contraction}, the next result can be
obtained by adapting the proof of \cite[Lemma 2.7]{ABV00}.

\begin{proposition}\label{pr.predisks} 
Given a local unstable manifold
$\Sigma\subset U$, a point  $x\in
\Sigma$ and $n\ge1$ a  $\alpha$-hyperbolic time for $(\omega,x)$,
 there exists a set $V_{\omega,n}(x)\subset \Sigma$ such that $f_\omega^n$ maps diffeomorphically onto a $cu$-disk of radius $\delta_1$ around $f^n_\omega(x)$. Moreover, 
$$
 \dist_{f^{n-k}_{\omega}(\Sigma)}( f^{n-k}_{\omega} (y)  ,
f^{n-k}_{\omega} (z) )\le \alpha^{k/2} \dist_{f^n_{\omega}( \Sigma)}(f^n
_{\omega}y, f^n_{\omega} z) ,
$$
for every $k=1,\dots,n$ and every   $y,z\in V_{\omega,n}(x)$. 
\end{proposition}

These sets $V_{\omega,n}(x)$ will be referred to as \emph{hyperbolic predisks}. Notice that their images $B^{cu}_{\delta_1} (f^n_\omega(x))$ are $cu$-disks of radius $\delta_1$.

\begin{remark}\label{re.bounded}
The proof of \cite[Lemma 2.7]{ABV00} works under the less restrictive assumption of having $\alpha^{-1/2}$ instead of  $\alpha^{-1/4}$ in~\eqref{eq:spare-contraction}. 
Also, that proof gives
 that $n$ is a $\alpha^{1/2}$-hyperbolic time for each $y\in V_{\omega,n}(x)$. Under~\eqref{eq:spare-contraction}, it is straightforward to check that $n$ is still an $\alpha^{1/4}$-hyperbolic time for any point $z\in U$ such that 
 $\dist(f^j_\omega(y),f^j_\omega(z))\le\delta_1$, for some $y\in V_{\omega,n}(x)$. 
%
%
\end{remark}

Using the previous lemma and the H\"older continuity   given
by Proposition~\ref{pr:bounded-curvature}  the following bounded
distortion result can be deduced as in \cite[Proposition~2.8]{ABV00}.

\begin{proposition} \label{p.distortion} There exists $C>1$ such that
given any $C^1$ $cu$-disk $\Sigma$   with $\kappa(\Sigma) \le C$, any
$x\in \Sigma$ and $n \ge 1$ a  $\alpha$-hyperbolic time for
$(\omega,x)$, then
$$
\log \frac{|\det
  Df^{n}_{\omega} \vert T_y \Sigma|} {|\det Df^{n}_{\omega}
  \vert T_z \Sigma|} \le C\dist_{f_\omega^{n }(\Sigma)}(f^n_\omega(y), f^n_\omega(z))^\eta.
$$
for every  $y,z\in V_{\omega,n}(x)$.
\end{proposition}

\subsection{Partition on a reference leaf}
\label{s.structure}

Here we find a $cu$-disk $\Sigma_\omega\subset D_\omega$ whose hyperbolic pre-disks contained in it return to a neighborhood of $\Sigma_\omega$ under forward iterations and their images project along stable leaves covering $\Sigma_\omega$ completely.
Then  we define a partition on~$\Sigma_\omega$ whose
construction is inspired essentially by \cite{ADL13,ADL17,ALP05,AP10}. In particular, we improve the product structure construction of \cite{ADL13}, which was performed for deterministic diffeomorphisms, to incorporate random perturbations in a partially hyperbolic setting. We then estimate the decay of return times adapting the ideas of~\cite{G06} to the random partially hyperbolic setting. 

%

\label{s.disk}

Let $D_\omega$ be a local unstable manifold as in  Theorem~\ref{randomPH}.  Given $\delta_1>0$ as in Proposition~\ref{pr.predisks}, we take $0<\delta_s<\delta_1/2$ such that points 
$x\in K_\omega$ have local stable manifolds $W_{\omega,\delta_s}^s(x)$ of size $\delta_s$. Notice that as we are assuming uniform contraction in the stable  direction $E^s$, these stable manifolds with uniform size exist; see e.g. \cite[Chapter~7]{A98}. Moreover, they depend continuously on $\omega$ in the $C^1$ topology. 
 Given any $cu$-disk $\Sigma_\omega\subset D_\omega$, we define the \emph{cylinder} over $\Sigma_\omega$
  $$\C (\Sigma_\omega)=\bigcup_{x\in\Sigma_\omega}W^s_{\omega,\delta_s}(x)$$
and consider $\pi_\omega$ the projection from $\C(\Sigma_\omega)$ onto $\Sigma_\omega$
along the local stable leaves.
We say that a $cu$-disk $\gamma$ {\em $u$-crosses} $\C(\Sigma_\omega)$ if $\pi_\omega(\gamma\cap \mathcal C(\Sigma_\omega))=\Sigma_\omega.$

The next result is purely deterministic, it is a consequence of the transitivity of $f$ on the partially hyperbolic set $K$ and it follows from \cite[Lemma 3.1 \& 3.2]{AP10}.

\begin{lemma}\label{l.N0q}
There are   $p\in D$ and $L\ge 1$ such that for  any  $\delta_0>0$ sufficiently small and each $cu$-disk
$\Sigma$ of radius $\delta_1/9$   there is $0\le\ell\le L$ such that  $f^{\ell}(\Sigma)$  intersects $W^s_{\delta_s/4} (p)$ and $u$-crosses $\mathcal C(B^{cu}_{2\delta_{0}}(p))$, where $B^{cu}_{2\delta_{0}}(p)$  is the
$cu$-disk in $D$ of radius $2\delta_0$  centred at \( p \).
\end{lemma}


Now we fix $p\in D$, $L\ge 1$ and $\delta_0>0$ small enough such that the conclusions of Lemma~\ref{l.N0q} hold, and define
\begin{equation}\label{eq.sigma01}
 \Sigma_0^0=
B^{cu}_{\delta_{0}}(p)\qand  \Sigma_0^1= B^{cu}_{2\delta_{0}}(p).
\end{equation}
For $i=0,1$ we consider the corresponding cylinders
\begin{equation}\label{eq.C0}
 \C^i_0=  \C(\Sigma_0^i).
 \end{equation}
 We also choose $\delta_0>0$ small so that any $cu$-disk intersecting $W^s_{\delta_s/4}(p)$ is at a distance at least $\delta_s/2$ from the top and bottom   of $\mathcal C_0^0$.

Since we assume  that each $f_\omega$ has a  partially hyperbolic  set $K_\omega$   with a local  unstable manifold $D_\omega\subset K_\omega$  close to $D$ in the $C^1$ topology, choosing $f_\omega$ sufficiently close to $f$, for each   $i=0,1$ there is a local unstable disk  $\Sigma_\omega^i \subset D_\omega$ close to $\Sigma_0^i$ such that  $$ \C^i_\omega=  \C (\Sigma_\omega^i)$$ is close to  $\C^i_0 $ in the Hausdorff distance. 
Here we use the fact that the local stable manifolds also depend continuously on the dynamics. 
%
Denoting $\pi_\omega$ the projection along the stable leaves of the cylinder $\mathcal C_\omega^1$, we have
$$\pi_\omega(\C^i_\omega)=\Sigma_\omega^i,\quad\mbox{for }i=0,1.$$

For each  $n\ge 1$ and $\omega\in\Omega$, define
 $$
 H_{\omega,n}=\{x\in \Sigma_\omega^0   \colon \text{$n$ is a $\alpha$-hyperbolic time for
 $(\omega,x)$}\}.
 $$
 Since $\Sigma_\omega^0\subset D_\omega$, 
From Proposition~\ref{pr.hyptimes}  and Proposition~\ref{pr.predisks} we have   for each   $x \in H_{\omega,n}$  a hyperbolic predisk  $V_{\omega, n}(x)\subset \Sigma_\omega^0$ which is mapped by $f_\omega^n$   diffeomorphically onto the disk $B^{cu}_{\delta_1} (f^n_\omega(x))$  with uniformly bounded distortion, by Proposition~\ref{p.distortion}. 
For each $0<a\le 1$, consider the set 
sets $
  V_{\omega,n}^a(x) \subset V_{\omega,n}(x)
$
such that
\begin{equation}\label{def.doubles}
\text{$f_\omega^n$  maps $V^a_{\omega,n}(x)$  diffeomorphically   onto   $B_{a\delta_1}(f_\omega^n(x))$.}\\
\end{equation}
Notice that $V_{\omega,n}^1(x) = V_{\omega,n}(x)$. 
It follows that for each  $x\in H_{\omega,k}$ and $y\in H_{\omega,n}$ with $n\ge k$, we   have
 \begin{equation}\label{eq.keywn}
 \begin{cases}
V_{\omega, n}^{1/9}(y)\cap V_{\omega, k}^{1/9}(x)\neq\emptyset\implies V_{\omega, n}^{1/9}(y)\subset V_{\omega, k}^{1/3}(x);\\
V_{\omega, n}^{1/3}(y)\cap V_{\omega, k}^{1/3}(x)\neq\emptyset\implies V_{\omega, n}^{1/3}(y)\subset V_{\omega, k}^1(x).
\end{cases}
\end{equation}
To see this,  observe first that by Proposition~\ref{pr.predisks} we have
\begin{equation}\label{eq.diamo}
\diam\left(f_\omega^k\left(V^{1/9}_{\omega,n}(y)\right)\right)\le \frac{2\delta_1}9\alpha^{\frac{n-k}{2}}\le \frac{2\delta_1}9.
\end{equation}
Then, assuming that $V^{1/9}_{\omega,n}(y)$  intersects $V^{1/9}_{\omega,n}(x)$, we necessarily have  $f_\omega^k(V^{1/9}_{\omega,n}(y))$   intersecting
$f_\omega^k(V^{1/9}_{\omega,n}(x))$, which by definition is the ball of radius $\delta_1/9$ around $f_\omega^k(x)$. Together with~\eqref{eq.diamo}, this implies that $f_\omega^k(V^{1/9}_{\omega,n}(y))$ is contained in the ball of radius $ \delta_1/3$ around $f_\omega^k(x)$, and so the first case of \eqref{eq.keywn} follows. The second case of \eqref{eq.keywn} can be obtained similarly.

\begin{remark}\label{la belle}
Since in Lemma~\ref{l.N0q} we have at most a finite number  of iterates $N_0$, shrinking~$\delta_0$ if necessary, a similar  conclusion can be drawn for the random perturbations:
 for any $\omega$ and any  $cu$-disk
$\Sigma$ of radius $\delta_1/9$   there is $0\le\ell\le L$ such that  $f_\omega^{\ell}(\Sigma)$  intersects $W^s_{\sigma^\ell\omega,\delta_s/2} (y)$ for some $y\in \Sigma_{\sigma^\ell\omega}^0$ and $u$-crosses $\mathcal C_{\sigma^\ell\omega}^1$.
Indeed, since the images $f^{n}_\omega(V_{\omega,n}^{1/9}(x))$ of   hyperbolic predisks $V_{\omega,n}^{1/9}(x)$ with $x\in\Sigma_\omega$ are $cu$-disks of radius $\delta_1/9$, we easily get that for each hyperbolic predisk $V_{\omega,n}^{1/9}(x)$ with $x\in\Sigma_\omega$, there is $0\le \ell_{\xi_\omega}\le L$ for which $f^{n+\ell_{\xi_\omega}}_\omega(V^{1/9}_{\omega,n}(x))$ intersects $W_{\sigma^{n+\ell_{\xi_\omega}}\omega, \delta_s/2}^s(y)$, for some $y\in\Sigma_{\sigma^{n+\ell_{\xi_\omega}}\omega}^0$, and $u$-crosses $\mathcal C^1_{\sigma^{n+\ell_{\xi_\omega}}\omega}$.
 \end{remark}

Hence, for each $i=0,1$ there are $cu$-disks 
$\xi_{\omega,x}^{n,\ell_{\xi_\omega}}\subset\tilde\xi_{\omega,x}^{n,\ell_{\xi_\omega}}\subset V^{1/9}_{\omega,n}(x)$ such that
\begin{equation}\label{D.candidate2}
 \pi_{\sigma^{n+\ell_{\xi_\omega}}\omega}\left(f_\omega^{n+\ell_{\xi_\omega}}(\xi_{\omega,x}^{n,\ell_{\xi_\omega}})\right)=
 \Sigma_{\sigma^{n+\ell_{\xi_\omega}}\omega}^0
 \qand  \pi_{\sigma^{n+\ell_{\xi_\omega}}\omega}\left(f_\omega^{n+\ell_{\xi_\omega}}(\tilde{\xi}_{\omega,x}^{n,\ell_{\xi_\omega}})\right)=\Sigma_{\sigma^{n+\ell_{\xi_\omega}}\omega}^1,
\end{equation}
with the property that $f_\omega^{n+\ell_{\xi_\omega}}(\xi_{\omega,x}^{n,\ell_{\xi_\omega}})$ intersects $W_{\sigma^{n+\ell_{\xi_\omega}}\omega, \delta_s/2}^s(y)$, for some $y\in\Sigma_{\sigma^{n+\ell_{\xi_\omega}}\omega}^0$. 

As condition \eqref{D.candidate2} may in principle hold for several values of $0\le\ell_{\xi_\omega}\le L$, for definiteness we will assume that $\ell_{\xi_\omega}$ takes the smallest possible value. Observe that the $cu$-disk $\tilde\xi_{\omega,x}^{n,\ell_{\xi_\omega}}$ is associated to $x$, by construction, but does not necessarily contain~$x$.
In the sequel we will often simplify the notation and  refer to elements of the type $\xi_{\omega,x}^{n,\ell_{\xi_\omega}}$ or $  \tilde\xi_{\omega,x}^{n,\ell_{\xi_\omega}}$ as $\xi_\omega$ or $\tilde\xi_\omega$, respectively. In such cases we will also consider
 $$
 \ell_{ \xi_\omega}=\ell_{\xi_\omega}\qand V_{\omega,n}^a(\xi)=V_{\omega,n}^a(x) 
 $$
 with $V_{\omega,n}^a(x)$  defined as in \eqref{def.doubles}.
Given  $x\in H_{\omega,k}$,   define for $\xi=\xi_{\omega,x}^{k,\ell_{\xi_\omega}}$ and  $n>  k  $
\begin{equation}\label{eq.anelan}
A_{\omega,n}\left(\xi \right)=\left\{y\in \tilde\xi : \dist_{f_\omega^{n+\ell_{\xi_\omega}}(D_\omega)}\left(   
 f_\omega^{k+\ell_{\xi_\omega}}(y), f_\omega^{k+\ell_{\xi_\omega}}(\xi )\right)\le \delta_0\alpha^{\frac{n-k}{2}}\right\}.
\end{equation}

\begin{remark}\label{re.sizeleast}
Without loss of generality, here we  assume  that  the  $cu$-disks $f_\omega^{k+\ell_{\xi_\omega}}(\tilde\xi )$  and $f_\omega^{k+\ell_{\xi_\omega}}(\xi )$  are  still   disks of radius $2\delta_0$ and $\delta_0$, respectively. 
 It follows  from the definition of $A_{\omega,n}(\xi)$   that $  f_\omega^{k+\ell_{\xi_\omega}}(\tilde\xi)$ contains a neighbourhood of  the outer component of the boundary of $  f_\omega^{k+\ell_{\xi_\omega}}(A_{\omega,n}(\xi))$ of size at least  $$2\delta_0-\delta_0\left(1+\alpha^{( n-k)/2}\right)=\delta_0\left(1-\alpha^{( n-k)/2}\right).$$
Using the fact that $0\le \ell_{\xi_\omega} \le L$ and  taking 
\begin{equation}\label{eq.defK0}
K_0=\max_{0\le\ell\le L}\{\| Df_\omega^{\ell} \|:  \omega\in\Omega  \},
\end{equation}
we easily get that $f_\omega^{k }(\tilde\xi)$ contains a neighbourhood of the outer component of the boundary of $f_\omega^{k+\ell_{\xi_\omega}}(A_{\omega,n}(\xi))$ of  size at least
  $$\frac{\delta_0\left(1-\alpha^{( n-k)/2}\right)}{K_0}.$$
  \end{remark}
  
This last remark partly  motivates the choice of the  constants that we introduce next. We take
\begin{equation}\label{eq.delta2}
\delta_2=\delta_0+\frac{\delta_1}2K_0
\end{equation}
and choose $N_1\in\mathbb N$ such that
\begin{equation}\label{eq.N1}
\delta_2\alpha^{\frac{N_1}2}\le \delta_0\qand \frac{2\delta_1}9\alpha^{\frac{N_1}{2}}\le \frac{\delta_0(1-\alpha^{N_1/2})}{K_0}.
\end{equation}
Below we describe the inductive process that leads to a partition $\mathcal P_{\omega}$ of each local unstable disk $\Sigma_\omega^0$. The sets $ \xi_{\omega,n}^{n,\ell_{\xi_\omega}}\subset \Sigma_\omega^0$ as in~\eqref{D.candidate2}   are the natural candidates to be in~$\mathcal P_\omega$.

\subsubsection*{First step of induction}

Fixing some large $N_0\in\mathbb{N} $, we ignore the dynamics until time
$N_0$.  Since $H_{\omega,N_0}$  is a compact set, there is a finite set  
$E_{\omega,N_0} \subset H_{\omega,N_0}\cap\Sigma_\omega^0$ such that
$$H_{\omega,N_0}\cap \Sigma_\omega^0\subset \bigcup_{x\in E_{\omega,N_0}}V_{\omega,N_0}^{1/9}(x).$$
Consider $x_1,\dots,x_{j_{N_0}}\in E_{\omega,N_0}$ such that 
$$
\mathcal P_{\omega,N_0}=\left\{
\xi_{\omega,N_0}^{x_1,\ell_{1}},\ldots,
\xi_{\omega,N_0 }^{x_{j_{N_0}}, \ell_{j_{N_0}}}\right\}
$$
is a maximal family of pairwise disjoint sets as in (\ref{D.candidate2}) contained in
$\Sigma_\omega^0$. The sets in $\mathcal P_{\omega,N_0}$ are precisely the elements of the partition $\mathcal P$ constructed in  our first step of the construction.
Consider also  
\begin{equation*} 
  \Sigma_{\omega,N_0}^c = {\Sigma_\omega^0}\setminus  \bigcup_{\xi\in\mathcal P_{\omega,N_0}}\xi.
\end{equation*}
 For each $\xi \in\mathcal P_{\omega,N_0}$,  we define
 $$ 
 S_{\omega,N_0}(\xi)=V_{\omega,N_0}^{1/3}(\xi),
 $$
 with $V_{\omega,N_0}^{1/3}(\xi)$   as in \eqref{def.doubles}.
We also define
\begin{equation*}
 S_{\omega,N_0}(\Sigma_\omega^0)  =\bigcup_{\xi\in
\mathcal P_{\omega,N_0}}{S}_{\omega,N_0}(\xi).
\end{equation*}
 and    for  ${\Sigma_\omega^c}=\Sigma_\omega^1\setminus{\Sigma_\omega^0}$ 
\begin{equation*}\label{exremark}
{S}_{\omega,N_0}\left(\Sigma_\omega^c\right)= \left\{x \in \Sigma_\omega^0:\,
\dist_{D_\omega}(x,\partial \Sigma_\omega^0)<2\delta_1\alpha^{N_0/2}\right\}.
\end{equation*}

\subsubsection*{General step of induction}

The  next steps of the construction follow the ideas of the first one with
minor modifications. Given $n>N_0$, assume that 
$\mathcal P_{\omega,k}$, $\Sigma_{\omega,k}$  and ${S}_{\omega,k}$ have   been defined  for all $ N_0\le k\leq n-1 $.
As before, let
$ E_{\omega,n}\subset  H_{\omega,n}\cap\Sigma^c_{\omega,n-1}$ be a finite set of points such that
\begin{equation}\label{eq.endef}
H_{\omega,n}\cap\Sigma^c_{\omega,n-1}\subset \bigcup_{x\in E_{\omega,n}}
V_{\omega,n}^{1/9}(x).
\end{equation}
Consider $ x_1,\dots,x_{j_{n}}\in E_n$ such that 
$$
\mathcal P_{\omega,n}=\left\{
\xi_{\omega,n}^{x_1,\ell_{1}},\ldots,
\xi_{\omega,n }^{x_{j_{n}}, \ell_{j_{n}}}\right\}
$$
is a maximal family  of pairwise disjoint sets as in \eqref{D.candidate2},  all contained in~$\Sigma^c_{\omega,n-1}$ and
satisfying for  each $1\le i\le j_n$
\begin{equation}\label{eq.omegaanel}
\xi_{\omega,n}^{x_i,\ell_{i}}\cap\left(
\bigcup_{k=N_0}^{n-1}\bigcup_{\xi\in \mathcal P_{\omega,k}} A_{\omega,n}(\xi)\right)=\emptyset.
\end{equation}
The sets in $\mathcal P_{\omega,n}$ are the elements of the partition $\mathcal P_\omega$ constructed in the
step  $n$ of the algorithm.
Consider also  for each $n\ge N_0$
 \begin{equation}\label{eq.deltans}
\Sigma^c_{\omega,n} = {\Sigma_\omega^0}\setminus \bigcup_{j=N_0}^{n}\bigcup_{\xi\in\mathcal P_{\omega,j}}\xi.
\end{equation}
Given $\xi\in \mathcal P_{\omega,k}$ for some $N_0\le k\le n$, we define for
$n-k < N_1$
$$
S_{\omega,n}\left(\xi \right)=V_{\omega,k}^{1/3}(\xi).$$
and for
 $n-k\ge N_1 $
\begin{equation}\label{eq.essenomega}
S_{\omega,n}\left(\xi\right)=\left\{y\in \tilde\xi: 0< \dist_{f_\omega^{n+\ell_{\xi_\omega}}(D_\omega)}\left(    f_\omega^{k+\ell_{\xi_\omega}}(y), f_\omega^{k+\ell_{\xi_\omega}}(\xi )\right)\le \delta_2\alpha^{\frac{n-k}{2}}  \right\};
\end{equation}
recall~\eqref{eq.delta2} and~\eqref{eq.N1}.
Then, we define
  \begin{equation}\label{eq.Sn1}
   S_{\omega, n}({\Sigma_\omega^0})= \bigcup_{j=N_0}^{n}\bigcup_{\xi\in
\mathcal P_{\omega,j}}{S}_{\omega,n}(\xi) 
\end{equation}
and 
\begin{equation*}
    S_{\omega,n}\left(\Sigma_\omega^c\right)=    \left\{x \in \Delta_{0}:\,
\dist_{D_\omega}(x,\partial\Sigma_\omega^{0})< \delta_1\alpha^{n/2}\right\}.
\end{equation*}

%
%

Finally, set 
$$
\mathcal P_\omega=\bigcup_{n\ge N_0}\mathcal P_{\omega,n}.
$$
By construction,   the elements in $\mathcal P_\omega$ are pairwise disjoint and contained in ${\Sigma_\omega^0}$. However,   there is still no evidence that the union of these elements   covers a full $\leb_{D_\omega}$ measure subset of~${\Sigma_\omega^0}$. This will be obtained in Proposition~\ref{pr.partition}.


Now we prove some properties about the sets $S_{\omega,n}(\xi)$   introduced above. First of all observe that, by definition,  for  each $\xi\in\mathcal P_{\omega,k}$ with $k\ge N_0$ we have
\begin{equation}\label{eq.sncontido}
V_{\omega,k}^{1/3}(\xi)\supset S_{\omega,k}(\xi)\supset S_{\omega,k+1}(\xi)\supset\cdots
\end{equation}
Moreover, if  $n-k\ge N_1$,   we even have
\begin{equation}\label{eq.WSn}
V_{\omega,k}^{1/9}(\xi)\supset S_{\omega,n}(\xi).
\end{equation}
From the construction of these sets and~\eqref{eq.keywn},  it follows that for all  $k_2\ge k_1\ge N_0$, 
given $\xi_1\in\mathcal P_{\omega,k_1}$ and $\xi_2\in\mathcal P_{\omega,k_2}$, we have 
 \begin{equation*}\label{eq.incluisses0}
S_{\omega,k_2}(\xi_2)\cap V_{\omega,k}^{1/3}(\xi)\neq\emptyset \implies  S_{\omega,k_2 }(\xi_2)\cup \xi_2\subset V_{\omega,k_1}(\xi_1).
\end{equation*}
and for any  $n \ge N_1$  
 \begin{equation}\label{eq.incluisses}
S_{\omega,k_2+n}(\xi_2)\cap V_{\omega,k}^{1/9}(\xi)\neq\emptyset \implies  S_{\omega,k_2+n}(\xi_2)\cup \xi_2\subset V_{\omega,k}^{1/3}(\xi).
\end{equation}

In the proof of the next result  we use in an important way property~\eqref{eq.omegaanel}, which establishes   that an element in  $\mathcal P_\omega$ obtained at a certain  stage of the construction not only does not  intersect an  element  constructed at a previous stage, but also it does not intersect a   larger  annulus of the type~\eqref{eq.anelan} around~it. 
This fact will  be   useful to deduce the   estimates on the tail of recurrence times in Subsection~\ref{s.tail esti}. 
We fix once and for all an integer    $Q_0\ge N_1$ such that
\begin{equation}\label{eq.chooseq}
\delta_2(1+K_0^2)\,\alpha^{Q_0/2}<\delta_0\qand \delta_2K_0\alpha^{Q_0/2}<\frac{\delta_0\left(1-\alpha^{1/2}\right)}{K_0}.
\end{equation}

 \begin{lemma}\label{le.existQ}
If   $k_2>k_1\ge N_0$, then    
for  all $\xi_1\in\mathcal P_{\omega,k_1}$ and all $\xi_2\in\mathcal P_{\omega,k_2}$ we have
 $$
   S_{\omega,k_2+Q_0}(\xi_1)\cap S_{\omega,k_2+Q_0}(\xi_2)=\emptyset.
 $$
\end{lemma}
\begin{proof}
Assume by contradiction that    
\begin{equation}\label{eq.esses}
\exists x\in S_{\omega,k_2+Q_0}(\xi_1)\cap   S_{\omega,k_2+Q_0}(\xi_2) .
\end{equation}
By~\eqref{D.candidate2}, we have
\begin{equation}\label{eq.bla}
\dist_{ f_\omega^{k_2+\ell_2}(D_\omega)}\left(f_\omega^{k_2+\ell_2}(x), f_\omega^{k_2+\ell_2}(\xi_2)\right)<\delta_2\alpha^{Q_0/2},
\end{equation}
for some $0\le \ell_2\le L$ such that 
$$
\pi_{\sigma^{k_2+\ell_{2}}\omega}\left(f_\omega^{k_2+\ell_{2}}(\tilde{\xi}_2)\right)=
\Sigma_{\sigma^{k_2+\ell_{2}}\omega}^1.
$$
 Using~\eqref{eq.defK0} and~\eqref{eq.bla}, we find   $y\in \xi_2$ such that
\begin{equation}\label{eq.deltasi2}
\dist_{ f_\omega^{k_2 }(D_\omega)}\left(f_\omega^{k_2 }(x),f_\omega^{k_2 }(y)\right)<\delta_2K_0\alpha^{Q_0/2}.
\end{equation}
Also, from~\eqref{eq.WSn} and~\eqref{eq.esses} it follows  in particular that
$S_{\omega,k_2+Q_0}(\xi_2)$ intersects $V_{\omega,k_1}^{1/9}(\xi_1) $, and so by \eqref{eq.incluisses} we have $S_{\omega,k_2+Q_0}(\xi_2)\cup\xi_2\subset V_{\omega,k}^{1/3}(\xi_1)$. Then, using~\eqref{eq.deltasi2} and   Proposition~\ref{pr.predisks}   we obtain
\begin{eqnarray}
\dist_{f_\omega^{k_1}(D_\omega)}\left(f_\omega^{k_1 }(x),f_\omega^{k_1 }(y)\right)&\le & \alpha^{(k_2-k_1)/2}\dist_{f_\omega^{k_2}(D_\omega)}\left(f_\omega^{k_2 }(x),f_\omega^{k_2 }(y)\right)\nonumber\\
&<&\delta_2K_0\alpha^{(Q_0+k_2-k_1)/2}.\label{eq.deltask1}
\end{eqnarray}
On the other hand, since $x\in S_{\omega,k_2+Q_0}(\xi_1)$ we also have by definition
\begin{equation}\label{eq.deltaQ0}
\dist_{ f_\omega^{k_1+\ell_1}(D_\omega)}\left(f_\omega^{k_1+\ell_1}(x), f_\omega^{k_1+\ell_1}(\xi_1)\right)<\delta_2\alpha^{(Q_0+k_2-k_1)/2},
\end{equation}
for some $0\le \ell_1\le L$ such that  
$$
\pi_{\sigma^{k_1+\ell_{1}}\omega}\left(f_\omega^{k_1+\ell_{1}}(\tilde{\xi}_1)\right)=
\Sigma_{\sigma^{k_1+\ell_{1}}\omega}^1.
$$
Recalling the first part of~\eqref{eq.chooseq}, we easily conclude that
$x\in A_{\omega,k_2}(\xi_1)$.
 Then,  Remark~\ref{re.sizeleast}  
   gives that $f_\omega^{k_1 }(\widetilde\xi_1)$ contains a neighbourhood of the outer component of the boundary of $f_\omega^{k_1}(A_{\omega,k_2}(\xi_1))$ of  size at least
  $$\frac{\delta_0\left(1-\alpha^{( k_2-k_1)/2}\right)}{K_0}\ge \frac{\delta_0\left(1-\alpha^{1/2}\right)}{K_0}.$$
  Recalling now the second part of~\eqref{eq.chooseq} and using~\eqref{eq.deltask1} we deduce that $f_\omega^{k_1}(y)\in f_\omega^{k_1 }(\widetilde\xi_1)$.  So, using~\eqref{eq.defK0} together with~\eqref{eq.deltask1} we obtain
  \begin{equation}\label{eq.deltaskell1}
\dist_{ f_\omega^{k_1+\ell_1}(D_\omega)}\left(f_\omega^{k_1+\ell_1 }(x),f_\omega^{k_1+\ell_1 }(y)\right) <\delta_2K_0^2\alpha^{(Q_0+k_2-k_1)/2},
\end{equation}
which jointly with~\eqref{eq.deltaQ0} and  the first part of~\eqref{eq.chooseq}  yields
$$
\dist_{ f_\omega^{k_1+\ell_1}(D_\omega)}\left(f_\omega^{k_1+\ell_1}(y),{f_\omega^{k_1+\ell_1}(\xi_1)}\right)< \delta_2(1+K_0^2)\,\alpha^{(Q_0+k_2-k_1)/2}<\delta_0\alpha^{(k_2-k_1)/2}.
$$
This implies that $y\in A_{\omega,k_2}(\xi_1)$ with  $\xi_1\in\mathcal P_{\omega,k_1}$. Since $y\in \xi_2$ and $\xi_2\in\mathcal P_{\omega,k_2}$ we have a contradiction with~\eqref{eq.omegaanel}.
 \end{proof}

In the next result  we prove  the key fact   that every point having a $\alpha$-hyperbolic time at a given time $n$ will necessarily belong to either to an element of the partition or to  the   set ${S}_{\omega,n}({\Sigma_\omega^0})\cup S_{\omega,n}({\Sigma_\omega^c} )$. In Lemma~\ref{le.sumSn} we will show that the sum of the measure of all these  sets  is finite. This will be an important  step towards proving that~$\mathcal P_\omega$ is indeed a $\leb $ mod~0 partition of ${\Sigma_\omega^0}$.

\begin{lemma}\label{le.keyinc}  For each $n\ge N_0$, we have
 $ H_{\omega,n}\cap\Sigma_{\omega,n}^c\subset
{S}_{\omega,n}({\Sigma_\omega^0})\cup S_{\omega,n}({\Sigma_\omega^c}).
 $
 \end{lemma}
\begin{proof}
Consider   the finite set
$ E_{\omega,n}\subset  H_{\omega,n}\cap\Sigma_{\omega,n-1}^c$   as in~\eqref{eq.endef}.
Given  any $z\in H_{\omega,n}\cap \Sigma_{\omega,n}^c\subset H_{\omega,n}\cap\Sigma_{\omega,n-1}^c$, there is $y\in E_{\omega,n}$ such that   $z\in V_{\omega,n}^{1/9}(y)$. It is enough to show that $V_{\omega,n}^{1/9}(y)\subset {S}_{\omega,n}({\Sigma_\omega^0})\cup S_{\omega,n}({\Sigma_\omega^c} )$. Since $z\in  \Sigma_{\omega,n}^c$,  at least one of the following    cases holds:
\begin{enumerate}
\item \emph{$\xi_{\omega,y}^{n,\ell_{\omega,y}}\cap \xi\neq\emptyset$, for some $\xi\in\mathcal P_{\omega,n}$.}\\
In this case, we have $S_{\omega,n}(\xi )=V_{\omega,n}^{1/3}(\xi)$ and   $V_{\omega,n}^{1/9}(y)$ necessarily intersecting $V_{\omega,n}^{1/9}(\xi)$.  
Hence,  using \eqref{eq.keywn} we get
 $$ V_{\omega,n}^{1/9}(y)\subset V_{\omega,n}^{1/3}(\xi)=S_{\omega,n}(\xi )\subset {S}_{\omega,n}({\Sigma_\omega^0}).$$
 
 \item \emph{$\xi_{\omega,y}^{n,\ell_{\omega,y}}\cap A_{\omega,n}(\xi) \neq\emptyset$, for some $N_0\le k<n$ and $\xi\in\mathcal P_{\omega,k}$.}\\
Observe that by definition we have
$ 
A_{\omega,n}\left(\xi \right)\subset  V_{\omega,n}^{1/9}(\xi).
$
Assume first that  $n-k<N_1$. Then,   as in the previous situation, using \eqref{eq.keywn} we get 
  $$ V_{\omega,n}^{1/9}(y)\subset V_{\omega,k}^{1/3}(\xi)=S_{\omega,n}(\xi )\subset {S}_{\omega,n}({\Sigma_\omega^0}).$$
 Assume now that  $n-k\ge N_1$.   
 We claim that 
 \begin{equation}\label{eq.wholds}
f_\omega^k\left(V_{\omega,n}^{1/9}(y)\right)\subset f_\omega^k (\tilde\xi).
\end{equation}
Recall that by Remark~\ref{re.sizeleast} the set  $f_\omega^{k }(\tilde\xi)$ contains a neighbourhood of the outer component of the boundary of $f_\omega^{k}(A_{\omega,n}(\xi))$ of  size at least
  $$\frac{\delta_0(1-\alpha^{N_1/2})}{K_0}.$$
On the other hand, using  Proposition~\ref{pr.predisks} we obtain
 \begin{equation}\label{eq.diamum}
\diam\left(f_\omega^k\left(V_{\omega,n}^{1/9}(y)\right)\right)\le \frac{2\delta_1}9\alpha^{\frac{n-k}{2}}\le \frac{2\delta_1}9\alpha^{\frac{N_1}{2}}.
\end{equation}
Recalling the choice of $N_1$ in~\eqref{eq.N1}
 and   observing that in the situation  we are considering the set $f_\omega^k\left(V_{\omega,n}^{1/9}(y)\right)$ intersects $f_\omega^k\left(A_{\omega,n}(\xi)\right)$,  we conclude that~\eqref{eq.wholds} holds. 
 Then, using~\eqref{eq.defK0}  and~\eqref{eq.diamum} we also obtain
\begin{equation}\label{eq.diamois}
\diam\left(f_\omega^{k+\ell_\xi}\left(V_{\omega,n}^{1/9}(y)\right)\right)\le \frac{2\delta_1}9K_0\alpha^{\frac{n-k}{2}} .
\end{equation}
 Now, since the set $f_\omega^{k+\ell_\xi}\left(V_{\omega,n}^{1/9}(y)\right) $ intersects $f_\omega^{k+\ell_\xi}\left(  A_{\omega,n}(\xi)\right)$, we have for each $u\in f_\omega^{k+\ell_\xi}\left(V_{\omega,n}^{1/9}(y)\right) $
  $$\dist_{D_\omega}(u,{\Sigma_\omega^0})\le \delta_0\alpha^{\frac{n-k}{2}}+\frac{2\delta_1}9K_0\alpha^{\frac{n-k}{2}}=\delta_2\alpha^{\frac{n-k}{2}}.
  $$
 This shows that $V_{\omega,n}^{1/9}(y)\subset S_{\omega,n}(\xi)\subset {S}_{\omega,n}({\Sigma_\omega^0})$. 
 
 \item \emph{$\xi_{\omega,y}^{n,\ell_{\omega,y}}\cap {\Sigma_\omega^c}\neq\emptyset $.}\\
 This in particular  implies  that  $V_{\omega,n}^{1/9}(y)$ intersects $\partial{\Sigma_\omega^0}$. From  Proposition~\ref{pr.predisks}  we get
 \begin{equation*}\label{eq.diamumdois}
\diam\left( V_{\omega,n}^{1/9}(y) \right)\le \frac{2\delta_1}9\alpha^{\frac{n}{2}} ,
\end{equation*}
and so $V_{\omega,n}^{1/9}(y) \subset  S_{n}({\Sigma_\omega^c})$.
\end{enumerate}
\end{proof}


Now we give some estimates on the measure of the sets involved in the construction of~$\mathcal P_\omega$. We claim that there exists  $C_0> 0$ such that for any $n\ge N_0$ and   $\xi\in
\mathcal{P}_{\omega,n}$ we have
\begin{equation}\label{eq.sumle}
\leb_{D_\omega} \left(V_{\omega,n}(\xi)\right)\leq C_0 \leb_{D_\omega}  (\xi).
\end{equation}
In fact, 
using Proposition~\ref{p.distortion} we get some constant  $C_1>0$ such that
 \begin{equation}\label{eq.jose}
\frac{\leb_{D_\omega} \left(V_{\omega,n}(\xi)\right)}{\leb_{D_\omega} \left(\xi\right)} \le C_1\frac{\leb_{f_\omega^n(D_\omega)} (f_\omega^n(V_{\omega,n}(\xi)))}{\leb_{f_\omega^n(D_\omega)} (f_\omega^n(\xi))}
\le 
C_1\frac{C_2\delta_1^2}{\leb_{f_\omega^n(D_\omega)} (f_\omega^n(\xi))},
\end{equation}
where $C_2>0$ is some uniform constant; recall that $f_\omega^n(V_{\omega,n}(\xi))$ is a $cu$-disk of radius $\delta_1$. 
On the other hand, using that $f_\omega^{n+\ell_\xi}(\xi)$ is a $cu$-disk of radius $\delta_0$ and~\eqref{eq.defK0} we have
\begin{equation}\label{eq.wael}
\leb_{f_\omega^n(D_\omega)} (f_\omega^n(\xi))\ge \frac1{K_0} \leb_{f_\omega^{n+\ell_\xi}(D_\omega)} (f_\omega^{n+\ell_\xi}(\xi))\ge \frac{c_0\delta_0^2}{K_0} ,
\end{equation}
where   $c_0>0$ is again some uniform constant. From~\eqref{eq.jose} and~\eqref{eq.wael} we deduce~\eqref{eq.sumle}.

  \begin{lemma}\label{le.sncore} There
exists  $C> 0$ such that for all $n\geq k\ge N_0$ and  $\xi\in
\mathcal{P}_{\omega,k}$ we have
$$
\leb_{D_\omega} \left(S_{\omega,n}(\xi)\right)\leq C \alpha^{(n-k)/2} \leb_{D_\omega} (\xi).
$$
\end{lemma}

\begin{proof} 
Assuming first that $n-k<N_1$, as $S_n(\xi)\subset V_k(\xi)$, it follows   from~\eqref{eq.sumle}  
that 
\begin{equation}\label{eq.smlesn}
\leb_{D_\omega}\left(S_{\omega,n}(\xi)\right)\leq C_0 \leb_{D_\omega}(\xi).
\end{equation}
Assume now that $n-k\ge  N_1$. Recalling that
\begin{equation*} 
S_{\omega,n}\left(\xi\right)=\left\{y\in \tilde\xi: 0< \dist_{f_\omega^{k+\ell_\xi}(D_\omega)} \left(    f_\omega^{k+\ell_{\xi_\omega}}(y), f_\omega^{k+\ell_{\xi_\omega}}(\xi )\right)\le \delta_2\alpha^{\frac{n-k}{2}}  \right\},
\end{equation*}
then by the choice of $N_1$ in~\eqref{eq.N1} we have
$f_\omega^{k+\ell_\xi}(S_{\omega,n}(\xi))$ contained in a $cu$-disk of radius~$2\delta_0$; recall Remark~\ref{re.sizeleast}.
We also have  some uniform constant $C_2>0$ such that
$$\leb_{f_\omega^{k+\ell_\xi}(D_\omega)}\left(f_\omega^{k+\ell_\xi}(S_{\omega,n}(\xi))\right)\le C_2\alpha^{\frac{n-k}{2}}.$$
Hence, using~\eqref{eq.defK0}, \eqref{eq.wael} and Proposition~\ref{p.distortion}  we get some constant $C_1>0$  such  that   
  \begin{equation}\label{eq.snomega2}
\frac{\leb_{D_\omega} (S_{\omega,n}(\xi))}{\leb_{D_\omega} (\xi)} \le C_1K_0\frac{\leb_{f_\omega^{k+\ell_\xi}(D_\omega)}(f_\omega^{k+\ell_\xi}(S_{\omega,n}(\xi)))}{\leb_{f_\omega^{k+\ell_\xi}(D_\omega)}(f_\omega^{k+\ell_\xi}(\xi))}\le   \frac{C_1C_2K_0}{c_0\delta_0^2}\alpha^{\frac{n-k}{2}}.
\end{equation}
Finally, choose
$$
C\ge \frac{C_1C_2K_0}{c_0\delta_0^2}
$$
sufficiently large so that $C_0\le C\alpha^{ {N_1}/2}$. Using  \eqref{eq.snomega2} and \eqref{eq.smlesn} we easily derive  the desired conclusion.
\end{proof}

 \begin{lemma} \label{le.sumSn}
$\displaystyle{\sum_{n=N_0}^{\infty}\leb_{D_\omega} \left({S}_{\omega,n}(\Sigma_\omega^0)\cup S_{\omega,n}(\Sigma_\omega^c)\right) <\infty}$.
\end{lemma}
 \begin{proof}
First we consider the terms in $S_{\omega,n}(\Sigma_\omega^0)$. Recalling   \eqref{eq.Sn1} and using Lemma~\ref{le.sncore},   for each   $n\ge N_0$
 we may write
\begin{eqnarray*}
\leb_{D_\omega}\left( S_{\omega,n}(\Sigma_\omega^0)\right) &=&
 \sum_{k=N_0}^n\sum_{\xi\in
\mathcal P_{\omega,k}} \leb_{D_\omega} (S_{\omega,n}(\xi))\\
&\le&
 \sum_{k=N_0}^n\sum_{\xi\in
\mathcal P_{\omega,k}} C \alpha^{(n-k)/2} \leb_{D_\omega} (\xi)\\
&=&
C \sum_{k=N_0}^n\alpha^{(n-k)/2}\leb_{D_\omega}\left(\bigcup_{\xi\in
\mathcal P_{\omega,k}}  \xi\right)
\end{eqnarray*}
Hence
$$
\sum_{n\ge N_0} \leb_{D_\omega}\left( S_{\omega,n}(\Sigma_\omega^0)\right)=\sum_{n\ge N_0}\sum_{k\ge0}\alpha^{k/2}\leb_{D_\omega}\left(\bigcup_{\xi\in \mathcal P_{\omega,n}}\xi\right)=\frac{1}{1-\alpha^{1/2}}\leb_{D_\omega}(\Sigma_\omega^0).
$$
On the other hand, recalling that
\begin{equation*}\label{exremark0}
{S}_{\omega,n}\left(\Sigma_\omega^c\right)=\left\{x \in \Sigma_\omega^{0}:\,
\dist_{D_\omega}(x,\partial\Sigma_\omega^{0})< \delta_1\alpha^{n/2}\right\},
\end{equation*}
 we can find $C>0$ such that
$
\leb_{D_\omega}  \left({S}_{\omega,n}\left(\Sigma_\omega^c \right)\right)\leq C\alpha^{n/2}.
$
This obviously gives that the sum of these terms for $n\ge N_0$ is finite.
\end{proof}

 \begin{proposition}\label{pr.partition}
$\mathcal P_\omega$ is a $\leb_{D_\omega}$ mod 0 partition of $\Sigma_\omega^0$.
\end{proposition}

\begin{proof} 
Recalling  the definition  of  the sets $\Sigma_{\omega,n}^c$   in~\eqref{eq.deltans},
it is enough to show that the intersection of all these sets  has zero $\leb_{D_\omega} $ measure. Assume by contradiction that
\begin{equation}\label{eq:cap1}
\leb_{D_\omega}  \left(\bigcap_{n\ge N_0}\Sigma_{\omega,n}^c\right)>0.
\end{equation}
Then, since $\leb_{D_\omega} $ almost every point in $\Sigma_\omega^0$ has infinitely many $\alpha$-hyperbolic times, there must be  some Borel set $B\subset\Sigma_\omega^{0}$ with $\leb_{D_\omega} \left(B\right)>0$ such that for every $x\in B$ we can find infinitely many times  $n_1<n_2<\cdots $ (in principle depending on $\omega,x$)  so that 
 $x\in H_{\omega,n_k}\cap \Sigma_{\omega,n_k}^c$  for all $k\in\mathbb N.$ It follows from 
Lemma~\ref{le.keyinc} that 
\begin{equation}\label{eq.infinets}
x \in S_{\omega,n_k}(\Sigma_\omega^{0})\cup S_{\omega,n_k}\left(\Sigma_\omega^{c}\right),\quad \text{for all $ k\in\mathbb N.$}
\end{equation}
On the other hand,  using Lemma~\ref{le.sumSn} and 
Borel-Cantelli Lemma we easily deduce that for $\leb_{\Sigma_\omega^1}$ almost every
$x\in \Sigma_\omega^{0}$ we cannot have  $ x \in S_{n}(\Sigma_\omega^{0})\cup S_{n}\left(\Sigma_\omega^{c}\right)$ for infinitely many values of $n$. Clearly, this gives a contradiction with the fact that   $\leb_{D_\omega} (B)>0$ and  \eqref{eq.infinets} holds  for   every $x\in B$. 
\end{proof}

 \subsection{Return times}
\label{s.tail esti}


In the previous section we have constructed for each $\omega$ a $\leb_{\Sigma_\omega^1}$ mod~0 partition   $$\mathcal P_\omega=\bigcup_{n\ge N_0}\mathcal P_{\omega,n}$$ of the $cu$-disk   $\Sigma_\omega^0\subset D_\omega$.
This partition is  formed by elements $  \xi_{\omega,x}^{n,\ell_{\xi_\omega}}\in\mathcal P_{\omega,n}$ of the type described in \eqref{D.candidate2}. In particular,  each $  \xi_{\omega,x}^{n,\ell_{\xi_\omega}}$ satisfies
\begin{equation*}
 \xi_{\omega,x}^{n,\ell_{\xi_\omega}}\subset  V^{1/9}_{\omega,n}(x)\qand   \pi_{\sigma^{n+\ell_{\xi_\omega}}\omega}\left(f_\omega^{n+\ell_{\xi_\omega}}(\xi_{\omega,x}^{n,\ell_{\xi_\omega}})\right)=
 \Sigma_{\sigma^{n+\ell_{\xi_\omega}}\omega}^0,
\end{equation*}
for some $0\le\ell_{\xi_\omega}\le L$. 
Naturally, for each $y\in
\ell_{\xi_\omega}\in\mathcal P_{\omega,n}$ we set the \emph{recurrence time} 
\begin{equation}\label{eq.rectime}
R_\omega(y)=n+\ell_{\xi_\omega}.
\end{equation}

\begin{remark}\label{re.VIR}
Consider   $f\in\diff^{1^+}(M)$ as in Theorem~\ref{randomPH}. For  the deterministic case, it has been proved in \cite[Section~5]{ADL17} that the construction on the reference leaf $\Sigma_0^0$ performed  in Subsection~\ref{s.structure} gives rise to a set with a hyperbolic product  structure contained in~$ K$   with integrable return times $\{R_i\}$ and partition $\mathcal P=\{\xi_i\}$. If we assume that  $f$  has a unique physical measure $\mu$ supported on $K$, then we necessarily  have $\mu=\pi_*\tilde\mu$, where $\pi$ is the natural projection from the tower to $M$ and $\tilde\mu$ is the physical measure for the tower map. Additionally, assuming  $\mu$  mixing, we have $\tilde\mu$  mixing as well. This implies   $\gcd\{R_i\}=1$, which then means that there are $\xi_1,\dots,\xi_n\in\mathcal P$ for which the corresponding return times $R_1,\dots,R_n$ satisfy $\gcd\{R_1\dots, R_n\}=1$. Now, shrinking the disk $\Sigma_0^0$, if necessary, and using the fact that we can choose $\Sigma_\omega^0$ arbitrarily close to $\Sigma_0^0$ for random perturbations sufficiently close to $f$, it is not difficult to see that for each $1\le i\le n$ we can take a domain $\xi_{\omega,i}\in \mathcal P_\omega$ with return time~$R_i$. 
 Observe that since we take only a finite number of domains under these conditions, the estimates on the tail of return times that we prove next remain valid.
\end{remark}

Our goal now is to prove that 
given $C,c>0$  and $0< \tau \le 1$  there exist $C',c'>0$ such that 
\begin{equation}\label{eq.EimpR}
\Leb_{D_\omega}\{\E_\omega>n\}=  Ce^{-cn^\tau} 
\implies
\Leb_{D_\omega}\{R_\omega>n\}= C'e^{-c'n^\tau}.
\end{equation}
First of all
 observe that by construction we have 
 $\mathcal
\{R_\omega>n\}\subset\Sigma_{\omega,n-N_0}^c$.
Hence,  it is enough to have the conclusion of~\eqref{eq.EimpR} with $\Leb_{D_\omega}(\Sigma_{\omega,n-N_0}^c)$ instead of $\Leb_{D_\omega}\{R_\omega>n\}$.
We start with a simple result, which is essentially a consequence of  Proposition~\ref{pr.hyptimes} and Lemma~\ref{le.keyinc}. From here on we fix  $\theta>0$  as in
Proposition~\ref{pr.hyptimes}.

\begin{lemma}\label{le.introX}
Assume that   $\dist_{D_\omega}(x,\partial{\Sigma_\omega^0})> \delta_1\alpha^{{\rho n}/{4}}$ and  $\mathcal E_\omega (x)\le n$ for some $x\in\Sigma_{\omega,n}^c$ with $n\ge 2N_0/\rho$. Then there are ${\theta n}/{2} \le t_1<\dots< t_k\le n$ with $k\ge [{\rho}n/{2}] $ such that
$$
x\in
\bigcap_{i=1}^k{S_{\omega,t_i}({\Sigma_\omega^0})}.
$$
\end{lemma}
\begin{proof}
Consider $n\ge 2N_0/\rho$ and $x\in\Sigma_{\omega,n}^c$ with    $\mathcal E (x)\le n$. If follows from Proposition~\ref{pr.hyptimes} that   $x$ has at least $[\rho n]$
$\alpha$-hyperbolic times between $1$ and $n$, and so  $x$ has at least
$[{{\rho}n}/{2}]$ $\alpha$-hyperbolic times  between ${{\rho}n}/{2}$ and~$n$.  This implies  that there are times $t_1<\dots< t_k\le n$ with  $k\ge [{{\rho}n}/{2}]$  and $t_1\ge {{\rho}n}/{2}\ge N_0$  such that 
\begin{equation}\label{eq.xhtin}
x\in H_{\omega,t_i}\cap \Sigma_{\omega,n}^c,\quad\text{for all $1\le i\le k$.}
\end{equation}
Using Lemma~\ref{le.keyinc} and recalling that $\Sigma_{\omega,n}^c\subset\Sigma_{\omega,t_i}^c$ we may write 
 for all  $1\le i\le k$
\begin{equation}\label{eq.london}
H_{\omega,t_i}\cap\Sigma_{\omega,n}^c\subset H_{\omega,t_i}\cap\Sigma_{\omega,t_i}^c\subset  {S}_{\omega,t_i}({\Sigma_\omega^0})\cup {S}_{\omega,t_i}(\Sigma_\omega^c).\end{equation}
Assuming also that $\dist_{D_\omega}(x,\partial{\Sigma_\omega^0})>  \delta_1\alpha^{{\rho n}/{4}}$,   we have
$\dist_{D_\omega}(x,\partial{\Sigma_\omega^0})>  \delta_1\alpha^{t_i/2}$  for each $1\le i\le k$.
This implies that
\begin{equation}
x\notin {S}_{\omega,t_i}(\Sigma_\omega^c),\quad \text{for all $1\le i\le k$.}
\end{equation}
Together with~\eqref{eq.xhtin} and \eqref{eq.london}, this gives    $x\in {S}_{\omega,t_i}({\Sigma_\omega^0})$ for all $1\le i\le k$.
\end{proof} 

Define for any $k,n \in\mathbb N$  the set
\begin{equation*}
X_{\omega,n}( k)= \left\{x\in\Sigma_{\omega,n}^c\,\mid\,\exists t_1<\cdots<t_k\le n \,: \,x\in
\bigcap_{i=1}^k{{S}_{\omega,t_i}({\Sigma_\omega^0})}\right\}. 
\end{equation*}
By Lemma~\ref{le.introX}, we may write
$$\Sigma_{\omega,n}^c \subset \{  \mathcal E_\omega >n\}\cup \left\{x\in{\Sigma_\omega^0}\mid \dist_{D_\omega}(x,\partial{\Sigma_\omega^0})
\leq \delta_1\alpha^{{\rho}n/{4}}\right\}\cup X_{\omega,n}\left( {\left[\frac{\rho n}2\right]}\right),$$
Observing  that  there exists   $C>0$ such that for all $n\in\mathbb N$  
\begin{equation}\label{in.lebdelta}
\leb_{D_\omega} \left\{x\in {\Sigma_\omega^0} \mid
\dist_{D_\omega}(x,\partial{\Sigma_\omega^0})\leq  \delta_1\alpha^{{\rho n}/{4}}\right\}\leq
C\alpha^{{\rho n}/{4}},
\end{equation}
  the proof of \eqref{eq.EimpR} will be complete  once we have proved that     $\leb_{D_\omega}(X_{\omega,n}( k))$ decays exponentially fast in $k$.
   For this  we need several  auxiliary lemmas that we   prove in the sequel.
Consider $Q_0\ge N_1$ as in~\eqref{eq.chooseq} and take some large   integer $Q_1\ge Q_0$, to be specified  later  in~\eqref{eq.pelinha}.
 Given   $x\in X_{\omega,n}( k)$, consider:
 \begin{itemize}
\item  the moments
  $u_1<\cdots<u_{ p}\le n$ for which~$x$
belongs to some $S_{\omega,u_i+n_i}(\xi_{i})$ with $\xi_{i}\in\mathcal P_{\omega,u_i}$ and $n_i\geq Q_1$;
\item the  moments $v_1<\cdots<v_q\le n$ for which $x$
belongs to some $S_{\omega,v_i+m_i}(\xi_{i})$ with $\xi_{i}\in\mathcal P_{\omega,v_i}$ and $  m_i< Q_1$.  
\end{itemize}
Observe that 
 \begin{equation}\label{eq.sumim}
\sum_{i=1}^p (n_i+1)+\sum_{i=1}^q (m_i+1)\ge k.
\end{equation}
We distinguish   two possible cases:

\begin{enumerate}
\item $\displaystyle\sum_{i=1}^{p}{n_i}\ge \frac k2$.
\\
Defining  for each   $ p \in\mathbb N$ and $n_1,\dots,n_p\ge Q_1$  the set
$$
Y_\omega (n_1,\dots,n_p) = \left\{ x\in{\Sigma_\omega^0} \mid \exists{ u_1<\cdots<u_p\le n \atop \xi_{1}\in\mathcal P_{\omega,u_1},\dots, \xi_{p}\in\mathcal P_{\omega,u_p}} :  
 x\in  
\bigcap_{i=1}^pS_{\omega,u_i+n_i}(\xi_{i})  \right\}
$$
and 
$$Y_{\omega,k}= \bigcup_{n_1,\ldots,n_p\geq
Q_1\atop\sum{n_i}\geq\frac{k}{2}}Y_\omega (n_1,\ldots,n_p ).$$
we   have in this case $x\in Y_{\omega,k}$.

\item  $\displaystyle\sum_{i=1}^{p}{n_i}< \frac k2$.
\\
 Since we assume $n_1,\ldots,n_p\ge Q_1$, we must have in this case $p<k/(2Q_1)$. Using~\eqref{eq.sumim} and the fact that $m_1,\dots,m_q<Q_1$, we may write
$$qQ_1+q\geq \sum_{i=1}^q (m_i+1)\ge k-\sum_{i=1}^p n_i-p \ge  \frac{k}2-\frac{k}{2Q_1}=
\frac{(Q_1-1)k}{2Q_1}\geq\frac{k}{2Q_1},
$$  
which then implies 
\begin{equation*}
q\ge \left[\frac{k}{4Q_1^2}\right].
\end{equation*}
Hence, defining for any positive  integers $n,q$  the set
\begin{equation}\label{def.znq}
Z_{\omega,n}(q) = \left\{ x\in\Sigma_{\omega,n}^c \mid  \exists {v_1<\cdots<v_q \le n \atop \xi_{1}\in\mathcal P_{\omega,v_1},\dots, \xi_{p}\in\mathcal P_{\omega,v_p} }:
 x\in
\bigcap_{i=1}^qS_{\omega,v_i}(\xi_i)  \right\},
\end{equation}
we have in this case $x\in Z_{\omega,n}\left({\left[ k / \left(4Q_1^2\right)\right]}\right)$.
\end{enumerate}

The overall conclusion   is that  
\begin{equation}\label{Z12}
X_{\omega,n}( k)\, \subset\, Y_k
\cup Z_{\omega,n}\left(\left[ \frac{k }{  4Q_1^2 }\right]\right).
\end{equation}
Our goal now is to show that the measure  of   the  sets $Y_{\omega,k}$ and $Z_{\omega,n}(k)$  decays exponentially fast in~$k$. 
We start with a preliminary  estimate on the measure of the sets used  in the definition of~$Y_{\omega,k}$.
 
\begin{lemma}\label{le.Ynk}
There is  $C_0>0$ 
such that for all $n_1,\ldots,n_p>Q_0$ we have
$$\leb_{D_\omega}\left(Y_{\omega}(n_1,\ldots,n_p)\right)\leq
C_0^p \alpha^{( n_1+\cdots +n_p)/2} .
$$
\end{lemma}
\begin{proof}
 Defining for each $ u_1\ge N_0$ and $\xi_1\in\mathcal P_{\omega,u_1}$ the set
 $$Y_{\omega, u_1}^{\xi_1}(n_1,\dots,n_p)= \left\{ x \mid \exists{ u_2<\cdots<u_p \atop \xi_{2}\in\mathcal P_{\omega,u_2},\dots, \xi_{p}\in\mathcal P_{\omega,u_p}} \!:   u_2>u_1,\; x\in  
\bigcap_{i=1}^pS_{\omega,u_i+n_i}(\xi_{i})  \right\},
$$
we may write
$$
Y_{\omega}(n_1,\dots,n_p) = \bigcup_{u_1\ge N_0} \bigcup_{\xi_1\in\mathcal P_{\omega,u_1}} Y_{\omega, u_1}^{\xi_1}(n_1,\dots,n_p).
$$
Noticing  that the elements  $\xi_1\in\mathcal P_{\omega,u_1}$  with $u_1\ge N_0$ are pairwise disjoint, it is enough to show that there is some constant $C_0>0$ such that for all $u_1\ge N_0 $ and $\xi_1\in\mathcal P_{\omega,u_1}$ we have
\begin{equation}\label{eq.ynunion}
\leb_{D_\omega}\left(Y_{\omega,u_1}^{\xi_1}(n_1,\dots,n_p)\right)\leq
C_0^p \alpha^{ (n_1+\cdots +n_p)/2}\leb_{D_\omega}(\xi_1).
\end{equation}
We shall  prove \eqref{eq.ynunion} by induction on  $p$. Considering  $C_0>C $, where $C>0$ is the constant in Lemma~\ref{le.sncore},   we immediately get
the result for $p=1$.
Now suppose that $p>1$. 
We may write
\begin{equation}\label{eq.Ynsomega}
Y_{\omega,u_1}^{\xi_1}(n_1,\dots,n_p)=\bigcup_{u_2> u_1} \bigcup_{\xi_2\in\mathcal P_{\omega,u_2}}Y_{\omega,u_2}^{\xi_2}(n_2,\dots,n_p)
\end{equation}
 By Lemma~\ref{le.existQ}, for all $\xi_2\in \mathcal P_{\omega,u_2}$  with  $u_2>u_1$  we have
\begin{equation}\label{eq.intertwoQ}
S_{\omega,u_2+Q_0} (\xi_1)\cap S_{\omega,u_2+Q_0} (\xi_2)=\emptyset.
\end{equation}
Assuming with no loss of generality  that $Y_{\omega, u_1}^{\xi_1}(n_1,\dots,n_p)$  is nonempty, we have  in particular
\begin{equation}\label{eq.intertwo}
 S_{\omega,u_1+n_1} (\xi_1) \cap
 S_{\omega,u_2+n_2} (\xi_2)\neq\emptyset.
\end{equation}
From~\eqref{eq.intertwoQ} and~\eqref{eq.intertwo} we easily deduce that $u_1+n_1< u_2+Q_0,$ or equivalently  $u_2-u_1>n_1-Q_0$. Observe that 
$$
\diam \left(f_\omega^{u_2+\ell_2}(S_{\omega,u_2+n_2} (\xi_2))\right)\le 2\delta_0,
$$
and so, by  \eqref{eq.defK0}
$$
\diam \left(f_\omega^{u_2}(S_{\omega,u_2+n_2} (\xi_2))\right)\le 2\delta_0 K_0.
$$
Then, using~\eqref{eq.sncontido} and Proposition~\ref{pr.predisks},
we get
\begin{equation}\label{eq.ineqs12}
\diam \left(f_\omega^{u_1}(S_{\omega,u_2+n_2} (\xi_2))\right)\leq \alpha^{\frac{u_2-u_1}{2}}
\diam \left(f_\omega^{u_2}(S_{\omega,u_2+n_2} (\xi_2))\right)\leq
2\delta_0K_0\alpha^{\frac{n_1-Q_0}{2}}.
\end{equation}

Consider now $\beta_\omega$   the outer component of the boundary of $f_\omega^{u_1}(S_{\omega,u_1+n_1} (\xi_1))$ and $\mathcal N_\omega$  a neighbourhood   of  $\beta_\omega$ of size $2\delta_0K_0\alpha^{( {n_1-Q_0})/{2}}$ inside $f_\omega^{u_1}(D_\omega)$. Since  $n_1\ge Q_0\ge N_1$, by definition  of  $S_{\omega,u_1+n_1} (\xi_1)$ there is $0\le \ell\le L$ such that $f^{\ell}(\beta_\omega)$ coincides with the boundary of a ball of radius   not exceeding $2\delta_0$ centred at~$p$. Then, using~\eqref{eq.defK0}  we deduce   that the volume of the submanifold $\beta_\omega$ is uniformly bounded. Hence, there must be some   constant $C>0$ (not depending on $u_1\ge N_0$ or on $\xi_1\in\mathcal P_{\omega,u_1}$) such that
 \begin{equation}\label{eq.nomega}
\leb_{f_\omega^{u_1}(D_\omega)}\left(\mathcal N_\omega \right)\le C\alpha^{n_1/2}.
\end{equation}
Now,  since $f_\omega^{u_1}(S_{\omega,u_2+n_2} (\xi_2))$ intersects $\beta_\omega$,  it follows  from~\eqref{eq.ineqs12}  that $f_\omega^{u_1}(S_{\omega,u_2+n_2} (\xi_2))$ is contained in  $\mathcal N_\omega$. This in particular implies that
 \begin{equation}\label{eq.incfu1no}
f_\omega^{u_1}(\xi_2)\subset \mathcal N_\omega.
\end{equation}
On the other hand, by the induction hypothesis we have
 \begin{equation*}
 \leb_{D_\omega}  \left(Y_{\omega,u_2}^{\xi_2}(n_2,\dots,n_p)\right)
  \leq C_0^{p-1} \alpha^{(n_2+\cdots +n_p)/2}\leb_{D_\omega}  (\xi_2),
 \end{equation*}
Using~\eqref{eq.sncontido} and the definition of $Y_{\omega,u_1}^{\xi_1}(n_1,\dots,n_p)$, we easily deduce that   
$$
Y_{\omega,u_1}^{\xi_1}(n_1,\dots,n_p)\subset S_{\omega,u_1+n_1}(\xi_1)\subset V_{u_1}(\xi_1).
$$
Considering the constant $C_1>0$ given by Proposition~\ref{p.distortion}, we get
 \begin{equation*}
 \leb_{f_\omega^{u_1}(D_\omega)}   \left(f_\omega^{u_1}\left( Y_{\omega,u_2}^{\xi_2}\left(n_2,\dots,n_p\right)  \right)\right)
  \leq C_1C_0^{p-1} \alpha^{(n_2+\cdots +n_p)/2}\leb_{f_\omega^{u_1}(D_\omega)} \left(f_\omega^{u_1}(\xi_2)\right).
 \end{equation*}
 Observe that the sets $\xi_2\in \mathcal P_{\omega,u_2}$ with $u_1<u_2 $ are pairwise disjoint. Moreover, 
 by~\eqref{eq.incluisses} these sets are all contained in $V_{\omega,u_1}(\xi_1)$. Since $f_\omega^{u_1}$ is injective on $V_{\omega,u_1}(\xi_1)$, we easily get that the sets $f_\omega^{u_1}(\xi_2)$ with $\xi_2\in \mathcal P_{\omega,u_2}$ and $u_1<u_2 $  are also pairwise disjoint. Then, using~\eqref{eq.Ynsomega}, \eqref{eq.nomega} and \eqref{eq.incfu1no}
we may write
 \begin{align*}
 \leb_{f_\omega^{u_1}(D_\omega)}   \left(f^{u_1 }\left(Y_{\omega,u_1}^{\xi_1}(n_1,\dots,n_p)\right)\right) &\leq
 \sum_{u_2=u_1+1}^{n}\sum_{\xi_2\in\mathcal P_{\omega,u_2}}\leb_{f_\omega^{u_1}(D_\omega)}  \left(f_\omega^{u_1 }\left(Y_{\omega,u_2}^{\xi_2}(n_2,\dots,n_p)\right)\right)\\
&\leq {C_1}C_0^{p-1} \alpha^{(n_2+\cdots +n_p)/2} \sum_{u_2=u_1+1}^{n}\sum_{\xi_2\in\mathcal P_{\omega,u_2}}\leb_{f_\omega^{u_1}(D_\omega)} 
 \left(f_\omega^{u_1}(\xi_2)\right)\\
 &\leq {C_1}C_0^{p-1} \alpha^{(n_2+\cdots +n_p)/2}\leb_{f_\omega^{u_1}(D_\omega)}  \left(\mathcal N_\omega\right)\\
&\leq {CC_1}C_0^{p-1} \alpha^{(n_1+\cdots +n_p)/2} .
 \end{align*}
Taking $C_0\geq CC_1$, we finish the proof.
\end{proof}

Observe that the constant $C_0$ given by Lemma~\ref{le.Ynk} in principle depends on the integer $Q_0$ introduced in~\eqref{eq.chooseq}, but not  on $Q_1\ge Q_0$. Since this holds, we may  choose $Q_1$ and integer sufficiently large such that 
\begin{equation}\label{eq.pelinha}
\alpha^{1/2}+C_0\alpha^{Q_1/2}<1.
\end{equation}
The next result gives the expected  estimate on  the measure of the set $Y_k $. 

\begin{proposition}\label{pr.Ynk}
There are  $C_1>0$ and $ \lambda_1<1$
such that for all 
 $k\ge 1$ we have
$$\leb_{D_\omega}\left(Y_{\omega,k}\right)\leq C_1\lambda_1^k .$$
\end{proposition}
\begin{proof}
By the definition of $Y_{\omega,k}$ and Lemma~\ref{le.Ynk} we just need to show that 
$$\sum_{n_1,\ldots,n_p\geq
Q_1\atop\sum{n_i}\geq\frac{k}{2}}C_0^p \alpha^{( n_1+\cdots +n_p)/2}
$$
decays  exponentially fast  with $k$.
We use the generating series
\begin{eqnarray*}
\sum_{n\ge 1}\sum_{n_1,\ldots,n_p  \geq  
Q_1\atop\sum{n_i}=n}C_0^p \alpha^{( n_1+\cdots +n_p)/2}z^n
&=&\sum_{p=1}^{\infty}\left(C_0\sum_{n=Q_1}^{\infty}\alpha^{n/2}z^n
 \right)^p\\
 &=&\frac{C_0\alpha^{Q_1/2}z^{Q_1}}{1-{\alpha^{1/2}}z-C_0\alpha^{Q_1/2}z^{Q_1}}.
\end{eqnarray*}
Under condition \eqref{eq.pelinha}, the function above has no  pole in a neighborhood of the unit disk
in~$\mathbb{C}$. Thus, the coefficients of its power series decay
exponentially fast: there are constants $C_1>0$ and $\lambda_1<1$
such that
$$\sum_{n_1,\ldots,n_p\geq
Q_1,\atop\sum{n_i}=n}C_0^p \alpha^{( n_1+\cdots +n_p)/2}\leq
C_1\lambda_1^n,
$$
and so we are done.
\end{proof}

At this point we introduce  some  simplification in the notation. Consider 
$$\mathcal T_\omega=\left\{ S_{\omega,n}(\xi)\cup\xi \mid \text{for some $\xi \in \mathcal P_{\omega,n}$ and $n \ge N_0$}\right\}.
$$
Given  $T=S_{\omega,n}(\xi)\cup \xi\in\mathcal T_\omega$  with $\xi\in\mathcal P_{\omega,n}$, define
 $$\xi_T=\xi\qand t(T)=n.$$
We will  refer to $\xi_T$ as the \emph{core} of $T\in \mathcal T_\omega$. Notice  that for any  $T_1,T_2\in\mathcal T_\omega$  we 
 have 
 \begin{equation}\label{eq.cores1}
T_1\neq T_2\implies  \xi_{T_1}\cap \xi_{T_2}=\emptyset,
\end{equation}
and from~\eqref{eq.incluisses} it follows  that  
\begin{equation}\label{eq.cores2}
T_1\cap T_2\neq\emptyset\implies T_2\subset V_{\omega,t(T_1)}(\xi_{T_1})
\end{equation}
Finally,
from 
Lemma~\ref{le.sncore} we easily deduce the existence of $C>0$ such that  for every $T\in\mathcal T_\omega$ we have
\begin{equation}\label{eq.satcore} \leb_{D_\omega}\left(T\right) \leq C\leb\left(\xi_T\right).
\end{equation}

Our goal now is to prove that the measure of $Z_{\omega,n}(q)$ decays exponentially fast with~$q$. For that  we need a couple of  auxiliary lemmas that we   prove in the sequel.
We start by fixing some large integer 
$Q_2$ 
 to be specified  in~\eqref{eq.Qduas}.  Given  
 $x\in Z_{\omega,n}(q)$, consider  the times  $v_1<\cdots<v_q\le n$ as
in the definition of $Z_{\omega,n}(q)$   in~\eqref{def.znq}. Take the smallest positive integer $u$ such that $Q_2u\ge n$ and, for each $1\le i\le u$, consider 
  from  the 
interval $( (i-1)Q_2,iQ_2]$ the first  element in $\{v_1,\dots,v_q\}$, if there is at least one. Denote the subsequence of those elements  by  $w_{1}<\dots<w_{p}$. We necessarily have $p\ge \left[  q/Q_2\right],$ and so
$(p+1)Q_2 \geq q$. Keeping only the elements  with odd indexes, we get a
sequence  $t_1<\ldots<t_\ell$ with $2\ell\geq p$. This implies that
\begin{equation}\label{eq.ellqQ2}
\ell\geq \frac{q-Q_2}{2Q_2}.
\end{equation}
  Moreover, by construction we have 
  \begin{equation}\label{eq.tdifer}
t_{i+1}-t_i\geq Q_2,\quad\text{for each $1\le
i\le\ell$.}
\end{equation}
According to the definition of  $Z_{\omega,n}(q)$,  for each  $1\le i\le\ell$
there is some  $T_i\in \mathcal T_\omega$ such that  $t(T_i)=t_i$.
Set  $$I=\{1\leq i\leq\ell \mid T_i\subset T_1\cap\dots\cap T_{i-1}\}.$$
Now we consider  the   two possible cases:

\begin{enumerate}
\item  $\# I\geq \ell/2$.
\\
Define for any positive integers $n,k\ge 1$ the set
\begin{equation*}\label{eq.ze2}
Z_{\omega,n}^0(k)  = \left\{ x\in\Sigma_{\omega,n}^c\, |\,\exists T_1 \supset 
\cdots\supset  T_k\text{ with $t(T_1)<\cdots< t(T_k) \le n$ and $x\in
 T_{\omega,k}$} \right\}.
\end{equation*}
Keeping only  
elements with indexes in $I$ and recalling~\eqref{eq.ellqQ2},  we  easily see that in this case we have 
$x\in Z^0_{\omega,n}\left({\left[   {(q-Q_2)}/{(4Q_2)} \right]}\right).$

\smallskip

\item   $\# I<\ell/2$.

Considering  $J=\{1,\dots,\ell\}\setminus I$, we necessarily have   $\# J \geq \ell/2$. We define
$$j_0=\sup J\qand i_0=\inf\left\{i<j_0\mid T_{j_0}\not\subset 
 T_i\right\}.$$ Next we define 
 $$j_1=\sup\{j\leq i_0\mid j\in J\}\qand i_1=\inf\left\{i<j_1\mid T_{j_1}\not\subset
 T_i\right\}.$$ Proceeding inductively, the process must necessarily stop at some    $i_{m_0}$. By construction, we have 
   $$J\subset\bigcup_{s=0}^{m_0}\{i_s+1,\dots,j_s\},$$ which then gives 
    $$
    \sum_{s=0}^{m_0}(j_s-i_s)\geq\# J\geq
     \frac\ell2.
     $$
     On the other hand, it follows from~\eqref{eq.tdifer} that   $|t_{j_s}-t_{is}|\ge Q_2(j_s-i_s)$ for all $0\le s\le m_0$.
Altogether, this yields
 $$\sum_{s=0}^{m_0} \frac{t(T_{j_s})-t(T_{i_s})}{Q_2} =\sum_{s=0}^{m_0}\frac{t_{j_s}-t_{i_s}}{Q_2} \geq
 \frac\ell2\ge  \frac{q-Q_2}{4Q_2}.
 $$  
 \end{enumerate}

Motivated by this second case,      we consider  for each     $T\in\mathcal T_\omega$
the set    $\mathcal F(T) $ of   finite sequences   $ (T_0, \dots,T_{2m})\in\mathcal T^{2m+1}_\omega$ with $m\ge 0$ and $T_0=T$ such that 
   $ t(T_{2i-2} )\leq
t(T_{2i-1}) \leq t(T_{2i})-Q_2$  and $T_{2i} \not\subset T_{2i-1}$  for all $1\le i\le m$. 
Then we define for each $k\ge 0$ and $T \in\mathcal T_\omega$  the sets
 \begin{equation*}
 Z_\omega(k,T)=
  \left\{x\in{\Sigma_{\omega,n}^0}\mid  
\exists (T_{0},\dots,T_{2m})\in\mathcal F(T) : \sum_{i=1}^{m}\frac{t(T_{2i})-t(T_{2i-1})}{Q_2} \geq k
\text{ and } \, x \in  
 \bigcap_{i=0}^{2m}T_{i}\,\right\}. \label{eq.ze1}
\end{equation*}
and 
$$ Z^1_\omega(k)=\bigcup_{T\in\mathcal T_\omega}Z_\omega(k,T).$$
The considerations  in the case $\# I<\ell/2$  above show that   
$x\in Z_\omega \left(\left[   {(q-Q_2)}/{(4Q_2)} \right], T\right),$
with  $T=T_{i_{m_0}}$ and   sequence $
(T_{i_{m_0}},T_{i_{m_0}},T_{j_{m_0}},\dots,T_{i_0},T_{j_0})\in\mathcal F( T)$. 
Hence,    we have 
\begin{equation}\label{eq.incluiznk}
 Z_{\omega,n}(q)\subset Z^0_{\omega,n}\left(\left[  \frac{q-Q_2}{4Q_2} \right]
\right)\cup Z_\omega^1\left(\left[   \frac{q-Q_2}{4Q_2} \right]\right).
\end{equation}
Our final goal   is to show that the measure of the  sets   $Z_{\omega,n}^0(k) $ and  $Z_\omega^1(k )$ decays exponentially fast with $k$.


\begin{lemma}\label{le.Z^0}
There exists $\lambda_2<1$ such that for all $k\ge 1$   we have
$$\leb_{D_\omega}\left(Z^0_{\omega,n}(k)\right)\leq \lambda_2^k\leb_{D_\omega}\left({\Sigma_{\omega,n}^0}\right).$$
\end{lemma}

\begin{proof}
We define $\mathcal {T}_{\omega,1}$ as the class of elements  $T\in\mathcal T_\omega$ with $t(T)\le n$   not contained in any other element of~$\mathcal T_\omega$. Then, we define
$\mathcal T_{_\omega,2}$ as the class of elements  $T\notin \mathcal T_{\omega,1}$ with $t(T)\le n$ contained  in elements of~$\mathcal T_{\omega,1}$ and not contained in any other elements of $\mathcal T_\omega$.
Proceeding inductively, this process must stop in a finite number of steps.   We say that  the  elements in $\mathcal T_{\omega,i}$ have 
\emph{rank}~$i$.
Then we define $$C_{\omega,k}=\bigcup_{i=1}^{k}\bigcup_{T\in \mathcal T_{\omega,i}}\xi_T,$$
and
$$\tilde{Z}_{\omega,k}=\left(\bigcup_{T\in{\mathcal T_k}}T\right)\setminus C_{\omega,k}.$$

We claim  that $Z^0_{\omega,n}(k)\subset\tilde Z_{\omega,k}$. Actually, given any $x\in
Z_{\omega,n}^0(k)$, we have $T_1,   \dots, T_k\in \mathcal T_\omega$ with $T_1\supset    \cdots\supset  T_k$ and  $t(T_1)<\cdots <t(B_k)\leq n$ such that $x\in T_k\cap\Sigma_{\omega,n}^c$. Clearly, $T_k$ has rank $r\geq k$. Take $T'_1\supset \cdots\supset  T'_r$  with $T'_i
\in \mathcal T_{\omega,i}$ and $T'_r=T_k$. In particular, $x\in T'_i$ for
$i=1,\dots,k$, and so $x\in \bigcup_{T\in\mathcal T_{\omega,k}}T$. On the
other hand, as  $x\in\Sigma_{\omega,n}^c$ and $C_{\omega,k} \cap\Sigma_{\omega,n}^c=\emptyset$,
we get $x \notin C_{\omega,k}$. This gives    $x\in \tilde{Z}_{\omega,k}$.

Now we show that for each $k$ we have
\begin{equation}\label{eq.ztildes}
\leb_{D_\omega}\left(\tilde{Z}_{\omega,k+1}\right)\le \frac{C}{C+1}\leb_{D_\omega}\left(\tilde{Z}_{\omega,k}\right),
\end{equation}
where $C>0$ is the constant given in~\eqref{eq.satcore}.
 To see this, we start by showing that for all $T \in \mathcal T_{k+1}$  we have
 \begin{equation}\label{eq.coresubz}
\xi_T\subset
\tilde{Z}_{\omega,k}\setminus \tilde{Z}_{\omega,k+1}.
\end{equation}
Take any $T \in \mathcal T_{\omega,k+1}$ and let $T'\in\mathcal T_\omega$ be
an element of rank $k$ containing $T$. As the cores are pairwise
disjoint, we have $\xi_T\cap C_{\omega,k}=\emptyset$, and so $\xi_T\subset
T'\setminus C_{\omega,k}\subset \tilde{Z}_{\omega,k}$. As, by definition, we have $\xi_T\subset
C_{\omega,k+1}$, it follows that $\xi_T\cap
\tilde{Z}_{\omega,k+1}=\emptyset$. This gives \eqref{eq.coresubz}. 

Let us now prove~\eqref{eq.ztildes}. Using~\eqref{eq.satcore} and the fact that the cores of the elements in $\mathcal T$ are pairwise disjoint, we may write
\begin{eqnarray*}
  \leb_{D_\omega}\left(\tilde{Z}_{\omega,k+1}\right) &\leq &\sum_{T\in
\mathcal{T}_{\omega,k+1}}\leb_{D_\omega}\left(T\right)\\ 
& \le &C\sum_{T\in\mathcal T_{\omega,k+1}}\leb_{D_\omega}\left(\xi_T\right)\\
                            & \le &C\leb_{D_\omega}\left(\tilde{Z}_{\omega,k}\setminus \tilde{Z}_{\omega,k+1}\right).
\end{eqnarray*}
Therefore,
\begin{eqnarray*}
  (C+1)\leb_{D_\omega}\left(\tilde{Z}_{\omega,k+1}\right) &\le& C\leb_{D_\omega}\left(\tilde{Z}_{\omega,k+1}\right)+C\leb_{D_\omega}\left(\tilde{Z}_{\omega,k}\setminus \tilde{Z}_{\omega,k+1}\right) \\
    &=& C\leb_{D_\omega}\left(\tilde{Z}_{\omega,k}\right).
\end{eqnarray*}
Inductively, this yields 
 $$
 \leb_{D_\omega}\left(\tilde{Z}_{\omega,k}\right)\leq
\left(\frac{C}{C+1}\right)^k\leb_{D_\omega}({\Sigma_{\omega,n}^0}).$$  
Since $Z_{\omega,k}\subset \tilde{Z}_{\omega,k}$, the result follows with $\lambda_2=C/(C+1)$. 
\end{proof}

\begin{lemma}\label{le.zkt}
There is $C_2>0$   such that for all $k\ge 0$ and  $T\in\mathcal T_\omega$ we have
$$
\leb_{D_\omega}\left(Z_\omega(k,T) \right)\leq C_2\left(C_2 \alpha^{Q_2/2}\right)^{k}\leb_{D_\omega}\left(\xi_T\right).
$$
\end{lemma}

\begin{proof}
We   prove the result by induction on $k\ge 0$.
For $k=0$, we have by~\eqref{eq.satcore}
\begin{equation*}\label{eq.z1}
 \leb_{D_\omega}(Z_\omega(0,T))\le \leb_{D_\omega}(T)\leq C\leb_{D_\omega}(\xi_T).
\end{equation*}
So, in this case it is enough to take
$
C_2\ge C.
$
Given any $k\geq 1$, we may write
\begin{eqnarray*}
Z_\omega(k,T)\subset \bigcup_{m=1}^{k}\bigcup_{T_1\cap T\neq
\emptyset}\bigcup_{T_1\cap T_2\neq \emptyset,T_2\not\subset
T_1\atop \frac{t(T_2)-t(T_1)}{Q_2}  \geq m}
Z_\omega(k-m,T_2).\end{eqnarray*}
Our goal now is to estimate the measure of the triple union above. 

Take   $T_1\in\mathcal T_\omega$ intersecting $T$ and   $T_2\in\mathcal T_\omega$ intersecting $T_1$ with $T_2\not\subset
T_1 $ and $  {t(T_2)-t(T_1)}  \geq mQ_2$. 
Consider  $t_1=t(T_1)$ and $t_2=t(T_2)$. Notice  that  for each $i=1,2$, the set  $f_\omega^{t_i}(T_i)=f_\omega^{t_i}(V_{t_i}^{1/3}(\xi_{T_i}))$ is, by definition,   a ball of radius $\delta_1/3$.
Hence, recalling~\eqref{eq.sncontido} and using Proposition~\ref{pr.predisks}, we may write
 \begin{equation}\label{eq.diaminho}
\diam \left(f_\omega^{t_1}(T_2)\right)\leq
{{\alpha^{\frac{t_2-t_1}{2}}}{\diam\left(f_\omega^{t_2}(T_2)\right)}}\leq
{\frac23{\delta_1}{\alpha^{\frac{t_2-t_1}{2}}}}\leq
{\frac23{\delta_1}{\alpha^{\frac{mQ_2}{2}}}}.
\end{equation}
 Now, let  $\beta_\omega$ be   the boundary of the ball $f_\omega^{t_1}(T_1)$ and consider  $\mathcal N_\omega$ a neighbourhood   of  $\beta_\omega$  of size $2{\delta_1}{\alpha^{ {mQ_2}/{2}}}/3$ inside $f_\omega^{t_1}(D_\omega)$. Notice that  $\mathcal N_\omega$  is  necessarily contained in $f_\omega^{t_1}\left(V_{t_1}(\xi_{T_1})\right)$. Let   $U_\omega$ be the subset of $ V_{\omega,t_1}(\xi_{T_1})$ which is sent diffeomorphically   onto  $\mathcal N_\omega$ under~$f_\omega^{t_i}$. 
One can easily see that there exists some constant $\tilde C>0$ (not depending on $t_1\ge N_0$ nor on $T_1$) such that
 \begin{equation}\label{eq.u1meas}
\leb_{f_\omega^{t_1}(D_\omega)}\left(f_\omega^{t_1}(U_\omega )\right)\le \tilde C\alpha^{\frac{mQ_2}{2}}\leb_{f_\omega^{t_1}(D_\omega)}\left(f_\omega^{t_1}(T_1)\right).
\end{equation}
Recalling that $T_1\subset V_{\omega,t_1}(\xi_{T_1})$, considering $C_1>0$  the constant given by Proposition~\ref{p.distortion}  and using~\eqref{eq.satcore}, we obtain
 \begin{equation}\label{eq.u1omega}
\leb_{D_\omega}\left( U_1 \right)\le C_1\tilde C\alpha^{\frac{mQ_2}{2}}\leb_{D_\omega}\left( T_1\right) \le  C_1\tilde CC\alpha^{\frac{mQ_2}{2}}\leb_{D_\omega}\left( \xi_{T_1}\right).
\end{equation}
Now,  as $T_2\cap T_1\neq\emptyset$ with $T_2\not\subset
T_1 $, it follows that  $f_\omega^{t_1}(T_2)$    intersects~$\beta_\omega$. Then, using~\eqref{eq.diaminho}, we easily deduce   that $f_\omega^{t_1}(T_2)$ is contained in  $\mathcal N_\omega= f_\omega^{t_1}(U_\omega)$. This gives in particular  
$ \xi_{T_2}\subset U_\omega.
$
Since the cores of these possible $T_2$ are pairwise disjoint, from~\eqref{eq.u1omega} we get
\begin{equation}\label{eq.fund1}
\sum_{T_1\cap T_2\neq \emptyset,T_2\not\subset
T_1\atop \frac{t(T_2)-t(T_1)}{Q_2}  \geq m}\leb_{D_\omega}\left(\xi_{T_2}\right)  \le  C_1\tilde CC\alpha^{\frac{mQ_2}{2}}\leb_{D_\omega}\left( \xi_{T_1}\right).
\end{equation}
On the other hand, since  $T_1\cap T\neq\emptyset$, it follows from~\eqref{eq.cores2} that $T_1\subset V_{\omega,t(T)}(\xi_T)$.
Moreover, as the cores of these possible $T_1$ are pairwise disjoint,  using~\eqref{eq.sumle} we may write
\begin{equation}\label{eq.fund2}
 \sum_{T_1\cap T\neq
\emptyset}\leb_{D_\omega} \left(\xi_{T_1}\right)\leq \leb_{D_\omega}\left(V_{\omega,t(T)}(\xi_T)\right) \le C_0\leb_{D_\omega}\left(\xi_T\right).
\end{equation}
Finally, using \eqref{eq.fund1}, \eqref{eq.fund2} and the inductive hypothesis, we may write
\begin{eqnarray*}
\leb_{D_\omega}\left(Z_\omega(k,T)\right)&\leq& 
\sum_{m=1}^{k}\sum_{T_1\cap T \neq
\emptyset}\sum_{T_1\cap T_2\neq \emptyset,T_2\not\subset
T_1\atop\left[\frac{t(T_2)-t(T_1)}{Q_2}\right]
\geq m}\leb_{D_\omega}\left(Z_\omega(k-m,T_2)\right)\\
&\leq& 
\sum_{m=1}^{k}\sum_{T_1\cap T \neq
\emptyset}\sum_{T_1\cap T_2\neq \emptyset,T_2\not\subset
T_1\atop\left[\frac{t(T_2)-t(T_1)}{Q_2}\right]
\geq m}C_2\left(C_2 \alpha^{Q_2/2}\right)^{k-m}\leb_{D_\omega}\left(\xi_{T_2}\right)\\
&\leq& \sum_{m=1}^{k}\sum_{T_1\cap T\neq \emptyset}C_2(C_2
\alpha^{{Q_2}/{2}})^{k-m} C_1\tilde CC\alpha^{{mQ_2}/{2}}\leb_{D_\omega}\left(\xi_{T_1}\right)\\
&\leq& C_2(C_2\alpha^{{Q_2}/{2}})^{k} C_0C_1\tilde C C
\sum_{m=1}^{k}C_2^{-m}\leb_{D_\omega}\left(\xi_T\right).
\end{eqnarray*}
Hence, taking  $C_2>0$ sufficiently large enough for
$$
\frac{C_0C_1\tilde CCC_2^{-1}}{1-C_2^{-1}}\leq1,$$
we finish the proof.
\end{proof}

At this point we are able to specify the choice of $Q_2$. We  take $Q_2$ an integer 
sufficiently  large such that  
\begin{equation}\label{eq.Qduas}
C_2\alpha^{Q_2/2}<1,
\end{equation}
with $C_2>0$ as in 
Lemma~\ref{le.zkt}.
Under this choice, we easily deduce that   $\leb_{D_\omega}(Z_\omega^1(k))$ decays exponentially fast in $k$.

%

\subsection{Hyperbolic product structure}\label{s.product}
Here we obtain a set $\Lambda_\omega$ with a hyperbolic product structure verifying  properties (P$_0$)-(P$_5$) as in Subsection~\ref{se.hps}. 
Consider the center-unstable disk $\Sigma_\omega^0\subset D_\omega$ as in \eqref{eq.sigma01} and the $\leb_{D_\omega}$ mod 0
partition $\mathcal P_\omega$ of $\Sigma_\omega^0$ defined in
Section~\ref{s.structure}. We define
$$\Gamma_\omega^s=\left\{ W^s_{\omega,\delta_s}(x):\,x\in \Sigma_\omega^0\right\}.$$
Moreover, we define
 $\Gamma_\omega^u$ as the set of all  local unstable manifolds contained in $\mathcal C_\omega^0$ which $u$-cross~$\mathcal C_\omega^0$. Clearly,
$\Gamma_\omega^u$ is nonempty because $\Sigma_\omega^0\in \Gamma^u$. We need to see
 that the union of the leaves in $\Gamma_\omega^u$ is compact. 
 By the domination
property and Ascoli-Arzela Theorem, any limit leaf $\gamma_\infty$
of leaves in $\Gamma_\omega^u$ is still a $cu$-disk $u$-crossing $\mathcal C_0$. So, by definition of $\Gamma_\omega^u$, we have
$\gamma_\infty\in\Gamma_\omega^u$. We   define our set \( \Lambda_\omega\) with hyperbolic product structure as the intersection of these families of stable and unstable leaves.
The cylinders $\{\mathcal C(\xi )\}_{\xi \in\mathcal P_\omega}$ then clearly form a countable collection of
\(s\)-subsets of \( \Lambda_\omega\) that can be used as the sets    $\Lambda_{1,\omega},\Lambda_{2,\omega},\dots$ in (P$_1$)
with the corresponding return times \( R_\omega(\xi )\), for each $\xi\in\mathcal P_\omega$.
Since (P$_0$) is obvious, we are left  to verify  conditions (P$_1$)-(P$_5$).

\subsubsection{Markov and contraction on stable leaves}
Condition (P$_1$) is essentially an immediate consequence of the construction.  We just need to check that $f_\omega^{R_\omega(\xi)}(\mathcal C(\xi))$ is a $u$-subset, for any $\xi\in\mathcal P_\omega$. In fact, choosing the integer $N_0$ in the first step of the construction of $\mathcal P_\omega$  sufficiently large, and using the fact that the local stable manifolds are uniformly contracted by forward iterations under  $f_\omega$, we can easily see that the ``height'' of $f_\omega^{R_\omega(\xi)}(C(\xi))$ is at most $\delta_s/4$, for each $\xi\in\mathcal P_\omega$. On the other hand, by the choice of $\delta_0$ we have that $f^{R_\omega(\xi)}_\omega(\xi)$ intersects $W_{\sigma^{R_\omega(\xi)}\omega, \delta_s/2}^s(y)$, for some $y\in\Sigma_{\sigma^{R_\omega(\xi)}\omega}^0$, and $u$-crosses $\mathcal C^1_{\sigma^{R_\omega(\xi)}\omega}$; recall Remark~\ref{la belle}.
Hence, the local unstable leaves that form $\mathcal C(\xi)$ necessarily  $u$-cross~$\mathcal C_0$, and so  (P$_1$) holds.
 Property (P$_2$) is clearly verified in our case.

\subsubsection{Expansion and bounded distortion}
The expansion and bounded distortion on unstable leaves, property (P$_3$), follows from Proposition~\ref{pr.predisks} and Proposition~\ref{p.distortion}. Indeed, by construction, for each $\xi\in\mathcal P_\omega$  there is a
hyperbolic pre-ball $V_{\omega,n(\xi)}(x)$ containing $\xi$
associated to some point $x\in \Sigma_\omega^0$ with $\alpha$-hyperbolic time
$n(\xi)$.
Therefore, recalling Remark~\ref{re.bounded} and taking $\delta_s <\delta_1$, for any $\gamma\in \Gamma_\omega^u$ we have that $n$ is a $\alpha^{1/4}$-hyperbolic
time for every point in $\mathcal C(\xi)\cap\gamma$. 
Since $R_\omega(\xi)-L\le n(\xi)\le R_\omega(\xi)$, choosing $N_0$ such that $K_0\alpha^{N_0/2}<1$, by~\eqref{eq.defK0} we obtain (P$_3$).

\subsubsection{Regularity of the foliations}
\label{sec.regularity}
Property (P$_4$)  is standard for uniformly hyperbolic attractors. In the rest of this section we shall adapt classical ideas to our setting.
We begin with the statement of a useful lemma on vector bundles whose proof can be found in \cite[Theorem 6.1]{HP70}. We should say that a metric $d$ on a vector bundle $E$ is \emph{admissible} if there is a complementary bundle $E'$ over $X$, and an isomorphism $h\colon E\oplus E'\to X\times B$ to a product bundle, where $B$ is a Banach space, such that $d$ is induced from the product metric on $X\times B$.

\begin{lemma}\label{th.hirsch-pugh}
Let $p\colon E\to X$ be a vector bundle over a metric space $X$ endowed with an admissible metric. Let $D\subset E$ be the unit  ball bundle, and $F\colon D\to D$ a map covering a Lipschitz homeomorphism $f\colon X\to X$. Assume  that there is $0\le \kappa<1$ such that for each $x\in X$ the restriction $F_x\colon D_x\to D_x$ satisfies $\lip(F_x)\le \kappa$. Then
\begin{enumerate}
  \item there is a unique section $\mathfrak s\colon X\to D $ whose image is invariant under $F$;
  \item if  $\kappa \lip(f)^\eta<1$ for some $0<\eta\le 1$, then $\mathfrak s$ is H\"older condition with exponent~$\eta$.
\end{enumerate}
\end{lemma}

\begin{proposition}\label{th.Holder}
Assume  the random perturbations $f_\omega$ sufficiently close to $f$ in the $C^1$ topology.
If    $T_{K_\omega}M=E^{cs}_\omega \oplus E^{cu}_\omega $ is a dominated splitting, then the fibre bundles $E^{cs}_\omega$ and $E^{cu}_\omega$ are H\"older continuous on $K_\omega$.
\end{proposition}

\begin{proof} We consider only the centre-unstable bundle as the other one is similar.
For each $x\in K_\omega$, let $L_{\omega,x}$ be the space of bounded linear maps from $E_{\omega,x}^{cu}$ to $E_{\omega,x}^{cs}$ and let $L_{\omega,x}^1$ denote the unit ball around $0\in L_{\omega,x}$. Define  $G_{\omega,x}: L_{\omega,x}^1\to L_{\sigma\omega,f_\omega(x)}^1$   the graph transform induced by
 $Df_\omega(x)$ 
    $$G_{\omega,x}(\phi)=(Df_\omega \vert E_{\omega,x}^{cs})\,\phi\,(Df_{ \omega}^{-1} \vert E^{cu}_{\sigma\omega,f_\omega(x)}).$$
Consider $L_\omega$ the vector bundle over $K_\omega$ whose fibre over each $x\in K_\omega$ is $L_{\omega,x}$, and let $L_\omega^1$ be its unit   bundle. Then $G_\omega : L_\omega^1\to L_{\sigma\omega}^1$ is a bundle map covering $f_\omega\vert K_\omega$ with
 $$\lip(G_{\omega,x})\le  \|Df_\omega \vert E_{\omega,x}^{cs}\|\cdot \|Df_{ \omega}^{-1} \vert E^{cu}_{\sigma\omega,f_\omega(x)}\|\le\lambda<1.$$
Let $C$ be a Lipschitz constant for $f_\omega|_{K_\omega}$, and choose $0<\eta\le 1$ small so that  $\lambda c^\eta<1$. By Lemma~\ref{th.hirsch-pugh} there exists a unique section $\mathfrak s_0\colon M\to L_\omega^1 $ whose image is invariant under~$G_\omega$ and it satisfies a H\"older condition with exponent~$\eta$. This unique section is necessarily the null section.
\end{proof}

We now show that the holonomy map $\Theta_\omega$ is absolutely continuous and derive a formula for the density $\rho_\omega$. This will be done in few steps following the ideas of \cite{M87}.
\begin{corollary}\label{co.produtorio}
There are $C>0$ and $0<\beta<1$ such that for all $y\in\gamma^s_\omega(x)$ and $ n\ge 0$
 $$\displaystyle
 \log \prod_{i=n}^\infty\frac{\det D^uf_{\sigma^i\omega}(f^i_\omega(x))}{\det D^uf_{\sigma^i\omega}(f^i_\omega(y))}\le C\beta^{n}.
 $$
\end{corollary}
\begin{proof} Since $Df_{\sigma^i\omega}$ is H\"older continuous, Proposition~\ref{th.Holder} implies that $\log|\det D^uf_{\sigma^i\omega}|$ is H\"older continuous. The corollary then follows from Proposition \ref{pr.predisks} which implies uniform backward contraction on unstable leaves.
\end{proof}

We now introduce some notion. We say that $\phi: M_1\to M_2$, where $M_1$ and $M_2$ are submanifolds of $M$,  is \emph{absolutely continuous} if it is an injective map for which there exists $J:M_1\to\RR$ such that
 $$
 \leb_{M_2}(\phi(A))=\int_A Jd\leb_{M_1}.
 $$
$J$ is called the \emph{Jacobian} of $\phi$.
Property (P$_4$) (1) can be restated as follows, replacing $\Theta^{-1}_\omega$ by $\phi_\omega$:

\begin{proposition}\label{pr.regulstable}
Given
$\gamma,\gamma'\in\Gamma_\omega^u$, define
$\phi_\omega\colon\gamma'\to\gamma$  by
$\phi_\omega(x)=\gamma^s(x)\cap \gamma$. Then $\phi_\omega$ is absolutely continuous and the Jacobian of $\phi_\omega$ is given by
        $$
        J_\omega(x)=
        \prod_{i=0}^\infty\frac{\det D^uf_{\sigma^i\omega}(f_\omega^i(x))}{\det
        D^uf_{\sigma^i\omega}(f_\omega^i(\phi_\omega(x)))}.$$
\end{proposition}
Obviously, by Corollary~\ref{co.produtorio}, the infinite product in Proposition \ref{pr.regulstable} converges uniformly.
We prove Proposition~\ref{pr.regulstable} in the following three lemmas.

\begin{lemma}\label{le.jacomane}
Let $M_1$ and $M_2$ be manifolds, $M_1$ with finite volume, and for each $n\ge 1$,  let $\phi_n:M_1\to M_2$ be an
absolutely continuous map with Jacobian $J_n$. Assume that
\begin{enumerate}
  \item $\phi_n$ converges uniformly to an injective
continuous map $\phi:M_1\to M_2$;
  \item $J_n$ converges uniformly to an integrable function $J: M_1\to\RR$.
\end{enumerate}
Then
$\phi$ is absolutely continuous with Jacobian $J$.
\end{lemma}
The proof of Lemma \ref{le.jacomane} is given in \cite[Theorem~3.3]{M87}. For the sake of completeness, we observe that there is a slight difference in our definition of absolute continuity. Unlike~\cite{M87}, we do not assume continuity of the maps $\phi_n$. However, the proof of \cite[Theorem~3.3]{M87} uses only the continuity of the limit function $\phi$, and so it still works in our case.

Consider now $\gamma,\gamma'\in\Gamma^u$ and
$\phi\colon\gamma'\to\gamma$  as in Proposition~\ref{pr.regulstable}. 


We consider the following sequence of  random return times
for points in $\Lambda_\omega$,
 \begin{equation*}\label{def.rs}
 r_1=R_\omega\qand r_{n+1}=r_{n}+R_{\sigma^{r_n}\omega}\circ f_\omega^{r_{n}},\quad\text{for $n\ge1$}.
 \end{equation*}
Notice that the $r_n$'s are defined $\leb_\gamma$ almost everywhere on each $\gamma\in \Gamma^u_\omega$ and they are piecewise constant.

\begin{remark}\label{re.sp} Using the above sequence of random return times, one can construct a sequence of  $\leb_\gamma$ mod 0  partitions   $(\mathcal Q_{\omega,n})_n$ by $s$-subsets of $\Lambda_\omega$ with $r_n$ being constant on each element of $\mathcal Q_{\omega,n}$, for which (P$_1$)-(P$_5$)-(1) hold with $r_n$ replacing $R_\omega$ and the elements of $\mathcal Q_{\omega,n}$ replacing the $s$-subsets of $\Lambda_\omega$. Moreover, the constants $C>0$ and $0<\beta<1$ in the backward contraction, implied by Proposition \ref{pr.predisks} can be chosen not depending on $n$.
\end{remark}

\begin{lemma}\label{le.mane}
For each $n\ge 1$, there is an absolutely continuous  $\varrho_{\sigma^{r_n}\omega}:f_\omega^{r_n}(\gamma)\to f_\omega^{r_n}(\gamma')$ with Jacobian $G_{\sigma^{r_n}\omega}$ satisfying
\begin{enumerate}
  \item $\displaystyle \lim_{n\to\infty}\sup_{x\in \gamma}\left\{ \dist_{f_\omega^{r_n}(\gamma')}(\varrho_{\sigma^{r_n}\omega}(f_\omega^{r_n}(x)),f_\omega^{r_n}(\phi_\omega(x))\right\}=0$;
  \item $\displaystyle \lim_{n\to\infty}\sup_{x\in f_\omega^{r_n}(\gamma)} \left\{ |1-G_{\sigma^{r_n}\omega}(x)|  \right\}=0$.
\end{enumerate}

\end{lemma}
\begin{proof}
Since $\Gamma^u_\omega$ is a continuous family of $C^1$ disks, there are
a compact set~$K$, a unit disk~$D$ in some $\mathbb R^k$ and an injective  continuous function
 $\Phi\colon K\times D\to M$ 
 such that
  $$\Gamma^u_\omega=\left\{\Phi(\{x\}\times D) \colon x\in K\right\}.$$
Moreover,  $\Phi$ maps $K\times D$ homeomorphically onto its image, and  $x\mapsto \Phi\vert_{\{x\}\times D}$ defines a continuous map from $K$ into $\text{Emb}^1(D,M)$.
Since $\Lambda_\omega$ has a full product structure,  For $\mathcal Q, \mathcal Q'\in (\mathcal Q_{\omega,n})_n$, the disks $f_\omega^{r_n}(\mathcal Q)$ and $f_\omega^{r_n}( \mathcal Q')$ necessarily belong in $\Gamma^u_\omega$. Thus, we may consider points $x^{}_{\mathcal Q,n}, x^{}_{\mathcal Q',n} \in K$ such that
 $$
\Phi(\{x^{}_{\mathcal Q,n}\}\times D)= f_\omega^{r_n}( \mathcal Q)\qand \Phi(\{x_{\mathcal Q',n}\}\times D)=f_\omega^{r_n}( \mathcal Q').
 $$
This naturally induces  $C^1$ diffeomorphisms 
  $\Phi^{}_{\mathcal Q,n}\colon 
  D\to f_\omega^{r_n}( \omega)\qand 
  \Phi_{\mathcal Q',n}\colon 
   D\to f_\omega^{r_n}( \gamma').$
Set
  $
 \varrho_{\sigma^{r_n}\omega}= \Phi_{\mathcal Q,n}\circ \Phi^{-1}_{\mathcal Q',n}
  $. The two conclusions  are now consequence of the contraction property (P$_2$), together with  the fact that $x\mapsto \Phi\vert_{\{x\}\times D}$ defines a continuous map from $K$ into $\text{Emb}^1(D,M)$. Note  that   $f^{r_n}_\omega( \mathcal Q)$ and  $f^{r_n}_\omega( \mathcal Q')$  are disks  in $\Gamma^u_\omega$  close to each other in the $C^1$ topology for large $n$, uniformly on  $\mathcal Q\in\mathcal Q_{\omega,n}$. This gives in particular the second item. For the first one, note also that $f^{r_n}_\omega\circ\phi_{\gamma',\gamma}(x)$ is the point in the unstable disk $f_\omega^{r_n}(\mathcal Q')$ resulting from the intersection of $f^{r_n}_\omega(\mathcal Q')$ with $\gamma^s ((f_\omega^{r_n})(x))$. This point is necessarily close to~$f_\omega^{r_n}(x)$ for large $n$, since also we assume $\Gamma_\omega^s$ is a continuous family of $C^1$ disks.  
\end{proof}
We define, for each $n\ge 1$, and $\gamma,\gamma' \in\Gamma^u_\omega$, $\phi_{\omega,n}: \gamma\to\gamma'$ as
 \begin{equation}\label{eq.phin}
 \phi_{\omega,n}=f^{-r_n}_\omega\varrho_{\sigma^{r_n}\omega} f_\omega^{r_n}.
 \end{equation}
Then $\phi_{\omega,n}$ is absolutely continuous with Jacobian
  \begin{equation}\label{eq.Jn}
  J_n(x)=\frac{|\det D^uf_\omega^{r_n}(x)|}{|\det
        D^uf_\omega^{r_n}(\phi_{\omega,n}(x))|}\cdot G_{\sigma^{r_n}\omega,n}(f_\omega^{r_n}(x)).
  \end{equation}
Observe that $\phi_{\omega,n}$ is defined $\leb_\gamma$ almost everywhere. We can find a Borel set $A\subset \gamma$ with full $\leb_\gamma$ measure on which $\phi_{\omega,n}$ is defined everywhere. We extend $\phi_{\omega,n}$ to $\gamma$ by letting $\phi_{\omega,n}(x)=\phi_\omega(x)$ and $J_n(x)=J(x)$ for all $n\ge 1$ and $x\in \gamma\setminus A$. Since $\gamma\setminus A$ has zero $\leb_\gamma$ measure, $J_n$ is the Jacobian of $\phi_{\omega,n}$.

\begin{lemma}\label{le.phin}
$(\phi_{n,\omega})_n$ converges   uniformly to $\phi_\omega$ and   $(J_n)_n$ converges uniformly to $J$.
\end{lemma}

\begin{proof} It is sufficient to prove the convergence of the sequence restricted to $A$ as described above. In particular, $\phi_{\omega,n}$ and $J_{\omega,n}$ are given by~\eqref{eq.phin} and~\eqref{eq.Jn} respectively. We first prove the convergence of $(\phi_{\omega,n})_n$. By Remark~\ref{re.sp}, we may write for each $x\in\gamma$
\begin{eqnarray*}
  \dist_{\gamma'}(\phi_{\omega,n}(x),\phi_\omega(x))  &=& \dist_{\gamma'}(f^{-r_n}_\omega\varrho_{\sigma^{r_n}\omega} f_\omega^{r_n}),f_\omega^{-{r_n}}f_\omega^{r_n}\phi_\omega(x)) \\
   &\le& C\beta^{r_n}\dist_{f_\omega^{r_n}(\gamma')}(\varrho_{\sigma^{r_n}\omega} f_\omega^{r_n}(x),f_\omega^{r_n}\phi_\omega(x)).
\end{eqnarray*}
Since ${r_n}\to\infty$ as $n\to\infty$ and $\dist_{f_\omega^{r_n}(\gamma')}(\varrho_{\sigma^{r_n},\omega} f_\omega^{r_n}(x),f_\omega^{r_n}\phi_\omega(x))$ is bounded, by Lemma~\ref{le.mane}, $\phi_{\omega,n}$ converges uniformly to $\phi_\omega$.

 We now prove that $(J_n)_n$ converges to $J$ uniformly. By \eqref{eq.Jn}, we have
  $$J_n(x)=\frac{|\det D^uf_\omega^{{r_n}}(x)|}{|\det
        D^uf_\omega^{{r_n}}(\phi_\omega (x))|}\cdot \frac{|\det D^uf_\omega^{{r_n}}(\phi_\omega (x))|}{|\det
        D^uf_\omega^{{r_n}}(\phi_{\omega,n}(x))|}\cdot G_{\sigma^{r_n}\omega}(f_\omega^{{r_n}}(x)).
        $$
 By the chain rule and  Corollary~\ref{co.produtorio}, the first term in the product converges uniformly to $J(x)$. Moreover, by Lemma~\ref{le.mane}, the third term converges uniformly to~1. Finally, recalling Remark~\ref{re.sp}, by  bounded distortion we have
  \begin{eqnarray*}
    \frac{|\det D^uf_\omega^{{r_n}} (\phi_\omega (x))|}{|\det
        D^uf_\omega^{{r_n}}(\phi_{\omega,n}(x))|}  &\le &\exp\big(C\dist_{f_\omega^{r_n}(\gamma')}(f_\omega^{r_n}(\phi_\omega(x)),f_\omega^{r_n}(\phi_{\omega,n}(x)))^{\eta}\big) \\
     &=& \exp\big(C\dist_{f_\omega^{r_n}(\gamma')}(f_\omega^{r_n}(\phi(x)),\varrho_{\sigma^{r_n}\omega}(f_\omega^{r_n}(x)))^{\eta}\big).
  \end{eqnarray*}
  Similarly we obtain
  $$\frac{|\det D^uf_\omega^{{r_n}}(\phi_\omega (x))|}{|\det
        D^uf_\omega^{{r_n}}(\phi_{\omega,n}(x))|}  \ge  \exp\big(-C\dist_{f_\omega^{r_n}(\gamma')}(f_\omega^{r_n}(\phi_\omega(x)),\varrho_{\sigma^{r_n}\omega}(f_\omega^{r_n}(x)))^{\eta}\big).
        $$
        Therefore, by Lemma~\ref{le.mane}, $G_{\sigma^{r_n}\omega}\circ f_\omega^{{r_n}}$ converges uniformly to~1.
 \end{proof}
This completes the proof of Proposition~\ref{pr.regulstable} and of (P$_4$)(1); i.e., the holonomy map $\Theta_\omega$ is absolutely continuous and  
	\[
	\rho_\omega(x)= \frac{d([\Theta_\omega]_\ast\Leb_{\gamma_\omega})}{d\Leb_{\gamma_\omega'}}(x)=\prod_{i=0}		           	^\infty\frac{\det D^uf_{\sigma^i\omega}(f^i_\omega(x))}{\det D^uf_{\sigma^i\omega}(f^i_\omega(\Theta^{-1}_\omega(x)))}.
	 \]

 We finally prove (P$_4$)(2) in the next result.

\begin{lemma}\label{co.new}
For any $\gamma_\omega\in\Gamma_\omega^u$ and any $x, y\in \gamma_\omega$  we have $\log\dfrac{\rho_\omega(x)}{\rho_\omega(y)}\le C\beta^{s(x, y)}$. 	
\end{lemma}
\begin{proof}
Let $k\approx[ s(x,y)]/2$. We have
\begin{align}
\left|\log\dfrac{\rho_\omega(x)}{\rho_\omega(y)}\right|&=\left|\log \prod_{i=0}^\infty\frac{\det D^uf_{\sigma^i\omega}(f^i_\omega(x))}{\det D^uf_{\sigma^i\omega}(f^i_\omega(\Theta^{-1}_\omega(x)))}-
\log\prod_{i=0}^\infty\frac{\det D^uf_{\sigma^i\omega}(f^i_\omega(y))}{\det D^uf_{\sigma^i\omega}(f^i_\omega(\Theta^{-1}_\omega(y)))}\right|
\nonumber\\
&\le\left|\log \prod_{i=0}^{k-1}\frac{\det D^uf_{\sigma^i\omega}(f^i_\omega(x))}{\det D^uf_{\sigma^i\omega}(f^i_\omega(\Theta^{-1}_\omega(x)))}-
\log\prod_{i=0}^{k-1}\frac{\det D^uf_{\sigma^i\omega}(f^i_\omega(y))}{\det D^uf_{\sigma^i\omega}(f^i_\omega(\Theta^{-1}_\omega(y)))}\right|\nonumber\\
&+\left|\log \prod_{i=k}^\infty\frac{\det D^uf_{\sigma^i\omega}(f^i_\omega(x))}{\det D^uf_{\sigma^i\omega}(f^i_\omega(\Theta^{-1}_\omega(x)))}-
\log\prod_{i=k}^\infty\frac{\det D^uf_{\sigma^i\omega}(f^i_\omega(y))}{\det D^uf_{\sigma^i\omega}(f^i_\omega(\Theta^{-1}_\omega(y)))}\right|\nonumber\\
&\le\sum_{i=0}^{k-1}\left|\log \frac{\det D^uf_{\sigma^i\omega}(f^i_\omega(x))}{\det D^uf_{\sigma^i\omega}(f^i_\omega(y))}\right|+
\sum_{i=0}^{k-1}\left|\log\frac{\det D^uf_{\sigma^i\omega}(f^i_\omega(\Theta^{-1}_\omega(y)))}{\det D^uf_{\sigma^i\omega}(f^i_\omega(\Theta^{-1}_\omega(x)))}\right|\nonumber\\
&+\left|\log \prod_{i=k}^\infty\frac{\det D^uf_{\sigma^i\omega}(f^i_\omega(x))}{\det D^uf_{\sigma^i\omega}(f^i_\omega(\Theta^{-1}_\omega(x)))}\right|+\left|
\log\prod_{i=k}^\infty\frac{\det D^uf_{\sigma^i\omega}(f^i_\omega(y))}{\det D^uf_{\sigma^i\omega}(f^i_\omega(\Theta^{-1}_\omega(y)))}\right|.
\label{eq.notthistorsion}
\end{align}
Denote
$J_{\sigma^i\omega}(x)= |\det D^uf_{\sigma^i\omega} \vert T_{f_\omega^i(x)} \gamma_i|$, for all  $0\le i <k$ and $x\in \gamma$. By Proposition~\ref{pr:bounded-curvature}, $J_{\sigma^i\omega}$ is
$(L,\zeta)$-H\"older continuous, for some $L>0$. Then, there is some uniform constant $C'>0$ such that
$$
 \sum_{i=0}^{k-1}  \log\frac{|\det D^uf_{\sigma^i\omega} \vert T_{f_\omega^i(y)} \gamma_i|}{|\det D^uf_{\sigma^i\omega} \vert T_{f_\omega^i(x)} \gamma_i|}
= \sum_{i=0}^{k-1} \big(\log J_{\sigma^i\omega}(y) - \log J_{\sigma^i\omega}(x)\big)\le  \sum_{i=0}^{n-1} C' \dist(f^i_\omega(x), f^i_\omega(y))^\zeta.
$$
 Since we have chosen   $k\approx [ s(x,y)]/2$,   the points  $ f^i_\omega(x)$ and $ f^i_\omega(y)$ have at least $s(x,y)-i $ simultaneous returns to $\Lambda_{\sigma^i\omega}$, for $0\le i<k$. So, using the backward expansion on unstable leaves there is a $C''>0$ such that 
$
 \dist(f_\omega^i(x), f_\omega^i(y))^\zeta \le C''\beta^{\zeta (s(x,y)-i)}.
 $
Recalling that $s(x,y)-k\approx s(x,y)/2$, we get
for   $\beta_1=\beta^{\zeta/2}$ and some $C_1>0$
\begin{equation}\label{lastone}
\sum_{i=0}^{k-1}  \log\frac{|\det D^uf_{\sigma^i\omega} \vert T_{f_\omega^i(y)} \gamma_i|}{|\det D^uf_{\sigma^i\omega} \vert T_{f_\omega^i(x)} \gamma_i|}\le C_1\beta_1^{s(x,y)}.
\end{equation}
The same conclusion holds for the other finite sum  in \eqref{eq.notthistorsion}. Using Corollary~\ref{co.produtorio}, for the infinite products  in~\eqref{eq.notthistorsion}, and \eqref{lastone} we can find $C_2>0$ so that 
$$\left|\log\dfrac{\rho_\omega(x)}{\rho_\omega(y)}\right|\le 2C_1\beta_1^{s(x,y)}+2C\beta^k\le C_2\beta_1^{s(x,y)}.$$
\end{proof}

\subsection{Axiom A attractors} \label{se.AxA}
Here we prove Theorem~\ref{co.DCAxA}. 
Let $f\in \diff^{1+}(M)$  have a topologically mixing uniformly hyperbolic attractor  $K\subset M$. Classical results give that under this conditions $f$ has a unique physical measure which is actually mixing; see~\cite{B75}.
Recall that in this case the attractor necessarily contains the local unstable manifolds. Moreover, such manifolds persist under small $C^1$ random perturbations and depend continuously on the  perturbations; see e.g. \cite[Sections~1~\&~2]{Y86}. Consequently, we can choose the reference leaf $ \Sigma_\omega$ and the cylinder $\C (\Sigma_\omega)$ of Subsection~\ref{s.structure} depending continuously on $\omega\in\Omega$. Also, the unstable  leaves $u$-crossing these cylinders depend continuously on $\omega\in\Omega$, and so the measurability of $\Omega\ni\omega\mapsto\Lambda_\omega$ follows. 
In this uniformly hyperbolic case, we trivially have that
$\Leb_{D_\omega}\{\E_\omega>n\}$ decays exponentially fast in $n$, uniformly in $\omega$.
For the aperiodicity see Remark~\ref{re.VIR}.
So, recalling Remark~\ref{re.exdelta} and applying Theorem~\ref{exrates} and Theorem~\ref{randomPH} we obtain
Theorem~\ref{co.DCAxA}. 

\subsection{Derived from Anosov}\label{se:DA}

To obtain Theorem~\ref{co.Derived} we are going to apply Theorem~\ref{co.DCAxA} with $K=M$. 
Firstly, observe that property (5) above and  \cite[Lemma 7.4]{AAV07} guarantee that for any  unstable disk $D$ for $f$ there exist $C,c>0$  and $0< \tau \le 1$  such that 
$\Leb_{D }\{\E_\omega>n\}=  Ce^{-cn } $ for all $\omega\in\Omega$. Moreover,  taking $D_\omega=D$ we trivially have that  $\{f_\omega\}_{\omega\in\Omega}$  is a random perturbation   $C^1\!$-close to $f|_D$ on domains $\{D_\omega\}_{\omega\in\Omega}$ of $cu$-nonuniform expansion. 
Hence, by Theorem~\ref{co.DCAxA}   there exists $\Lambda_\omega\subset K_\omega$ with a   hyperbolic product structure and  there exist $C',c'>0$ such that $\Leb_{D_\omega}\{R_\omega>n\}= C'e^{-c'n^\tau}$ for all $\omega\in\Omega$. If the random perturbation is not aperiodic, then we take $N$ equal to the greatest common divisor of the return times. Notice that $\Lambda_\omega$ is now aperiodic measurable hyperbolic product structure for $\{f^N_\omega\}_{\omega\in\Omega}$. Finally, using Theorem~\ref{existence} and Theorem~\ref{exrates} we get the conclusion.

\section{Solenoid with intermittency}\label{se.sole}
In this section we prove Theorem~\ref{Appl}. We will show that the random dynamical system $\{g_\omega\}_{\omega\in\Omega}$ introduced in Subsection~\ref{se.solemio} admits a hyperbolic product structure. First notice that for almost every $\omega\in\Omega$ the random map  $g_\omega$ has an attractor in $M$ which is given by 
$$
K_\omega=\bigcap_{n\ge 1} g^{n-1}_{\sigma^{-n}\omega}(M).
$$
Note that $K_\omega$ is measurable in $\omega$ by    \cite[Theorem 3.11]{CF94}. First we show that for every $\omega$ the attractor $K_\omega$ is partially hyperbolic. By a straightforward computation for  $p= (x, y, z) \in M$ we obtain 
\[
Dg_\omega(x, y, z)=\begin{bmatrix}T'_\omega(x) & 0 &0 \\ -(\sin x)/2 &{1}/{10} &0 \\   (\cos x)/2 &   0 & 1/ 10  \end{bmatrix} .
\]
It is easily checked that  $E^s_\omega\{0\}\times \DD^2$ is invariant and uniformly contracted.  We will construct $E^{cu}_\omega(p)$ as a countable  intersection of  centre-unstable cones. For $p\in M$ define
\[
\C^{cu}(p)=\{ (v_1, v_2, v_3)\in T_pM\mid v_1\in T\SS^1, (v_2, v_3)\in \RR^2 \textnormal{ and }  |v_1|\ge 2(v_2^2+v_3^2)^{1/2}\}.
\]
The first step is to  show that $Dg_\omega \C^{cu}(p)\subset \C^{cu}(g_\omega(p))$.  For $(v_1, v_2, v_3)\in  \C^{cu}$ let  $(\bar v_1, \bar v_2, \bar v_3)=Dg_\omega (p)(v_1, v_2, v_3)$. Using the fact $A\sin x+B\cos x\le \sqrt{A^2+B^2}$ and the definition of $\C^{cu}$ we have 
\[
\begin{aligned}
 \bar v_2^2+\bar v_3^2 &= (-\frac 12v_1\sin x+\frac {1}{10}v_2)^2+(\frac 12v_1\cos x+\frac {1}{10}v_3)^2\le \\
& \frac14v_1^2+\frac1{100}(v_2^2+v_3^2)+\frac1{10}|v_1|(v_2^2+v_3^2)^{1/2}\le\\
&  \frac14v_1^2+\frac1{400}v_1^2+\frac1{40}v_1^2< \bar v_1^2
  \end{aligned}
\]
since $T'_\alpha \ge 1$ on $M$. In addition for any $(v_1, v_2, v_3)\in \C^{cu}(p)$ we have 
\begin{equation}\label{eq:expcon}
\|Dg_\omega(p)(v_1, v_2, v_3)\|^2\ge |T'_{\omega}(x) v_1|^2 \ge v_1^2\ge \frac 45(v_1^2+v_2^2+ v_3^2). 
\end{equation}

Define  \(E^{cu}_\omega(p)=\bigcap_{k=0}^\infty Dg_{\sigma^{-k}\omega}^k(g_{\sigma^{-k}\omega}^{-k}(p))\C^{cu}(g_{\sigma^{-k}\omega}^{-k}(p))
\). The second step is to show that  $E^{cu}_\omega(p)$ is a line. By the previous argument, the finite intersections 
\[
\C^{(n)}=\bigcap_{k=0}^n Dg_{\sigma^{-k}\omega}^k(g_{\sigma^{-k}\omega}^{-k}(p))\C^{cu}(g_{\sigma^{-k}\omega}^{-k}(p))= Dg_{\sigma^{-n}\omega}^n(g_{\sigma^{-n}\omega}^{-n}(p))\C^{cu}(g_{\sigma^{-n}\omega}^{-n}(p))
\] 
form a nested sequence. It is sufficient to prove that  the angle between any two non-zero vectors $\C^{(n)}$ converges to zero as $n$ goes to infinity. For any $(v^{(1)}_1, v^{(1)}_2, v^{(1)}_3), (w^{(1)}_1, w^{(1)}_2, w^{(1)}_3)\in \C^{(1)}$ we have 
\[
\begin{aligned}
\left|\frac{v^{(1)}_2}{v^{(1)}_1}-\frac{ w^{(1)}_2}{w^{(1)}_1}\right|=&\left|(-\frac 12v_1^{(0)}\sin x+\frac {1}{10}v_2^{(0)})/(v^{(0)}_1T_\omega'(x))-(-\frac 12w_1^{(0)}\sin x+\frac {1}{10}w_2^{(0)})/(w^{(0)}_1T_\omega'(x))\right|\\
&\le \frac{1}{10}\left|\frac{v^{(0)}_2}{v^{(0)}_1}-\frac{ w^{(0)}_2}{w^{(0)}_1}\right|.
\end{aligned}\]
Similarly
\[
\left|\frac{v^{(1)}_3}{v^{(1)}_1}-\frac{ w^{(1)}_3}{w^{(1)}_1}\right|\le \frac{1}{10}\left|\frac{v^{(0)}_3}{v^{(0)}_1}-\frac{ w^{(0)}_3}{w^{(0)}_1}\right|.
\]
By induction, we obtain  
\[\left|\frac{v^{(n)}_i}{v^{(n)}_1}-\frac{ w^{(n)}_i}{w^{(n)}_1}\right|\le C 10^{-n},  i=2, 3\text{  for any }(v^{(n)}_1, v^{(n)}_2, v^{(n)}_3), (w^{(n)}_1, w^{(n)}_2, w^{(n)}_3)\in \C^{(n)}.
\]
This shows that the angle between $(v^{(n)}_1, v^{(n)}_2, v^{(n)}_3)$ and $(w^{(n)}_1, w^{(n)}_2, w^{(n)}_3)$ converges to zero. This implies the existence of a continuous splitting. Moreover, this splitting is continuous in $\omega$ since $g_\omega$ is smooth in $\omega$.

We now outline the construction of a local unstable manifold, using the classic graph transformation idea. Let $p_0=(0,\frac{5}{9},0)$ and notice that $g_\omega(p_0)=p_0$ for any $\omega\in\Omega$. Let $\gamma:J_{p_0}\to \mathbb D^2$ denote a $C^1$ function  defined on small interval $J_{p_0}\subset \mathbb S^1$ around $p_0\in { K_\omega}$. Let $L(p_0)=\{(x, { \gamma}(x))\mid x\in J\}$ denote its graph.  Suppose that $T_qL(p_0)\subset \C^{cu}(p_0)$ for all $q\in L(p_0)$. Define a sequence of push-forwards restricted to a sufficiently small ball $B_\eps(p_0)$
\[
L_n(p_0)={ g^n_{\omega}}(L({ g^{-n}_{\sigma^{-n}\omega}}(p_0)))\cap B_\eps(p_0)=\{{ g^n_{\omega}}(x, {\gamma}(x))\mid x\in J_{{ g^{-n}_{\sigma^{-n}_\omega}}p_0}\}\cap B_\eps(p_0).
\]
Let $\gamma$ and $\gamma'$ be $C^1$ functions as above and $L(p_0,\gamma), L(p_0,\gamma')$ be the corresponding graphs. Then using the formula of $g_\omega$ in \eqref{map}, we have
$$\|g_\omega(x,\gamma(x))-g_\omega(x,\gamma'(x))\|=\frac{1}{10}\|\gamma(x)-\gamma'(x)\|.$$
Therefore, under iteration of $g_\omega$, $\gamma$ converges, pointwise, to a limiting curve, call it  $W^{u}_{\omega,\eps}(p_0)$.
Now, notice $T_{g^n_\omega(q)}L_n(p_0)\subset \C^{(n)}$ and by the strict contraction of the cone field, the angle between $T_{g^n_\omega(q)}L_n(p_0)$ and $E^{cu}_\omega$ converges to zero as $n\to\infty$. Therefore, $E^{cu}_\omega$ is the tangent space for $W^{u}_{\omega,\eps}(p_0)$, which is measurable in $\omega$. Thus, $W^{u}_{\omega,\eps}(p_0)$ is contained in $K_\omega$ and is the local unstable manifold at $p_0$.
Now the global centre-unstable manifold at $p_0$ is defined as 
 \[
W^{u}_{\omega}(p_0)= \bigcup_{k=0}^\infty g_{\sigma^{-k}\omega}^kW^{u}_{\sigma^{-k}\omega,\eps}(g_{\sigma^{-k}\omega}^{-k}(p_0)),
 \]
 which winds around $K_\omega$. Obviously, the stable manifold at $p_0$ is $\{p_0\}\times\mathbb D^2$.  We will now construct a set  $\Lambda_\omega$ with a hyperbolic product structure for almost every every $\omega\in \Omega$.  We let $\pi: M\to\mathbb S^1$ be a projection.  Let $\gamma^u_\omega(p_0)\subset W^u_\omega(p_0)$ such that $\pi_\omega =\pi|\gamma^u_\omega(p_0): \gamma^u_\omega(p_0)\to (0, 1)$ induces a $C^1$ diffeomorphism. For each $\omega$ we define a random sequence of pre-images of $\frac{1}{2}$ as follows. 
Let  $x^\pm_1(\omega)=\frac{1}{2}$,  and
\begin{equation}\label{eq:x_n}
x_n^\pm(\omega) =\left(T_\omega^\pm\right)^{-1}x_{n-1}(\sigma\omega) \,\, \text{for} \,\, n\ge 2,
\end{equation}  
where $T_\omega^+=T_\omega|[\frac12,1)$ and $T_\omega^-=T_\omega|(0,\frac12)$. Set $I_n(\omega)^-=(x_n^-(\omega), x_{n-1}^-(\omega))$ and $I_n(\omega)^+=( x_{n-1}^+(\omega),x_n^+(\omega))$. Now  we start defining continuous family of $C^1$ unstable manifolds inductively. Let $\Gamma_0^\omega=\{\gamma^u_\omega(p_0)\}$ and let $K_{\omega, k}^\pm=\{p\in M\mid \pi(p)\in I_k(\omega)^\pm\}$.  For $n\ge 1$ let 
\[
\Gamma_n^\omega=\bigcup_{k\ge 1}\left\{g_{\sigma^{-k}\omega}^k(K_{\sigma^{-k}\omega, k}^\pm\cap\gamma_{\sigma^{-k}\omega})\mid \gamma_{\sigma^{-k}\omega}\in \Gamma_{n-1}^{\sigma^{-k}\omega}\right\}, \text { and }\Lambda_\omega=\overline {\bigcup_{n\ge 0}\bigcup_{\gamma_n\in \Gamma_n^{\omega} }\gamma_n}.
\]
It follows that $\pi_\omega:\gamma_\omega\to (0, 1)$ is a diffeomorphism for any $\gamma_\omega\in\Gamma_n^{\omega}$.  By construction  for any $q\in \Lambda_\omega$ 
there is a sequence $\gamma_{n_k}^\omega\in\Gamma_n^\omega$ and points $q_{k}\in\gamma_{n_k}$ converging to $q$, as $k\to \infty$.  Now,  as subsection \ref{s.product} we argue that  $\gamma_{n_k}^\omega$  converge to an unstable manifold $\gamma_\omega$. Define $\Gamma_\omega^u$ to be set of $\gamma_{n}^\omega$  and the limiting manifolds $\gamma_\omega$. Notice that $\Lambda_\omega=\cup_{\gamma^\omega\in \Gamma^u_\omega}\gamma^\omega$.  Finally, set $\Gamma_\omega^s=\{\{x\}\times \mathbb D^2\mid x\in \mathbb S^1\}$. 
Thus, $\Lambda_\omega$ admits a hyperbolic product structure.

We now show the random systems satisfies   {(P$_0$)}-{(P$_4$)} for a suitable partition of $\Lambda_\omega$.   We define 
$s$-subsets as $\Lambda_{n, \omega}^\pm=\{\{x\}\times \gamma^s(\omega)\mid x\in I_n^\pm(\omega)\}$. Obviously,  $\Lambda_{n, \omega}^\pm$ is a $\text{mod } 0$ partition of $\Lambda_\omega$  and $g^{n+1}_\omega(\Lambda_{n, \omega}^\pm)$ is a $u$-subset, since $T_\omega^n(I_n(\omega)^-)=(\frac 12, 1)$, $T_\omega^n(I_n(\omega)^+)=(0,\frac 12)$ and consequently $T_\omega^{n+1}(I_n(\omega)^\pm)=(0,1)$. 
Properties { (P$_0$)}-{(P$_2$)} are then automatically satisfied. To prove uniform expansion in {(P$_3$)}  we first note that  $(T^{-}_\alpha)'(x)=1+(\alpha+1)(2x)^\alpha$, and $(T^{-}_\alpha)^{-1}(\frac12)>\frac14$  for all $\alpha\in(0, 1)$. Therefore,  $x_2^-(\omega)>\frac14$ for all $\omega\in\Omega$.  Thus 
\begin{equation}
(T_\omega^{-})'(x)\ge 1+2^{-\omega_0}(\omega_0+1) \text{ for all } x\in [x_2^-(\omega), \frac12].   
\end{equation}
It is easily seen from there that $(T^-_\omega(x))'\ge 2$ for all $x\in [x_2^-(\omega), \frac12]$. An analogous estimate holds for the right branch and therefore we have  $(T_\omega^{R_\omega})'(x)\ge 2$, since before making a return every orbit visits $(x_2^-(\omega), \frac12)$ or $(\frac12, x_2^+(\omega))$.  Finally,  for any $(v_1, v_2, v_3)\in \C^{cu}(p)$ by \eqref{eq:expcon} we have
\begin{equation}
\|Dg_\omega^{R_\omega}(p)(v_1, v_2, v_3)\|^2 \ge |(T_\omega^{R_\omega})' (x)v_1|^2 \ge \frac 85(v_1^2+v_2^2+ v_3^2). 
\end{equation}
Now for each $\gamma_\omega\in \Gamma^u_\omega$, the projection $\pi: M\to\mathbb S^1$ induces a $C^1$ diffeomorphism, $\pi_{\gamma_\omega}:\gamma_\omega\to (0,1)$. Note that the derivatives of $\pi_{\gamma_\omega}$ and $\pi^{-1}_{\gamma_\omega}$ are $\log$ Lipschitz uniformly in $\omega$ because the tangent space on $\gamma_\omega$ is in the cone filed, which is independent of $\omega$.
Observe that
$$g_\omega^{R_\omega} |_{\gamma_\omega\cap\Lambda_{n,\omega}}=\pi^{-1}_{g_\omega^{R_\omega}(\gamma_\omega\cap\Lambda_n,\omega)}\circ T_\omega^{R_\omega}\circ\pi_\omega|_{\gamma_\omega\cap\Lambda_n,\omega}.$$
 Therefore, bounded distortion in {(P$_3$)} follows from the bounded distortion of $T^{R_\omega}_\omega$, which was proved in \cite{BBD14, BBR19, R18}. In this situation {(P$_4$)} is obvious, since the projection along stable leaves is just the orthogonal projection to the $x$-axis. Aperiodicity is easy and can be   verified in the same way as in \cite{BBR19}. Further the tail assumptions, including the uniform tail assumption, of Theorem \ref{polyrates} are verified in the same way as in \cite{BBR19}
 and we obtain
 \begin{equation}
	\begin{cases} \label{apptail}
	\Leb_{\gamma_\omega}\{R_\omega>n\}\le C n^{-1/\alpha_0}\log n, \quad \text{whenever} \quad n\ge n_1(\omega),\\
	P\{n_1>n\} \le C e^{-\tau n^c};
	\end{cases}
		\end{equation} 	
moreover, $\int\Leb_{\gamma_\omega}\{R_\omega=n\}dP(\omega)\le C (\log n)^{\frac{1}{\alpha_0}+1}n^{-\frac{1}{\alpha_0}-1}$. Therefore, by Theorem \ref{polyrates} we have
$$|\mathcal C_n(\varphi, \psi,\mu_\omega)|\le  C_{\varphi, \psi}\max\left\{C_\omega n^{1-1/\alpha_0+\eps},  \delta_{\sigma^{[n/4]}\omega,[n/4]} \right\}$$ 
and $P\{C_\omega>n\}\le Ce^{b'n^{v'}}$ for some $b'>0$ and $c', \theta'\in(0, 1]$. We now obtain an estimate on $\delta_{\sigma^{[n/4]}\omega,[n/4]}$. Notice that
\[
\delta_{\sigma^k\omega, k}=\sup_{A\in \mathcal Q_{2k}^\omega}\diam(\pi_{\sigma^k\omega}(F^k_\omega(A)))\le\max\left\{10^{-k},
\sup_{{\bar A}\in \bar{\mathcal Q}_{2k}^\omega}\diam(\bar\pi_{\sigma^k\omega}(\bar F^k_\omega(\bar A)))\right\},
\]
and $\bar B=\bar F^k_\omega(\bar A)\in\bar{\mathcal Q}_{k}^{\sigma^k\omega}$. We consider two cases. First, if $\hat R_{\sigma^k\omega}|\bar B>k$, then $\pi_{\sigma^k\omega}\bar B\subset\{R_{\sigma^k\omega}>k\}$. Thus, $\diam(\bar\pi_{\sigma^k\omega}\bar B)\le C k^{-1/\alpha_0}\log k $ for $k>n_1(\sigma^k\omega)$. Second if $\hat R_{\sigma^k\omega}|\bar B\le k$, then there exists $N$ such that  $S_{N-1}|\bar B< k\le S_{N}|\bar B$
and $m\le N$ such that $S_m-S_{m-1}\ge \frac{k}{N}$. Let $U=f^{S_{m-1}}_\omega(\bar B)$. Observe that 
$U\subset\bar Q\in \bar{\mathcal Q}_{\omega'}$, where $\omega'=\sigma^{k+S_{m-1}}$, and $\hat R_{\omega'}|\bar Q>\frac{k}{N}$. By {(P$_3$)} we have $|\bar B|\le D\theta^{N-1}|\bar Q|$. Thus, if there exists $c>0$ such that $N\sim k^c$ then  $|\bar B|\le D\theta^{k^c}$. On the other hand, if there exists $0<c<1$ such that $N< k^c$, then using \eqref{apptail} we obtain $|\bar B|\le D\theta^{N-1}N^{1/\alpha_0}k^{-1/\alpha_0}\log k\le Ck^{-1/\alpha_0}\log k$, for any $k^{1-c}> n_1(\omega')$. Finally, define
\[
n_2(\omega)=\min\left\{n\mid \forall k>n, n_1(\sigma^j\omega)<k^{1-c}, k\le j< 2k\right\},
\]
and notice that
$$P(n_2>n)\le \sum_{k\ge n}\sum_{j=k+1}^{2k-1}P(n_1(\sigma^j\omega)\ge k^{1-c})\le Ce^{-\tau'n^{c'}},$$ for some $\tau'>0$. Now define 
$\bar C_\omega=\max\{C_\omega,n_2(\omega)^{1/\alpha_0}\}$. Then the theorem holds with $\bar C_\omega$.


\begin{thebibliography}{10}

\bibitem{ANV15}
R.~Aimino, M.~Nicol, and S.~Vaienti.
\newblock Annealed and quenched limit theorems for random expanding dynamical
  systems.
\newblock {\em Probab. Theory Related Fields}, 162(1-2):233--274, 2015.

\bibitem{AAV07}
J.~F. Alves, V.~Ara{\'u}jo, and C.~H. V{\'a}squez.
\newblock Stochastic stability of non-uniformly hyperbolic diffeomorphisms.
\newblock {\em Stoch. Dyn.}, 7(3):299--333, 2007.

\bibitem{ABV00}
J.~F. Alves, C.~Bonatti, and M.~Viana.
\newblock S{RB} measures for partially hyperbolic systems whose central
  direction is mostly expanding.
\newblock {\em Invent. Math.}, 140(2):351--398, 2000.

\bibitem{ADL13}
J.~F. Alves, C.~L. Dias, and S.~Luzzatto.
\newblock Geometry of expanding absolutely continuous invariant measures and
  the liftability problem.
\newblock {\em Ann. Inst. H. Poincar\'e Anal. Non Lin\'eaire}, 30(1):101--120,
  2013.

\bibitem{ADL17}
J.~F. Alves, C.~L. Dias, S.~Luzzatto, and V.~Pinheiro.
\newblock S{RB} measures for partially hyperbolic systems whose central
  direction is weakly expanding.
\newblock {\em J. Eur. Math. Soc. (JEMS)}, 19(10):2911--2946, 2017.

\bibitem{AL15}
J.~F. Alves and X.~Li.
\newblock {G}ibbs-{M}arkov-{Y}oung structures with (stretched) exponential tail
  for partially hyperbolic attractors.
\newblock {\em Adv. Math.}, 279(0):405 -- 437, 2015.

\bibitem{ALP05}
J.~F. Alves, S.~Luzzatto, and V.~Pinheiro.
\newblock Markov structures and decay of correlations for non-uniformly
  expanding dynamical systems.
\newblock {\em Ann. Inst. H. Poincar\'e Anal. Non Lin\'eaire}, 22(6):817--839,
  2005.

\bibitem{AP08a}
J.~F. Alves and V.~Pinheiro.
\newblock Slow rates of mixing for dynamical systems with hyperbolic
  structures.
\newblock {\em J. Stat. Phys.}, 131(3):505--534, 2008.

\bibitem{AP08}
J.~F. Alves and V.~Pinheiro.
\newblock Topological structure of (partially) hyperbolic sets with positive
  volume.
\newblock {\em Trans. Amer. Math. Soc.}, 360(10):5551--5569, 2008.

\bibitem{AP10}
J.~F. Alves and V.~Pinheiro.
\newblock Gibbs-{M}arkov structures and limit laws for partially hyperbolic
  attractors with mostly expanding central direction.
\newblock {\em Adv. Math.}, 223(5):1706--1730, 2010.

\bibitem{A98}
L.~Arnold.
\newblock {\em Random dynamical systems}.
\newblock Springer Monographs in Mathematics. Springer-Verlag, Berlin, 1998.

\bibitem{ALS09}
A.~Ayyer, C.~Liverani, and M.~Stenlund.
\newblock Quenched {CLT} for random toral automorphism.
\newblock {\em Discrete Contin. Dyn. Syst.}, 24(2):331--348, 2009.

\bibitem{BB16}
W.~Bahsoun and C.~Bose.
\newblock Mixing rates and limit theorems for random intermittent maps.
\newblock {\em Nonlinearity}, 29(4):1417--1433, 2016.

\bibitem{BBD14}
W.~Bahsoun, C.~Bose, and Y.~Duan.
\newblock Decay of correlation for random intermittent maps.
\newblock {\em Nonlinearity}, 27(7):1543--1554, 2014.

\bibitem{BBR19}
W.~Bahsoun, C.~Bose, and M.~Ruziboev.
\newblock Quenched decay of correlations for slowly mixing systems.
\newblock {\em Trans. Amer. Math. Soc.} 372(9), 6547--6587, 2019.

\bibitem{B95a} V.I.~ Bakhtin. 
\newblock Random processes generated by a hyperbolic sequence of mappings. I. 
\newblock (Russian) {\em Izv. Ross. Akad. Nauk Ser. Mat.} 58, no. 3, 40--72 ,1994; translation in {\em Russian Acad. Sci. Izv. Math.} 44, no. 3, 247--279, 1995.

\bibitem{B95b} V.I.~ Bakhtin. 
\newblock Random processes generated by a hyperbolic sequence of mappings. II. 
\newblock (Russian) {\em Izv. Ross. Akad. Nauk Ser. Mat.} 58, no. 3, 184--195 ,1994; translation in {\em Russian Acad. Sci. Izv. Math.} 44, no. 3, 617--627, 1995.

\bibitem{BBM02}
V.~Baladi, M.~Benedicks, and V.~Maume-Deschamps.
\newblock Almost sure rates of mixing for i.i.d. unimodal maps.
\newblock {\em Ann. Sci. {\'E}cole Norm. Sup. (4)}, 35(1):77--126, 2002.

\bibitem{BBM03}
V.~Baladi, M.~Benedicks, and V.~Maume-Deschamps.
\newblock Corrigendum: ``{A}lmost sure rates of mixing for i.i.d. unimodal
  maps'' [{A}nn. {S}ci. \'{E}cole {N}orm. {S}up. (4) {\bf 35} (2002), no. 1,
  77--126; {MR}1886006 (2003d:37027)].
\newblock {\em Ann. Sci. {\'E}cole Norm. Sup. (4)}, 36(2):319--322, 2003.



\bibitem{BDL18}
V.~Baladi, M.~F. Demers, and C.~Liverani.
\newblock Exponential decay of correlations for finite horizon {S}inai billiard
  flows.
\newblock {\em Invent. Math.}, 211(1):39--177, 2018.
\bibitem{BBPS19} J.~ Bedrossian, A.~Blumenthal, S.~Punshon-Smith.
\newblock Almost-sure exponential mixing of passive scalars by the stochastic Navier-Stokes equations.
\newblock preprint arXiv:1905.03869, 2019.
\bibitem{Blu17}
A.~Blumenthal.
\newblock Statistical properties for compositions of standard maps with increasing coefficient. 
\newblock \emph{Ergodic Theory and Dynam. Systems.} 41(4), 981--1024, 2021.


\bibitem{B75}
R.~Bowen.
\newblock {\em Equilibrium states and the ergodic theory of {A}nosov
  diffeomorphisms}.
\newblock Lecture Notes in Mathematics, Vol. 470. Springer-Verlag, Berlin,
  1975.

\bibitem{BLV03}
H.~Bruin, S.~Luzzatto, and S.~Van~Strien.
\newblock Decay of correlations in one-dimensional dynamics.
\newblock {\em Ann. Sci. \'{E}cole Norm. Sup. (4)}, 36(4):621--646, 2003.

\bibitem{Bu99} J.~Buzzi. Exponential decay of correlations for random Lasota-Yorke maps. {\em Comm. Math. Phys.} 208, no. 1, 25--54, 1999. 

\bibitem{C04}
A.~Castro.
\newblock Fast mixing for attractors with a mostly contracting central
  direction.
\newblock {\em Ergodic Theory Dynam. Systems}, 24(1):17--44, 2004.

\bibitem{CF94}
H.~Crauel and F.~Flandoli.
\newblock Attractors for random dynamical systems.
\newblock {\em Probab. Theory Related Fields}, 100(3):365--393, 1994.

\bibitem{D00}
D.~Dolgopyat.
\newblock On dynamics of mostly contracting diffeomorphisms.
\newblock {\em Comm. Math. Phys.}, 213(1):181--201, 2000.

\bibitem{DKK04} D.~Dolgopyat, V.~ Kaloshin, L.~ Koralov.
\newblock Sample path properties of the stochastic flows. 
\newblock {\em Ann. Probab.} 32, no. 1A, 1--27, 2004.

\bibitem{DFG19}
D.~Dragi{\v c}evi{\'c}, G.~Froyland, C.~Gonz{\'a}lez-Tokman, and S.~Vaienti.
\newblock A spectral approach for quenched limit theorems for random hyperbolic
  dynamical systems.
\newblock {\em Trans. Amer. Math. Soc.} 373, 629--664, 2020.

\bibitem{DFG18}
D.~Dragi\v{c}evi\'{c}, G.~Froyland, C.~Gonz\'{a}lez-Tokman, and S.~Vaienti.
\newblock Almost sure invariance principle for random piecewise expanding maps.
\newblock {\em Nonlinearity}, 31(5):2252--2280, 2018.

\bibitem{DFG18a}
D.~Dragi\v{c}evi\'{c}, G.~Froyland, C.~Gonz\'{a}lez-Tokman, and S.~Vaienti.
\newblock A spectral approach for quenched limit theorems for random expanding
  dynamical systems.
\newblock {\em Comm. Math. Phys.}, 360(3):1121--1187, 2018.

\bibitem{DH20} 
D.~Dragi\v{c}evi\'{c} and Y.~Hafouta.
\newblock Limit theorems for random expanding or Anosov dynamical systems and vector-valued observables.
\newblock \emph{Ann. Henri Poincar\'e}, 21, 3869--3917, 2020.

\bibitem{D15}
Z.~Du.
\newblock {\em On mixing rates for random perturbations}.
\newblock PhD thesis, National University of Singapore, 2015.

\bibitem{G05}
S.~Gou{\"e}zel.
\newblock Berry-{E}sseen theorem and local limit theorem for non uniformly
  expanding maps.
\newblock {\em Ann. Inst. H. Poincar\'e Probab. Statist.}, 41(6):997--1024,
  2005.

\bibitem{G06}
S.~Gou{\"e}zel.
\newblock Decay of correlations for nonuniformly expanding systems.
\newblock {\em Bull. Soc. Math. France}, 134(1):1--31, 2006.

\bibitem{H20} 
Y.~Hafouta.
\newblock 
Limit theorems for random non-uniformly expanding or hyperbolic maps with exponential tails.
\newblock{\em Ann. Henri Poincar\'e}, 2021.

\bibitem{HK18}
Y.~Hafouta and Y.~Kifer.
\newblock {\em Nonconventional limit theorems and random dynamics}.
\newblock World Scientific Publishing Co. Pte. Ltd., Hackensack, NJ, 2018.

\bibitem{HL19}
O.~Hella and J.~Lepp{\"a}nen.
\newblock Central limit theorems with a rate of convergence for time-dependent
  intermittent maps.
\newblock {\em Stochastics and Dynamics} Vol. 20, No. 04, 2050025, 2020.

\bibitem{HS19}
O.~Hella and M.~Stenlund.
\newblock Quenched normal approximation for random sequences of
  transformations.
\newblock {\em JJ. Stat. Phys.}, 178, 1--37, 2020.

\bibitem{HP70}
M.~W. Hirsch and C.~C. Pugh.
\newblock Stable manifolds and hyperbolic sets.
\newblock In {\em Global {A}nalysis ({P}roc. {S}ympos. {P}ure {M}ath., {V}ol.
  {XIV}, {B}erkeley, {C}alif., 1968)}, pages 133--163. Amer. Math. Soc.,
  Providence, R.I., 1970.
  
\bibitem{H99} H.~Hu. 
\newblock Conditions for the existence of SBR measures for `almost Anosov' diffeomorphisms. 
\newblock {\em Trans. Amer. Math. Soc.} 352(5): 2331--2367, 1999.

\bibitem{KH95}
A.~Katok and B.~Hasselblatt.
\newblock {\em Introduction to the modern theory of dynamical systems},
  volume~54 of {\em Encyclopedia of Mathematics and its Applications}.
\newblock Cambridge University Press, Cambridge, 1995.
\newblock With a supplementary chapter by Katok and Leonardo Mendoza.

\bibitem{L18a}
J.~Lepp\"{a}nen.
\newblock Intermittent quasistatic dynamical systems: weak convergence of
  fluctuations.
\newblock {\em Nonauton. Dyn. Syst.}, 5(1):8--34, 2018.

\bibitem{LV18}
X.~Li and H.~Vilarinho.
\newblock Almost sure mixing rates for non-uniformly expanding maps.
\newblock {\em Stoch. Dyn.}, 18(4):1850027, 34, 2018.

\bibitem{LQ95} P-D. Liu and M. Qian. Smooth ergodic theory of random dynamical systems. Lecture Notes in Mathematics, 1606. Springer-Verlag, Berlin, 1995.

\bibitem{M87}
R.~Ma{\~n}{\'e}.
\newblock {\em Ergodic theory and differentiable dynamics}, volume~8 of {\em
  Ergebnisse der Mathematik und ihrer Grenzgebiete (3) [Results in Mathematics
  and Related Areas (3)]}.
\newblock Springer-Verlag, Berlin, 1987.

\bibitem{MN08}
I.~Melbourne and M.~Nicol.
\newblock Large deviations for nonuniformly hyperbolic systems.
\newblock {\em Trans. Amer. Math. Soc.}, 360(12):6661--6676, 2008.

\bibitem{MN09}
I.~Melbourne and M.~Nicol.
\newblock A vector-valued almost sure invariance principle for hyperbolic
  dynamical systems.
\newblock {\em Ann. Probab.}, 37(2):478--505, 2009.

\bibitem{PSS20}
Y.~Pesin, S.~Senti, and  F.~Shahidi. 
\newblock Area preserving surface diffeomorphisms with polynomial decay of correlations are ubiquitous. 
\newblock{\em Preprint, Arxiv}, 2020.

\bibitem{P72}
V.~A. Pliss.
\newblock On a conjecture of {S}male.
\newblock {\em Differencial'nye Uravnenija}, 8:268--282, 1972.

\bibitem{R18}
M.~Ruziboev.
\newblock Almost sure rates of mixing for random intermittent maps. 
\newblock Differential equations and dynamical systems, 141–152, Springer Proc. Math. Stat., 268, Springer, Cham, 2018.

\bibitem{SSV21}  M.~Stadlbauer, S.~Suzuki, P.~Varandas. Thermodynamic formalism for random non-uniformly expanding maps, {\em Comm. Math. Phys.}, 385, 369--427, 2021.

\bibitem{S11a}
M.~Stenlund.
\newblock Non-stationary compositions of {A}nosov diffeomorphisms.
\newblock {\em Nonlinearity}, 24(10):2991--3018, 2011.

\bibitem{T04}
A.~Tahzibi.
\newblock Stably ergodic diffeomorphisms which are not partially hyperbolic.
\newblock {\em Israel J. Math.}, 142:315--344, 2004.

\bibitem{Y86}
L.-S. Young.
\newblock Stochastic stability of hyperbolic attractors.
\newblock {\em Ergodic Theory Dynam. Systems}, 6(2):311--319, 1986.

\bibitem{Y98}
L.-S. Young.
\newblock Statistical properties of dynamical systems with some hyperbolicity.
\newblock {\em Ann. of Math. (2)}, 147(3):585--650, 1998.

\bibitem{Y99}
L.-S. Young.
\newblock Recurrence times and rates of mixing.
\newblock {\em Israel J. Math.}, 110:153--188, 1999.

\bibitem{ZH19}
 X.~Zhang and H.~Hu.
\newblock Polynomial decay of correlations for almost Anosov diffeomorphisms. 
\newblock{\em Ergodic Theory Dynam. Systems} 39(3): 832--864, 2019.
\end{thebibliography}
\end{document}